\documentclass[11pt,leqno]{amsart}
\makeatletter
\AtBeginDocument{%
  \let\original@@tocwrite=\@tocwrite
  \newif\ifAHVflag
  \def\jlreq@uniqtoken{\jlreq@uniqtoken}
  \def\jlreq@endmark{\jlreq@endmark}
  \long\def\jlreq@getfirsttoken#1#{\jlreq@getfirsttoken@#1\bgroup\jlreq@endmark}
  \long\def\jlreq@getfirsttoken@#1#2\jlreq@endmark#3\jlreq@endmark{#1}
  \renewcommand{\@tocwrite}[2]{%
    \begingroup
      \AHVflagfalse 
      \@ifempty{#2}{}{%
        \expandafter\expandafter\expandafter\ifx\jlreq@getfirsttoken#2\jlreq@uniqtoken{}\jlreq@endmark\Sectionformat\expandafter\@firstoftwo\else\expandafter\@secondoftwo\fi
        {%
          \def\Sectionformat##1##2{\@ifempty{##1}{}{\AHVflagtrue}}%
          #2
        }{\AHVflagtrue}%
      }%
      \def\@tempa{}%
      \ifAHVflag\def\@tempa{\original@@tocwrite{#1}{#2}}\fi
    \expandafter\endgroup
    \@tempa
  }%
}
\makeatother

\usepackage{a4wide}
\usepackage{amssymb}
\usepackage{amsmath}
\usepackage{amsthm}
\usepackage[all]{xy}
\usepackage[usenames,dvipsnames]{xcolor}
\usepackage{lmodern}
\usepackage[T1]{fontenc}

\newcommand{\Hom}{\operatorname{Hom}}

\newcommand{\Ind}{\operatorname{Ind}}

\newcommand{\ind}{\operatorname{ind}}

\newcommand{\Ker}{\operatorname{Ker}}
\newcommand{\End}{\operatorname{End}}

\newcommand{\Sp}{\operatorname{Sp}}
\newcommand{\Res}{\operatorname{Res}}

\newcommand{\Mod}{\operatorname{Mod}}

\newcommand{\Lie}{\operatorname{Lie}}

\DeclareMathOperator{\tr}{tr}

\DeclareMathOperator{\Aut}{Aut}

  \DeclareMathOperator{\Irr}{Irr}

\renewcommand{\)}{\textup{)}}

\theoremstyle{plain} 
\newtheorem{theorem}{Theorem}[section]
\newtheorem{corollary}[theorem]{Corollary}
\newtheorem{lemma}[theorem]{Lemma}
\newtheorem{proposition}[theorem]{Proposition}

\newtheorem{proposition-definition}[theorem]{Proposition-Definition}

\theoremstyle{definition}
\newtheorem{definition}[theorem]{Definition}

\theoremstyle{remark}
\newtheorem{remark}[theorem]{Remark}
\newtheorem{example}[theorem]{Example}
\newtheorem{notation}[theorem]{Notation}

\numberwithin{equation}{section}

\title{Representations of a reductive $p$-adic group in characteristic distinct from $p$}

\author{Guy Henniart, Marie-France Vign\'eras}

\begin{document} 

\maketitle

\begin{abstract} We investigate the irreducible cuspidal $C$-representations of a reductive $p$-adic group $G$ over a field $C$ of characteristic different from $p$. 
 In all known cases, such a representation is the  compactly induced representation $\ind_J^G \lambda$ from  a smooth $C$-representation $\lambda$ of a compact modulo centre subgroup $J$ of $G$. When $C$ is algebraically closed, for many groups $G$, a list of pairs $(J,\lambda)$   has been produced, such that any irreducible cuspidal $C$-representation of $G$ has the form $\ind_J^G \lambda$, for a pair $(J,\lambda)$  unique up to conjugation. We verify that those lists are stable under the action of field automorphisms of $C$,  and we produce similar lists  when $C$ is no longer assumed algebraically closed. Our other main result concerns supercuspidality. This notion makes sense for the irreducible cuspidal $C$-representations of $G$, but also for the  representations $\lambda$ above, which involve representations of finite reductive groups. In most cases  we prove that  $\ind_J^G\lambda$ is supercuspidal if and only if $\lambda$ is supercuspidal. 
  \end{abstract}

\setcounter{tocdepth}{2}  
\tableofcontents

\section{Introduction}\label{s:0}
Let  $F$ be a non-archimedean local field with finite residue field of characteristic $p$,  $\underline G$  a connected reductive linear group defined over $F$ and  $C$ a field of characteristic $c$.  We are interested in irreducible smooth
$C$-representations of  $G= \underline G(F)$.  

  Most of the literature  supposes that the coefficient field $C$ is the field  $\mathbb C$ of complex numbers, but  the study of congruences of automorphic  forms and  the modularity conjectures of Galois representations use representations over number fields or finite fields.  

 We concentrate here on the case  $c\neq p$,  which we always
assume, and for which our basic reference is \cite{V96}. With no further assumption on $C$, we investigate the irreducible 
$C$-representations of $G$; we concentrate on the cuspidal ones since every irreducible smooth $C$-representation of $G$ embeds  in  a representation  parabolically induced  from an irreducible cuspidal $C$-representation of
some Levi subgroup of $G$.
 
 When $C$ is algebraically closed, all known irreducible cuspidal $C$-representations of $G$  are  compactly induced from an open  subgroup, compact modulo the centre,  and one conjectures that it is the case for any $G$. 
  In this paper we extend many known results to a general coefficient field $C$, as we now explain. For us  a cuspidal $C$-type in $G$ is a pair $(J, \lambda)$   where $J\subset G$ is an open compact modulo the centre subgroup  and $\lambda$ is  an isomorphism class of  smooth  $C$-representations  of $J$ such that  the representation $\ind_J^G\lambda$ of  $G$  compactly induced  from  $(J,\lambda)$ is irreducible, hence (as we show) cuspidal.
 A set  $\mathfrak X$ of 
cuspidal $C$-types in $G$ is said to satisfy exhaustion if all irreducible cuspidal $C$-representations  are of this form, unicity if $\ind_J^G\lambda$ determines $(J,\lambda)$ modulo $G$-conjugation, intertwining if the endomorphism $C$-algebras  of $\lambda$ and $ \ind_J^G\lambda$ are isomorphic 
(that condition is automatic when $C$ is algebraically closed), $
H$-stability for a subgroup $H\subset \Aut(C)$ if for
$(J, \lambda) \in\mathfrak X$ and $\sigma \in  H$, $(J, \sigma(\lambda))$ is also in $\mathfrak X$.
When $C $ is algebraically closed, and for many of our groups $G$, a list of cuspidal $C$-types
$(J, \lambda)$ has been produced, which satisfies   exhaustion (sometimes only   for level $0$ representations)  and often unicity. 
For those lists we verify $\Aut(C)$-stability,  which allows us to
produce similar lists when $C$ is no longer assumed algebraically closed:

 \begin{theorem}\label{thm:01} Let $C$ be a field of characteristic  $c \neq p$ and  $C^a$  an algebraic closure of $C$. 
 
1)  Any irreducible cuspidal $C$-representation of $G$ of level $0$ is induced from a type in a list of cuspidal types in $G$ satisfying  intertwining, unicity, and $\Aut(C)$-stability \footnote{Those types are called of level $0$}.
 
2)  Any irreducible cuspidal $C$-representation of $G$   is induced from a type in a list $\mathfrak Y$ of cuspidal types in $G$ satisfying  intertwining and unicity, if  $G$ admits  a set $\mathfrak X^a$ of  cuspidal $C^a$-types  satisfying  unicity, exhaustion, and  $\Aut_C(C^a)$-stability.
 If $\mathfrak X^a$ satisfies  $\Aut(C^a)$-stability, then $\mathfrak Y$ satisfies  $\Aut(C)$-stability.

\end{theorem} 

  To construct $\mathfrak Y$   from  $\mathfrak X^a$, we replace  each $(J,\lambda^{a})\in \mathfrak X^a$ by   $(K,\rho^a)$ where $K$ is the  $G$-normalizer  of $J$ and $\rho^a =ind_J^K \lambda^{a}$. 
We prove that $K$ is  open and compact modulo the centre (Proposition \ref{prop:cusptype2}) and that the   set $\mathfrak Y^a$   of $C^a$-types $(K,\rho^a)$ in $G$ associated to $\mathfrak  X^a$
satisfies  the same properties as $\mathfrak X^a$ (Proposition \ref{prop:II.55}).
The new set  $\mathfrak Y^a$  has the advantage that it satisfies $\Aut_C(C^a)$-unicity: if $\sigma\in \Aut_C(C^a)$ then $\sigma(\ind_K^G(\rho))\simeq \ind_K^G(\rho)$ implies $\sigma(\rho)\simeq \rho$. 
   We obtain $\mathfrak Y$ by replacing   $(K,\rho^a)\in \mathfrak Y^a$ by $(K,\rho)$ where $\rho$ is an  irreducible smooth $C$-representation of $K $ such that $\rho^a$ is $\rho$-isotypic (isomorphic to a direct sum of representations isomorphic to $\rho$) as a $C$-representation of $K$; this relies on 
  the decomposition theorem  of  $C^a\otimes_C \pi$ for a simple module $\pi$ over a $C$-algebra, with an endomorphism ring of finite $C$-dimension  \cite{HV19}. We recall that decomposition theorem in Section \ref{s:1}.

\bigskip Applying our method to the list of  cuspidal $C^a$-types in $G$   constructed  by  Bushnell-Kutzko \cite{BK93}, Moy-Prasad \cite{MP96}, Morris\cite{M99}, Weissmann \cite{W19}, Minguez-S\'echerre \cite{MS14}, Cui \cite{C19}, \cite{C20}, Kurinczuk-Skodlerack-Stevens \cite{KSS},    Skodlerack \cite{Sk20}, Yu-Fintzen \cite{F2}, we obtain:

\begin{theorem}\label{thm:02}  Let $C$ be a field of characteristic  $c \neq p$.

1) Any irreducible cuspidal $C$-representation of $G$ of level $0$ is compactly induced, and $G$ admits  a list of  level $0$  cuspidal $C$-types in $G$ satisfying  intertwining, unicity, exhaustion, and $\Aut(C)$-stability.

2) Any  irreducible cuspidal $C$-representation of $G$ is compactly induced, and $G$ admits  a list of  cuspidal $C$-types  satisfiying    intertwining, unicity, exhaustion, and  $\Aut(C)$-stability, in the following cases:

 the semimple rank of $G$ is $\leq 1$ (except for unicity, which is not known for all $G$  of rank $1$), 
 
  $G=SL(n,F)$,

$G=GL(n,D)$ for a central division algebra $D$ of finite dimension over $F$, 
  
  $G$ is a classical group (a unitary,symplectic or special orthogonal group as in \cite{KSS}) and $p\neq 2$,
 
 $G$  a quaternionic form of a classical group as above.
 
 $G$ is a moderately ramified connected reductive group and $p$ not dividing the order of the absolute Weyl group.
  \end{theorem}

Theorem \ref{thm:01} applies rather generally. Indeed we show that if $C^a$ and $C'^a$ are two
algebraically closed fields with the same characteristic $c\neq p$ and $G$ admits a set  of $C^a$-types 
satisfying  unicity, exhaustion, and $\Aut(C^a)$-stability, then $G$ also admits a list of $C'^a$-types satisfying the same  properties.

\bigskip 
Our other main results concern supercuspidality. An irreducible smooth $C$-representation $\pi$ of $G$ is supercuspidal  if it is not a subquotient of a representation parabolically induced from a proper Levi subgroup of $G$ \cite{V96}. This notion of supercuspidality also makes sense for finite reductive groups. The explicit cuspidal $C^a$-types $(J, \lambda)$ considered above involve cuspidal $C^a$-representations of finite reductive groups. More precisely $J$ has two normal  open subgroups $J^1\subset J^0$ and the quotient $J^0/J^1$ is naturally a finite reductive group. The restriction of $\lambda$ to $J^0$ is constructed as a tensor product of an irreducible $C^a$-representation $\kappa$ of  $J^0$, which we call here a preferred extension (see  \S \ref{ss:3.4}  for detail), and a $C^a$-representation $\rho$ of $J^0$  trivial on $J^1$, inflated from a cuspidal representation of $J^0/J^1$. We say accordingly that $\lambda$ is supercuspidal if the irreducible components of $\rho$ are inflated from supercuspidal representations of the finite reductive group $J^0/J^1$. For a cuspidal $C$-type $(J,\lambda) $ obtained via Theorem \ref{thm:01}, we say that $\lambda$ is supercuspidal if the irreducible components of $C^a \otimes_ C \lambda$  (which are cuspidal $C^a$-types) are supercuspidal.

\begin{theorem}\label{thm:03}  

 Let $(J, \lambda)$ be a  cuspidal $C$-type in $G$. Then  $\lambda$ is supercuspidal if and only if   $\ind_J^G \lambda$ is supercuspidal, in the following cases:

-  $(J, \lambda)$ has level $0$,

 -  $C$ is algebraically closed, $(J,\lambda)$ is   in the list of cuspidal $C$-types of $G=GL(n,F)$ or  $G$ is a classical group and $p$ is odd, or $G$  splits over a tame Galois extension of $F$ and  $p$ is odd and does not divide the order of the absolute Weyl group of $G$,  
 constructed   by Bushnell-Kutzko [12],  Minguez-S\'echerre [38],  or  Kurinczuk-Skodlerack-Stevens [34], or  Yu \cite{Y01}, Fintzen \cite{F2} \footnote{This case is conditional on the verification of the second adjunction by Dat}.  
 \end{theorem}

To prove  Theorem \ref{thm:03} 
we use injective hulls. Indeed if $\pi$ is an irreducible $C$-representation  of a finite reductive group, and $I_\pi $ is an injective hull of $\pi$, then $\pi$ is supercuspidal if and only if $I_\pi$ is cuspidal. In practice we work with representations whose restriction to a maximal  torsion-free subgroup $Z^\sharp$  of the centre $Z$ of $G$ is a multiple of a fixed irreducible $C$-representation $\omega$. In that setting we show that  
if  an injective hull $I_{\lambda, \omega}$ of $\lambda$ is cuspidal  then  $\ind_J^G I_{\lambda, \omega}$ is cuspidal  (with the same length as $I_{\lambda, \omega}$) and is an injective hull  and a projective cover of $ \pi=\ind_J^G \lambda$.   In the reverse direction we show that if $I_{\pi, \omega}$ is cuspidal then $ I_{\lambda, \omega}$ is supercuspidal. 
When the second adjointness holds, then $\pi$ is supercuspidal if and only if $I_{\pi, \omega}$ is cuspidal, and we get  (Theorems \ref{thm:sct0} and  \ref{thm:sct0positive}):
  
 For  a cuspidal $C$-type  $ (J,\lambda)$   in $G$ of level $0$ or satisfying   the properties (i) to (vi) of \S \ref{ss:3.4}, if $\lambda$ is supercuspidal then $\pi=\ind_J^G\lambda $ is.  The converse is true if $ (J,\lambda)$   has level $0$, or if   $ (J,\lambda)$ satisfies also the  property  (vii) of \S \ref{ss:3.4} and $(G,C)$ satisfies the second adjunction.

 This result implies Theorem \ref{thm:03}. Indeed Dat proved the second adjointness for level $0$ representations and  for the groups $G=GL(n,F)$, the classical groups of Stevens, and the moderately ramified groups of  Fintzen-Yu. We check  that  the properties (i) to (vii) of \S \ref{ss:3.4} are satisfied  by the cuspidal $C^a$-types in $G$ when $G$ is
 $GL(n,D)$, a classical group as in \cite{KSS} and $p\neq 2$, a quaternionic form of such a classical group, a moderately ramified connected reductive group and $p$ not dividing the order of the absolute Weyl group,  constructed by Bushnell-Kutzko [12],  Minguez-S\'echerre [38],  Kurinczuk-Skodlerack-Stevens [34], Skodlerak  \cite{Sk17}, \cite{Sk20}, Yu \cite{Y01}.

\bigskip  The layout of the paper is the following.  In section \ref{s:1} we recall  useful consequences  
of the Decomposition Theorem of \cite{HV19}: when $V$ is a simple module over a unital $C$-algebra $A$ with 
finite-dimensional commutant, it describes the submodule structure of  $C^a \otimes_C\pi$
as a module over $C^a \otimes_C A$. Section \ref{s:2}  collects facts about various functors
on the category of smooth $C$-representations of a locally profinite group $G$; also,  we derive from section  \ref{s:1} a procedure
to produce cuspidal $C$-types in $G$ from cuspidal $C^a$-types. In section  \ref{s:3}, we first show  that all irreducible smooth $C$-representations of $G=\underline G(F)$, are admissible,
and that our procedure applies to such $G$;  moreover we show that good sets of cuspidal types can be transferred from one
algebraically closed field to another of the same characteristic.  Then we show   $\Aut (C^a)$-stability for
the level $0$ cuspidal $C^a$-types, yielding corresponding lists of $C$-types;  the more technical case of positive level cuspidal $C^a$-types is treated at the end of section \ref{s:pos}.
In section \ref{s:4}, we investigate the notion of supercuspidality and we prove that a level $0$ cuspidal $C$-type $(J, \lambda)$ is supercuspidal if and only if $\ind_J^G \lambda$ is supercuspidal.
Section \ref{s:pos}  is devoted to the explicit $C^a$-types of Theorem \ref{thm:02}. We show   the properties (i) to (vi) of \S \ref{ss:3.4} and $\Aut(C^a)$-stability for them, and  the property (vii)  for some of them to get Theorem \ref{thm:03}.

\bigskip 
Part of this paper was presented  in Carthage at the conference organized by Ahmed Abbes in June 2019. It is a pleasure to thank Muic, Fintzen, Aubert, S\'echerre, Stevens, Abbes and the participants of the Carthage conference and  of IMJ-PRG for their contribution or  stimulating interest.

\bigskip Notation: Throughout the paper $p$ is a prime number, $C$ is a field with characteristic $c$ different from $p$ (unless mentioned otherwise), and $C^a$ is an algebraic closure of $C$. All $C$-algebras are assumed to be associative, $\Aut(C)$ is the group of field automorphisms of $C$ and  $\Aut_C(C^a)$ is  the group of automorphisms of $C^a$ fixing $C$.
 In section \ref{s:2},
$G$ is a locally profinite group, most often with a compact open subgroup of
pro-order invertible in $C$, and $Z $ is a closed central subgroup of $G$. We write $\Mod_C (G)$ for the category of smooth $C$-representations of  $G$ and $\Irr_C(G) \subset \Mod_C(G)$ for the family of   irreducible  representations.
In sections \ref{s:3} to  \ref{s:pos}, $G= \underline G(F)$, where $F$ is a non-archimedean  local field
with finite residue characteristic $p$,  $ \underline G$ is a connected reductive $F$-group and $Z$ is
the centre of $G$ of maximal compact subgroup $Z^0$ and  $Z^\sharp $   a finitely generated torsion-free subgroup. As usual $O_F$ is the ring of integers of $F$, $P_F$ the maximal ideal of $O_F$, $O_F^*$ the group of units of $O_F$, and $k_F=O_F/P_F$ the residual field.

\section{The Decomposition Theorem}\label{s:1}

  Let $A$ be a unital $C$-algebra and $V$  a simple $A$-module such that $D= End_A(V)$ has finite dimension over $C$. The Decomposition Theorem (\cite{HV19}, Theorem1.1) analyzes the structure of 
$C' \otimes_C V$ as a module over $C' \otimes_C A$ when  $C'$ is any normal extension of 
$C$ containing a maximal subfield of $D$. 
Its lattice of submodules is isomorphic
to the lattice of right ideals in the Artinian ring $C' \otimes_C D$; in particular
$C' \otimes_C V$ has finite length.
We shall mostly use the following consequences, drawn in \cite{HV19}, when $C'=C^a$. 

\begin{theorem}\label{thm:DT} A) Let  $V$ be a simple  $A$-module with   $\dim_C(End_A(V))$ finite. Then, the $C^a \otimes_C  A$-module $C^a\otimes_C  V$  has finite length.
A simple   subquotient of $C^a\otimes_C  V$ is also isomorphic
to a submodule, and to a quotient; it is absolutely simple, and defined
over a finite extension of $C$.
The isomorphism classes of simple subquotients form a finite orbit
under $\Aut_C(C^a)$.

B) An absolutely simple $C^a \otimes_C  A$-module $W$ which is defined
over a finite extension of $C$  is a subquotient of $C^a \otimes_C V$, for some simple
$A$-module $V$ with $\dim_C(End_A(V))$ finite. The $A$-module $V$ is determined up to isomorphism
by the property that $W$ is $V$-isotypic\footnote{ $W$ is a direct sum of modiules isomorphic to $V$} as an $A$-module.
\end{theorem}

 The theorem implies that   the map sending $V$ to the set of irreducible subquotients of $C^a\otimes_C V$ induces a bijection  from the set of isomorphism classes of simple $A$-modules  with endomorphism ring of finite $C$-dimension to the set of orbits under $\Aut_C(C^a)$ of absolutely simple $C^a\otimes_C A$-modules defined over a finite extension of $C$.

\bigskip For  any extension  $C'/C$ we put $A_{C'}= C'\otimes_C A$. An  $A_{C'} $-module isomorphic to  $ C'\otimes_{C}V$ for
an $A$-module $V$,  is said to be defined over  $ C$;
 if $V$ is simple,  then
 $C'\otimes_C \End_{A} V\simeq \End_{A_{C'}}C'\otimes_{C}V$ (\cite{HV19}, Rem.II.2). An   $A$-module $V$ is said to be absolutely simple  if  the $A_{C'}$-module $C'\otimes_C V$ is simple for any extension $C'/C$.

  As a consequence of Theorem \ref{thm:DT}, a simple $A$-module $V$ with $\End_A(V)=C$ is absolutely simple. The converse is true (\cite{HV19}, Rem.II.3); in particular a simple $A_{C^a}$-module of finite $C^a$-dimension is absolutely simple.

 \begin{lemma}\label{le:DT}  (i) Let  $V$ be a simple  $A$-module with $D=\End_A(V)$ of finite $C$-dimension and let $C'/C$ be  an extension   contained in $D$.  Then   $C'\otimes_C V$ is a simple  $A_{C'}$-module if and only if  $C'=C$.
 
  (ii) An $A$-module $V$ such that  $D= End_A(V)$ has  finite $C$-dimension is 
  absolutely simple  if and only if $C'\otimes_CV$ is a simple $C'\otimes_CA$-module for all finite extensions $C'/C$.  
   \end{lemma}
      \begin{proof}  (ii) is a consequence of (i) by Theorem \ref{thm:DT}. We prove (i).
          If $C'\otimes_CV$ is a simple $C'\otimes_CA$-module, then  $C'\otimes_C D\simeq \End_{C'\otimes_C A}( C'\otimes_CV)$ is a division $C'$-algebra of finite  dimension containing  $C'\otimes_C C'  $.  An  integral $C'$-algebra of finite dimension  is a field, so $C'\otimes_C C'$ is a field. But $C'\otimes_C C'$ is a field if and only if $C'=C$ (the multiplication $x\otimes y\mapsto xy$ is a quotient map $C'\otimes_C C'\to C'$ hence is an isomorphism, and $C'=C$).
\end{proof}

\begin{remark}\label{r:DT2} Let    $V$ be  a simple $A$-module  with endomorphism ring  $D=End_{A}V$ and $B\subset A$  a central subalgebra (containing the unit).
Let $E$ be the image of $B$ in $\End_C(V)$. Then $E$ lies in the centre of $D$.
In the special case where $V$ is a finitely generated $B$-module, then
$E$ is a field and $V$ is a finite dimensional $E$-vector space 
(\cite{Bki-A8},  3.3, Corollary 2 of Proposition 3); that applies in particular when $B=A$,
in which case $\dim_E(V)=1$. In general, at least $E$ is  integral.

Assume now that $D$ has finite $C$-dimension. Then $E$ is a 
commutative finite dimensional $C$-algebra, and being an integral domain it is
necessarily a field, hence is a finite extension of $C$. The algebra $B$ acts on $V $ via its quotient
field $E$, which is a simple $B$-module, and $V$ as a $B$-module, is $E$-isotypic.
In the case that $A=C[G]$ for a group $G$, and $B= C[ Z]$ where $Z\subset G$ is  
a central subgroup,  we see that 
$V$ is a simple $E[ G]$-module, $Z$ acting   by an  homomorphism $Z\to  E^*$.
\end{remark} 

  We now
  give a kind of converse to  Theorem  \ref{thm:DT}, which will be used in Proposition \ref{prop:compactind}.

\begin{proposition}\label{prop:DT} 
Let $V$ be an $A$-module such that $ \End_A(V)$ is a division algebra. 
Assume that the $C^a \otimes_C A$-module
$C^a\otimes_CV$ has finite length  and that all its simple  subquotients
are absolutely simple,   and their isomorphism classes form an   orbit
under $\Aut_C(C^a)$. 
Then $V$ is simple and $ \End_A(V)$ has finite dimension over $C$.
\end{proposition}

\begin{proof}  Let $U$ be a simple $A$-subquotient of $V$. As  an $A$-module, $C^a \otimes_C U$ is a direct sum of modules isomorphic to $U$, hence each simple   subquotient of $C^a\otimes_CU$   is,  as  an $A$-module,  a direct sum of modules isomorphic to   $U$.
Since the isomorphism classes of the  simple  subquotients of $C^a\otimes_CV$ form an orbit
under $\Aut_C(C^a)$, they are also the   simple subquotients of $C^a\otimes_CU$.
As an $A$-module, $C^a \otimes_C V$ is a direct sum of modules isomorphic to $V$. Therefore the simple $A$-subquotients of $V$ are isomorphic to $U$. Since the $C^a\otimes_C A$-module
$C^a\otimes_CV$ has finite length, the $A$-module $V$
has finite length too.  As $U$ occurs as an $A$-submodule and a $A$-quotient of $V$, there exists   an  $A$-endomorphism of $V$  of image $U$; as $\End_A(V)$ is a division algebra, any non-zero $A$-endomorphism of $V$ is surjective. Therefore $V=U$ is simple. Since $C^a\otimes_CV$ has finite length,  and that all its simple  subquotients
are absolutely simple,   
$\End _{C^a\otimes A}( C^a\otimes_C  V)$ has finite dimension over $C^a$.
This is also the dimension of $ \End_A(V)$ over $C$. 
 \end{proof}

 \section{Smooth $C$-representations of locally profinite groups}\label{s:2}
 
 Let $G$ be a locally profinite group, $Z$ a closed central subgroup of $G$, and $C$ a field. 
 
\begin{definition}\label{def:almostfinite} We say that   $Z$ is almost finitely generated when $Z/Z^0$ is finitely generated
for some open compact subgroup $Z^0\subset Z$. This property does not depend on the choice of
$Z^0$.
\end{definition}
A $C$-representation $V$ of $G$ is called smooth if every vector in $V$ has open stabilizer in $G$,
and $V$ is called admissible if moreover the subspace $V^J$ of $J$-invariant vectors of $V$ has finite dimension
for any open subgroup $J$ of $ G$. Note that a $C$-representation $V$ of $G$   generated by $V^J$ for some  $J$  is smooth (as $gv$ is fixed by $gJg^{-1}$ for $g\in G, v\in V$).  We write $\Mod_C (G)$ for the category of smooth $C$-representations of  $G$ and $\Irr_C(G)$ for the family of   irreducible  smooth $C$-representations of $G$.

A homomorphism  $\chi: G \to  C^*$ is called a  $C$-character of $G$. 
The $C$-characters $\chi$  of $G$ act on the  $C$-representations of $G$, respecting irreducibility : if $(\pi, V)$ is a $C$-representation of $G$, then
$g\mapsto \chi(g)\pi(g)$ for $g\in G$, gives a $C$-representation of $G$  on $V$, written $\chi \pi$ and called the twist of $\pi$ by $\chi$. That action is compatible with morphisms of
 representations, so we also get an action, written in the same way, on isomorphism classes of $C[G]$-modules. 
 The smooth characters, i.e. with open kernel,  act on the smooth representations of $G$ and on their isomorphism classes.

 \subsection{Invariants under an open subgroup} \label{ss:2.1}

 Let  $J \subset G$ be an open subgroup. The    functor 
$V \to V^J $ from $C$-representations of $G$ to $C$-vector spaces is left exact, and exact if $J$ is compact
and has pro-order invertible in $C$.
If $V$ is irreducible, we  get a ring homomorphism $D=End_{C [G]}(V) \to
\End_C(V^J)$ which is injective if $V^J \neq 0$, because $D$ is a division algebra;
in particular  if $\dim_C(V^J)$ is finite, so is $\dim_C(D)$, and we 
can apply section \ref{s:1} to $V$. In that case $Z$ acts via a quotient field of $C[Z]$, finite over 
$C$ (Remark \ref{r:DT2}), so $Z$ acts via a character if $C$ is algebraically closed.
We conclude that we can apply section \ref{s:1} to an irreducible admissible $C$-representation of $G$.

In fact the functor $V \to V^J $ gives a functor   from $C$-representations of  $G$ to modules over the Hecke $C$-algebra
$H_C(G,J)$ of $J $ in $G$  (in order to get left modules, $H_C(G,J)$ is defined as  the opposite of the $C$-algebra $\End_{C[G]} (C[G/J])$).
 
The following is well-known when $C=\mathbb C$ is the field of complex numbers, and the proofs
in (\cite{BH06}1.4.3 Proof of Proposition (2)) carry over to any field $C$.

\begin{theorem} \label{thm:invariant} Let $J \subset G$ be an open compact subgroup  with pro-order invertible in $C$.

(i) If $V$ is an irreducible $C$-representation of $G$ with $V^J\neq 0$, then $V$ is smooth and $V^J$ is a
simple $H_C(G,J)$-module.

(ii) Let $M$ be a simple $H_C(G,J)$-module. Then  the $C [  G]$-module   $X_M= C [G/J ]\otimes_{H_C(G,J)}M$ is smooth, 
  has a unique largest submodule $X'_M$ not intersecting
$1 \otimes M$, and the quotient $Y_M=X_M/X'_M$ is an irreducible  smooth $C$-representation  of $G$.
The map sending $m \in M $ to the image in $Y_M$ of  $1 \otimes  m$ gives an isomorphism $M\to Y_M^J$
of $H_C(G,J)$-modules.

(iii) If $V $is an irreducible $C$-representation of $G$ such that $V^J \neq 0$, then taking $M=V^J$,
the natural map  $X_M\to  V$ induces an isomorphism  $Y_M \to V.$
\end{theorem}

That theorem gives an explicit  bijection between isomorphism classes of irreducible smooth  $C$-representations
of $G$ with non-zero $J$-invariants, and isomorphism classes of simple $H_C(G,J)$-modules.

\begin{corollary}\label{cor:invariant1} Let $V\in \Irr_C(G)$ with $V^J \neq 0$. Then the 
natural map   $\End_{C [ G]} (V) \to \End_{H_C(G,J)}(V^J)$ is an isomorphism.
\end{corollary}

That result was already established in (\cite{Mu19}, Theorem 4.1) when $C=\mathbb Q$ is the field of rational numbers.

\begin{proof} We already remarked that the map is injective. Let $a \in \End_{H_C(G,J)}(V^J)$.
Then $a$ induces an endomorphism of the  $C [G]$-module $X_M$ where $M=V^J$, which preserves $X'_M$
hence induces  $b\in \End_{C[G]}(V)$ by (iii) of the theorem; by construction $b$  induces
$a$ on $M=V^J$.
\end{proof}

\begin{remark}\label{rem:invariant1}Let $V\in \Irr_C(G) $  with  $V^J\neq 0$ and $\dim_C(V^J)$  finite. By the corollary we can
apply section \ref{s:1} to the  $C [G]$-module $ V$ and also to
the $H_C(G,J)$-module $V^J$; we get parallel results, in particular
the map $W\to W^J$ gives an isomorphism
of the lattice of subrepresentations of $C^a \otimes_C V$ onto the lattice
of $H_C(G,J)$-submodules of $(C^a\otimes_C V)^J$.
\end{remark}

 Let us consider an extension $C'/C$. If  $V$ is a $C [G]$-module, the inclusion  
$C' \otimes_C V^J \to  C'  \otimes_C V$ induces an isomorphism $C' \otimes_C V^ J\to (C' \otimes_C V)^J$;
it is an isomorphism of $H_{C'}(G,J)$-modules. Clearly if $V'$ is an irreducible $C' [G]$-module defined over $C$
with $ V'^J \neq 0$, then the $H_{C'}(G,J)$-module $V'^J$ is also defined
over $C$.   Conversely:

\begin{corollary}\label{cor:invariant2} Let $C'/C$ be an extension. Let $V'\in \Irr_{C'}(G)$   with $ V'^J \neq 0$. If    the $H_{C'}(G,J)$-module $ V'^J $ is  defined over $C$, then
 $V' $ is defined over $C$.
\end{corollary}
\begin{proof}Let $M $ be an $H_C(G,J)$-module such that $C' \otimes_C M \simeq  V'^J$. Then
$M$ is necessarily simple, because $V'^J$ is (by (i) of the theorem). Consider the irreducible
$C$-representation $Y_M$ of $G$  of $J$-invariants  isomorphic to $M$; then  $(C' \otimes_C Y_M )^J \simeq V'^J$ and  by (iii) of the theorem, $C'  \otimes_C Y_M \simeq 
 V' $  hence $V' $ is  defined over $C$.
\end{proof}

 \subsection{ Irreducible $C$-representations of $G$ with finite dimension
} \label{ss:2.2}

In this subsection,  we assume that  $G/Z$ is compact. 

\begin{proposition}\label{prop:finitedim}
Let $V$ be a finitely generated smooth $C$-representation
 of $G$. Then $V$ is trivial on an open subgroup. If  $V$ is irreducible and $Z$ is almost finitely generated (Definition   \ref{def:almostfinite}), 
then $\dim_C(V)$ is finite.
\end{proposition}

 The second assertion will be generalized  (Proposition \ref{prop:Zcompact}). 
 
\begin{proof} Let $ S$ be a finite set   generating $V$. For $v\in V$, the  $G$-stabilizer $J_v $ of   $v$  is an open subgroup of $G$; 
for $g \in G$,  $J_{gv}=gJ_vg^{-1}$ and depends only on $gZJ_v$.  So, because
$G/Z$ is compact, there are only finitely many open subgroups $J_{gv}$ for $g\in G, v\in S$.  Their intersection
is therefore an open subgroup of $G$ acting trivially on $V$. Moreover $V$ is a finitely generated
module over $C [Z]$ (as $G/Z$ is compact and  $V$ is a finitely generated $C[G]$-module with an open subgroup of $G$ acting trivially).  If $Z$ is almost finitely generated,
then any quotient field of $C[  Z]$ has finite $C$-dimension, 
and the second assertion is a
consequence of Remark \ref{r:DT2}.\end{proof}  

\begin{corollary}\label{cor:finitedim} When  $Z$ is almost finitely generated,  any $V\in \Irr_{C^a}(G)$
  has finite dimension, is absolutely irreducible and  is defined
over a finite extension of $C$.
\end{corollary}
 \begin{proof} $\End_{C^a[G]}(V) =C^a$ and   $Z$ acts on $V$ via a
character. The values of that character generate a finite
 extension $E$  of $C$ in $C^a$ (since $Z$ is almost finitely generated). On the other hand an open subgroup of $G$ acts trivially on $V$ so we may assume that $G/Z$ is finite;  taking representatives $g_i$ for $G/Z$, the matrix coefficients
of the action of the $g_i'$s on a basis of $V$ generate a finite extension $C'$ of $C$ in $C^a$, and we see that
$V$ is defined over $EC'$.
\end{proof}

 \subsection{$Z$-compactness
} \label{ss:2.3}
In this subsection, we assume that $G$ contains an open compact subgroup with pro-order invertible in $C$.

For each such subgroup $J\subset G$, we then have a canonical projector $e_J$, which acts on any smooth $C$-representation
$V$ of $G$,  it is $J$-equivariant and has image $e_JV=V^J$. 

A smooth $C$-representation $V$ of $G$ is  called $Z$-compact (\cite{V96} I.7.3 and 7.11) if for all small enough
open compact subgroups $J \subset G$, and all $v \in V$, the support of the  function $g \to e_Jgv $ is  $Z$-compact (i.e. compact modulo $Z$). 
 When $Z$ is trivial, we say compact instead of $Z$-compact.
It is clear that a subrepresentation of a  $Z$-compact smooth $C$-representation of $G$
is $Z$-compact, and a quotient representation is also.

It is known that a compact finitely generated smooth $C$-representation of $G$ is admissible (\cite{V96} I.7.4).
Let us analyze the situation in general. Let $V\in \Mod_C(G)$ and $J  $ an open compact
subgroup of $G$ with pro-order invertible in $C$. Let $v\in V$ and $V(v) $  the subrepresentation of $V$ 
 generated by $v$.  Then the vector space $V(v)^J$ is generated by the $e_Jgv, g \in G$. If $V $ is $Z$-compact, the function $g \to 
  e_Jgv$ vanishes outside a finite number of double cosets  $JgZJ_v$, where $J_v \subset G$ is the $G$-stabilizer 
of $v$. In particular $V(v)^J$ is a finitely generated $C[Z]$-module. More generally if $V$ is $Z$-compact and $W$ is a finitely
generated subrepresentation of $V$, then $W^J$ is finitely generated over $C [Z]$.
If $C [Z]$ acts on $W$ via a quotient $A$ with $\dim_C(A)$ finite, then $W^J$ is finite dimensional.

\begin{proposition} \label{prop:Zcompact} Assume that $Z$ is almost finitely generated. Then any $Z$-compact   $V\in \Irr_C(G)$  is admissible.
\end{proposition}

When $G/Z$ is compact, all smooth $C$-representations of $G$ are $Z$-compact, so the
proposition does generalize the last assertion of  Proposition \ref{prop:finitedim}.

\begin{proof} Choose a non-zero vector $v \in V$ and $J \subset G$ a compact
open subgroup  with pro-order invertible in $C$ fixing $v.$ By  the above, the simple
$H_C(G,J)$-module $V^J$ is finitely generated over $C  [Z]$. Since $Z$ is almost
 finitely generated, reasoning as for Proposition \ref{prop:finitedim} gives that $V^J$ has finite dimension.  
\end{proof}   

\begin{remark} \label{r:Zcompact1}
1) Assume that $Z$ is almost finitely generated. If there exists   a  
$Z$-compact irreducible smooth representation $V$ of $G $ of finite dimension over $C$, we claim that $G/Z $ is compact.  Indeed, since $V$ is smooth of finite dimension, an open  normal subgroup in $G$ acts trivially on $V$, and dividing by this subgroup we may assume $G $ discrete.
Taking now $J$ trivial in the property of  $Z$-compactness, we see that that   $gv\neq 0$   only for $g$ in  a finite number of $Z$-cosets; but that implies
that $G/Z$ is finite.

 2) Assume that $V\in \Irr_C(G)$ is  $Z$-compact. 
Then the image $A$  of  $C  [Z]$   in the division algebra $D=\End_ {C  [G]}(V)$  is
an integral domain, so has a fraction field $E \subset D$.
Since $D$ stabilizes $V^J$,  $A$ and  $E $  stabilizes $V^J$ too. By the above,
$V^J$ is finitely generated over $A$, so $\dim_{E}(V^J)$ is   finite.
If $\dim_C A$ (or equivalenly $\dim_C E$) is  finite, then $\dim_C(V^J)$ is   finite.  

 3) Let $C'$ be an extension of $C$  and let $V \in \Mod_C(G)$. Then
 $C' \otimes_CV \in \Mod_{C'}(G)$, and   $V$ is $Z$-compact
if and only if  $C' \otimes_CV $ is $Z$-compact.
\end{remark}

 The   space of linear forms $L:V\to C$   invariant
 under an open subgroup of $G$ with the natural action ($gL (gv)=L(v)$ for $v\in V,g\in G$) of $G$ is a smooth representation  $V^\vee \in \Mod_C(G)$ called the {\it contragredient} of $V$  (\cite{V96} I.7.1). 
 As
 $G$ contains a compact open subgroup of pro-order invertible in $C$,  the contragredient functor $V\mapsto V^\vee:\Mod_C(G)\to \Mod_C(G)$ is  exact  (\cite{V96} 4.18 Proposition (i)); the three properties: $V$ admissible, $V^\vee$ admissible, the natural map $V\to(V^\vee)^\vee$  is bijective, are   equivalent, and when $V$ is admissible then $V$ is irreducible if and only if $V^\vee$ is (\cite{V96} I.4.18 Proposition (iii) and (v)). 
A smooth {\it coefficient} of $V\in \Mod_C(G)$ is a function $g\to L(gv)$ from $G$ to $C$ for $v\in V, L\in V^\vee$.

\begin{proposition} \label{prop:Zcompact2}
Let $V\in \Mod_C(G)$. Then $V$ is $Z$-compact if and only if  the support of  any smooth
coefficient  of $V$ is  $Z$-compact.
\end{proposition}

\begin{proof}  That   is already established  with compact instead of  $Z$-compact  in \cite{V96} I.7.3 Proposition c). It is clear that if $V$ is $Z$-compact then any
  smooth coefficient of $V$ is $Z$-compact. Let us  prove the converse.
Fix $ v \in V$. To prove that the support of  the function $g\to e_Jgv$ is $Z$-compact  for
all compact open subgroups $J\subset G$ with pro-order invertible in $C$, we may as well assume that
$J $ fixes $v$ and that $v$ generates $V$. For each double coset $x=JhZJ, h\in G $, let $V(x)\subset V^J$  the subspace  generated by the $e_Jhzv$
for $ z\in Z$. Then $ V^J$ is the sum of the $V(x)$, because $v$ generates $V$.
Let $\mathcal X$ be the set of cosets $x $ such that $V(x)\neq  0$. The goal is to show
that $\mathcal X$ is finite. By the hypothesis on coefficients,  
any linear form on $V^J $ (which can be uniquely extended to a linear form on $V$ fixed by $J$)
vanishes outside a finite number of  subspaces $V(x)$. For each $x\in \mathcal X$, let us choose
a non-zero vector $ v_x \in V(x)$. Extract from the family $v_x, x\in \mathcal X$, a maximal linearly
independent subfamily $v_x, x \in \mathcal Y$, with  $\mathcal Y \subset \mathcal X$. There is a linear form on $V^J $ taking value $1$
at each $v_x$ for $ x \in \mathcal Y$, which implies that $\mathcal Y$ is finite by our hypothesis   on coefficients, so the family     $v_x, x \in \mathcal X$, generates
a finite dimensional subspace $W \subset V^J$. Choose a basis of $W^*=\Hom_C (W,C)$; each element of that basis vanishes
on  $v_x$ for all $ x\in \mathcal X$, except for finitely many, so the $v_x$ are $0$ except finitely many which finally shows
that $\mathcal X$ is finite, as desired.
\end{proof}

\begin{remark}\label{re:contra} Let $V\in \Mod_C(G)$ admissible.
From  Proposition \ref{prop:Zcompact2},  $V$ is $Z$-compact if and only if $V^\vee$ is. If $V$ is also irreducible, then $V$ is $Z$-compact if and only if  the support of  {\it some} smooth
coefficient  of $V$ is  $Z$-compact.   \end{remark}

 \subsection{Compact induction} \label{ss:2.4} 

In the setting of the introduction, all known constructions of cuspidal  irreducible
$C$-representations of  $\underline G(F)$ are for $C$ algebraically closed and are obtained via compact induction.
We now investigate the situation without assuming $C$ algebraically closed.

  Let $J \subset G$ be a subgroup. The  functor $C[G]\otimes_{C[J]} - $ from $C[ J ]$-modules to $C  [G]$-modules
  is exact (because $C[G]$ is a free $C[ J ]$-module)
and faithful; it is left adjoint to the restriction functor $\Res_J^G$. It is  obviously compatible with scalar extension through a field extension $C'/C$;
in particular, it is compatible with the action of $\Aut(C)$ and with the action of  $C$-characters of $G$ on $C[G]$-modules  (in the sense that   if $\chi$ is a  $C$-character of $G$ and  $\rho$ a  $C[J]$-module, then $\chi  (C[G]\otimes_{C[J]} \rho )\simeq  C[G]\otimes_{C[J]}\chi|_J \rho $ ).

  We now assume that $J $ is open. In that case the previous functor restricts to a functor $\Mod_C(J) \to \Mod_C(G)$;  we rather use the isomorphic functor
of compact induction (\cite{V96}, I.5.7) denoted  $\ind_J^G:\Mod_C(J) \to \Mod_C(G)$ while the smooth induction from $J$ to $G$ is denoted by   $\Ind_J^G:\Mod_C(J) \to \Mod_C(G)$  (\cite{V96} I.5.1).
 
 If $V\in \Mod_C(J)$, we thus get a ring homomorphism  \begin{equation}\label{eq:compactind}\End_{C[ J ]}(V)\to
\End_{C[ G ]}(\ind_J^G V)
\end{equation} which is injective by faithfulness. It is rarely surjective, though, even 
when $V$ is irreducible. By adjunction   $\End_{C[ G ]}(\ind_J^GV) \simeq  
\Hom_{C[ J ]} ( V, \Res_J^G\ind_J^GV)$, and  $\Res_J^G\ind_J^G V$
decomposes as a direct sum 
over double cosets $JgJ $ of the representation of $ J $ on the space $\ind_J^{JgJ}V$ of functions in $\ind_J^GV$ with support in $JgJ$, 
\begin{equation}\label{eq:compactind2}\Res_J^G\ind_J^G V\   = \ \oplus_{JgJ  }\,  \ind_J^{JgJ}V.
\end{equation}
 The trivial coset $J$ yields a representation of $J$ naturally
isomorphic to $V,$ and accounts for the embedding \eqref{eq:compactind}. The embedding is an isomorphism if and only if  no non-trivial
coset contributes. Note  that $\ind_J^GV$ can be admissible only if $V$ is and finitely many cosets contribute.
 
 Let us analyze a more general situation. Let $J'$ be another open subgroup of $G$ and $V' \in \Mod_C(J')$. By adjunction
 $$\Hom_{C[ G ]}( \ind_{J'}^G(V'), \ind_J^G(V)) \simeq  \Hom_{C[ J' ]} ( V', \Res_{J'}^G\ind_J^GV), \quad \Res_{J'}^G\ind_J^G V= \oplus_{JgJ'  }\,  \ind_J^{JgJ'}V.$$Consequently $\Hom_{C[ J' ]} ( V', \Res_{J'}^G\ind_J^GV)$ sits between the direct sum and the direct product of  the  
  $\Hom_{C[ J' ]} ( V',\ind_J^{JgJ'}V)$. More precisely, it  is made out of the collections of $\phi_{JgJ'}\in \Hom_{C[ J' ]} ( V',\ind_J^{JgJ'}V)$ such that for $v'\in V'$, $\phi_{JgJ'}(v')=0$ except for a finite number of double cosets $JgJ'$ in $G$ \footnote{The reader should be aware of slightly incorrect statements in (\cite{V96} I.8.3 Preuve (i) (ii), and \cite{KS} Remark 2.1)}.
Note that we have an isomorphism  
\begin{equation} \label{eq:compactind3} \Hom_{C[ J' ]} ( V',\ind_J^{JgJ'}V)\to  \Hom_{C[ J'  \cap g^{-1} Jg]} ( V',{}^gV)\end{equation} which  associates to $\phi$  the map $v'\mapsto \phi(v')(1)$, where  ${}^gV$ is the  representation of $g^{-1} Jg$ on $V$ via $(g^{-1}hg, v)\mapsto hv$.

\bigskip  Let us recall what intertwining  means. 
Let $G$ be a group and $H,K$ subgroups of $G$. Let $\rho$ be a $C$-representation of $H$ on a space $V$, and  $\tau$ a 
$C$-representation of $K$ on a space $W$. For $g\in G$, 
a map $\Phi\in \Hom_C(V,W)$  such that $\tau (k)\circ \Phi= \Phi \circ  \rho(g^{-1}kg)$  for $k\in K\cap gHg^{-1}$, is called a $g$-intertwiner of $\rho$ with $\tau$. The space $I(g,\rho,\tau)$ of  $g$-intertwiners of $\rho$ with $\tau$ is 
\begin{equation} \label{eq:compactind4} \Hom_{C[gHg^{-1}\cap K]}(V^g,W) = \Hom_{C[H\cap g^{-1}Kg]}(V,{}^gW), 
\end{equation} 
  where $V^g={}^{g^{-1} }V$ is the $g$-conjugate of $V$: the representation of $gHg^{-1}$ on $V$ via  $(ghg^{-1}, v)\mapsto hv$.  
    We say that $g$ intertwines $\rho$ with $\tau$  if  $I(g,\rho,\tau)\neq 0$; this is equivalent to saying that the set $KgH$ supports a  non-zero function $f: G\to \Hom_C(V,W)$  such that $f(kgh)=\tau(k) f(g) \rho(h)$ for $k\in K, h\in H$. Indeed, the map $f\mapsto \Phi=f(g)$ is an isomorphism from the space of such functions to the space of $g$-interwiners.
When  $g$ interwines $\rho$ with $\rho$, we simply say that $g$ intertwines $\rho$. The set of $g\in G$ which interwines $\rho$ is called  the $G$-intertwining of $\rho$. The $G$-normalizer of $\rho$ is the $K$-intertwining of $\rho$ where $K$ is the $G$-normalizer of $H$.

 An immediate but important remark is that the action of $\Aut(C)$
preserves intertwining.  Indeed, let $\sigma\in \Aut(C)$. Then $\sigma(V)=C\otimes_\sigma V$ identifies with $V$ by $1\otimes v$ corresponding to $v$, the action of $c\in C$ on $1\otimes v$ corresponding to the action of $\sigma(c)^{-1}$ on $v$ (as $c\otimes v=1 \otimes \sigma^{-1}(c)v$).
Clearly, a $g$-intertwiner   of $\rho$ with $\tau$ identifies with a $g$-intertwiner  of $\sigma(\rho)$ with $\sigma(\tau)$,
\begin{equation} \label{eq:compactind5} I(g, \rho, \tau) \simeq I(g, \sigma(\rho), \sigma(\tau)).
\end{equation}

  Recall that $\ind_J^G V $ is contained in the smooth induced representation  $ \Ind_J^G V$ and that  $( \ind_J^G V )^\vee$ is naturally
isomorphic to $\Ind_J^G( V^\vee)$.

\begin{remark}\label{re:compactind} Assuming that $G$ has a compact open subgroup
of pro-order invertible in $C$, let us briefly tackle the issue of admissibility. 

a) If $\ind_J^G V $ is admissible, so is $V$,
because $J \subset G$ is open and $V \subset \Res_J^G \ind_J^G V$.

b) Assume $ V $  admissible (so $V^\vee$ is admissible and $(V^\vee)^\vee\simeq  V$). Adapting the reasoning in \cite{B90} (see also
\cite{V96}, I.2.8), one proves that the following conditions are equivalent:
(i) $\Ind_J^G V$ is admissible
(ii) $\ind_J^G V$ is admissible
(iii) $\ind_J^G V=\Ind_J^G V$.
If those conditions are satisfied for $V$, they are also satisfied for $V^\vee$. In particular
$\ind_J^G (V^\vee)= \Ind_J^G (V^\vee)$ and all smooth coefficients of $\ind_J^G V$ have support
contained in a finite number of cosets $Jg$; consequently their support is
$Z$-compact if $J $ contains $Z$ with $J/Z$ compact, and then  $\ind_J^G V$ is  $Z$-compact.
\end{remark}

Our main interest goes to cases where $\ind_J^G V$ is irreducible, which can only happen when $V$ is.
Let us assume then that $V$ is irreducible.  Let $N_G(J,V)$ be the $G$-normalizer  of $(J, V)$, made out of the elements in the $G$-normalizer of $J$ which
transform $V$ into an isomorphic representation of $J$. In general the intermediate induction $\ind_J^ {N_G(J,V)}V$ is not irreducible.
So to ensure that  $\ind_J^GV$ be irreducible, it is better to assume that $J$ contains the centre of $G$.  For $g\in N_G(J,V) $ the coset $gJ=JgJ$ contributes
to $\End_{C  [G]}( \ind_J^GV)$; if the embedding   $\End_{C [J]} V \to \End_{C[G]}(\ind_J^GV)$ \eqref{eq:compactind} is an isomorphism, 
 then $ N_G(J,V)=J$. 

For $V$ irreducible and $V^a$   an irreducible subquotient of $C^a \otimes_CV$, we  derive information on $\ind_J^GV$ from  $\ind_J^GV^a$  using section \ref{s:1} when $\End_{C [J]} V$ has
finite $C$-dimension.

\begin{proposition}\label{prop:compactind} Let $V\in \Irr_C(J)$ such that $\End_{C [J]} V$ has
finite dimension. Let $V^a\in  \Irr_{C^a}(J)$ be a subquotient of $C^a \otimes_CV$.
The following two conditions are equivalent: 

(i) The embedding   $\End_{C [J]} V \to \End_{C[G]}(\ind_J^GV)$ \eqref{eq:compactind} is an isomorphism.

(ii) The embedding  $\Hom_{C^a[J]}(V^a, \sigma (V^a))\to
 \Hom_{C^a[G]}( \ind_J^G V^a, \ind_J^G \sigma( V^a))$ is an isomorphism, for any $\sigma \in \Aut_C(C^a)$.

Assume that $\ind_J^G V^a $ is absolutely irreducible and that  $\ind_J^G \sigma (V^a ) \simeq \ind_J^G V^a$  only if $\sigma (V^a)  \simeq  V^a$ for $\sigma \in \Aut_C(C^a)$. Then (i) holds true and $\ind_J^G V$ is irreducible.
\end{proposition}

 \begin{proof} Condition (i) means that $\Hom_{C[J]}(V, \ind_J^{JgJ} V) =0$  for any non-trivial coset $JgJ$.
Similarly, condition (ii) means that $\Hom_{C^a[J]}(V^a, \ind_J^{JgJ} \sigma (V^a) )=0$  for any non-trivial coset $JgJ$, and any $\sigma \in \Aut_C(C^a)$.

Let us fix $g \in G$. Because $V$ is irreducible, $\Hom_{C^a[J]}(C^a\otimes_C V, \ind_J^{JgJ} (C^a\otimes_C  V))\simeq 
  C^a\otimes_C  \Hom_{C[J]}(V, \ind_J^{JgJ} V)$.  From  Theorem \ref{thm:DT} the irreducible subquotients
of $C^a \otimes_C V$ have the form $\sigma(V^a)$, for $\sigma \in \Aut_C(C^a)$, and each of them is (isomorphic to)
a subrepresentation  of $C^a \otimes_C V$, and also a quotient.  We deduce  the  equivalence of the four properties:

(1)  $\Hom_{C[J]}(V, \ind_J^{JgJ} V) \neq 0$, 
 
(2)  $\Hom_{C^a[J]}(C^a\otimes_C V, \ind_J^{JgJ} (C^a\otimes_C  V))\neq 0$,  
 
 (3)   there exist $\sigma,  \sigma' \in 
\Aut_C(C^a)$ such that $\Hom_{C^a[J]}(\sigma (V^a), \ind_J^{JgJ} \sigma' (V^a))\neq  0$,

  (4) there exists $\tau\in \Aut_C(C^a)$ such that $\Hom_{C^a[J]}(V^a, \ind_J^{JgJ} \tau (V^a))\neq  0$.

\noindent Therefore, conditions  (i) and (ii) are equivalent.   
  
  Assume now that $\ind_J^G V^a$ is absolutely irreducible and that 
$\ind_J^G \sigma (V^a) \simeq  \ind_J^G V^a$ only if $\sigma (V^a ) \simeq V^a$  for $\sigma \in \Aut_C(C^a)$.
By   the decomposition
theorem \ref{thm:DT}  there is a finite normal extension $ C'$ of $C$ such that  $C' \otimes_C V $ achieves the length  
of $C^a \otimes_C V$ and such $V^a$ is defined over $C'$, so $\ind_J^G V^a $ is also defined over $C'.$
By assumption $\ind_J^G V^a$ is absolutely irreducible so all the irreducible subquotients of
$\ind_J^G (C^a \otimes_C V)$ which are its $Aut_C(C^a)$-conjugates, are absolutely irreducible as well.
It follows also that the length of $\ind_J^G (C^a \otimes_C V)$ is the same as that of $C^a \otimes_C V$ which
by the decomposition
theorem \ref{thm:DT} is finite.
The other part of the assumption implies that condition (ii) is satisfied, hence also condition (i).
Consequently we can apply  Proposition \ref{prop:DT} to $\ind_J^G V$, and we get that it is simple.
\end{proof}

\begin{remark}\label{re:compactind3} Assume that all conditions in the proposition are satisfied. Then applying the decomposition
theorem \ref{thm:DT}  to $V$ or  $\ind_J^G V$ gives parallel results. In particular compact induction from $J$ to $G$
gives an isomorphism from the lattice of subrepresentations of $C^a \otimes_C V$ to the lattice of
subrepresentations of $C^a \otimes_C \ind_J^G V$.
\end{remark}

 \subsection{$C$-types
} \label{ss:2.5}

\begin{definition}\label{def:type}  A  $C$-type in $G$ is a  pair $(J,V)$ where  $J\subset G$ is an open subgroup  and  $V $ an
isomorphism class of irreducible smooth   $C$-representations of $J$ such that $ \ind_J^GV$ is irreducible;  it  is called $Z$-compact if $Z\subset J$ and $J/Z$ is compact; it is said to have finite dimension if $\dim_CV$ is finite, to be defined over a subfield $C'$ of $C$ if $V$ is defined over $C'$.
 \end{definition}

Warning: That is not the usual definition of types  when $G=\underline G(F)$ \cite{BK98}.   It is simply a convenient one for us; we can take $J=G$, in particular.

\bigskip The group $\Aut(C)$ acts on the set of $C$-types in $G$ by its action on the component $V$ of the pair. 
Also, $G$ acts on that set by conjugation. The two actions respect $Z$-compact types.
 
When  $(J,V)$ is a $C$-type in $G$ and  $J' $ is a subgroup of $G$ containing $J$, the transitivity of the compact induction $\ind_J^G V \simeq  \ind_{J'}^G(\ind_J^{J' }V)$ shows that $\ind_J^{J'} V$ is irreducible,  so $(J', \ind_J^{J' }V)$ is a $C$-type in $G$, which is $Z$-compact if $(J,V)$ is  and $J$ has finite index in $J'$.

\begin{definition}\label{def:type2} A $C$-type $(J,V)$  in $G$ is said to satisfy intertwining if the   homomorphism $\End_{C[J]} V\to \End_{C[G]}( \ind_J^G V)$ \eqref{eq:compactind}
 is  an isomorphism.
\end{definition}

\begin{definition}  \label{def:type3} 
 Let  $\mathfrak X$ be a set of $C$-types in $G$.  The set $\mathfrak X$  satisfies intertwining if each element of $\mathfrak X $ does, it satisfies
unicity if for   $ (J, V), (J',\lambda')\in \mathfrak X$   such that  $ \ind_{J}^GV \simeq  \ind_{J'}^GV'$, there is $g\in G$ conjugating $(J,\lambda)$ to $(J',\lambda')$.

Let  $\mathcal Z$ be a set of isomorphism classes of irreducible $Z$-compact smooth $C$-representations of $G$.
The set $\mathfrak X$ satisfies $\mathcal Z$-exhaustion if for $(J,V) \in  \mathfrak X$, the isomorphism class of
$\ind_J^G V$ is in $\mathcal Z$  and any element of $\mathcal Z$  has that form.

Let $\sigma \in  \Aut(C)$.  The set $\mathfrak X$ is   $\sigma$-stable if for $(J,V) \in \mathfrak X$, then
 $ (J, \sigma (V))$ is also in $\mathfrak X$; $\mathfrak X$ is said to satisfy   $\sigma$-unicity when  moreover  $\ind_J^G V  \simeq \ind_J^G \sigma (V)$ implies $V\simeq  \sigma V$.

Let $H\subset  \Aut(C)$ be a subgroup. The set $\mathfrak X$ is    $H$-stable if it  is $\sigma$-stable  for any $\sigma \in H$; it satisfies $H$-unicity if  it satisfies $\sigma$-unicity for any $\sigma \in H$. \end{definition}

\begin{proposition}\label{prop:II.55}  Let   $\mathfrak X$  be  a set of  $C$-types in $G$.  
 Let $\mathfrak X'$ denote the set of $C$-types $(J', V')$, where
$J' $ is the $G$-normalizer of $J$ and $V'$  the isomorphism class of $\ind_J^{J' }V$. Let $\mathcal Z$ and $H$ be as in Definition  \ref{def:type3}. 

If $\mathfrak X$ satisfies intertwining (resp. unicity, resp. $\mathcal Z$-exhaustion, $H$-stability),
then so does $\mathfrak X'$.   If $\mathfrak X$ satisfies unicity and $H$-stability, then $\mathfrak X'$ satisfies $H$-unicity.
 \end{proposition}
 \begin{proof} The  composite
of the natural maps  $ \End_{C[ J]} V \to End_{C  [ J' ] }\ind_J^{J' }V  \to \End_{C[   G] }\ind_J^G V$ is  \eqref{eq:compactind}. The assertion for intertwining follows.
The assertion for unicity comes from the fact that if $g \in G$ conjugates $(J,V)$ to $(J_1,V_1) $ then it also conjugates
$(J', \ind_J^{J'} V) $ to $((J_1)', \ind_{J_1}^{(J_1)' }V)$. The assertion  for  $\mathcal Z$-exhaustion comes from transitivity of induction.
The assertion for $H$-stability is due to the fact that compact induction is
compatible with the action of $\Aut(C)$. Let us assume that $\mathcal  X$ satisfies unicity and $H$-stability, and let
$(J,V) \in \mathfrak X$ and $\sigma \in H$ be such that  $\ind_J^G \sigma (V )\simeq \ind_J^G V$. By unicity there is $g \in 
G$ conjugating $(J,V)$  to $(J, \sigma (V))$. But then $g$ is in the $G$-normalizer $J' $ of $J$    so $\ind_J^{J'} V  \simeq  \ind_J^{J' }\sigma( V)$.
 \end{proof}

  \subsection{From  $C^a$-types to  $C$-types}\label{ss:2.6} 
  
   When $G$ is a reductive group with centre $Z$ as in the introduction, many lists of   $C^a$-types $(J,V)$ are known. Our purpose is to produce $C$-types from $C^a$-types.
We now describe a general procedure using Proposition \ref{prop:compactind}.

We start from a set $\mathfrak Y^a $ of  $C^a$-types $(J,V^a)$ in $G$ such that:

a)   Each  $V^a$ is absolutely irreducible
and defined over a finite extension of $C$ (by Corollary  \ref{cor:finitedim} that is automatic if  the types   in $\mathfrak Y^a $ are $Z$-compact).

b)  $\mathfrak Y^a$ satisfies intertwining, unicity, $\Aut_C(C^a)$-stability and $\Aut_C(C^a)$-unicity.

We let  $\mathcal Z^a$ be the set of isomorphism classes of the $C$-representations
$\ind_J^G V^a$  for $(J,V^a) \in \mathfrak Y^a$. Those representations are absolutely irreducible (because $ V^a $ is and $\mathfrak Y^a$
satisfies intertwining), defined over
a finite extension of $C$ (because $ V^a $ is), and  $\mathcal Z^a$ is stable under $\Aut_C(C^a)$.

Let $(J, V^a)\in \mathfrak Y^a$; there is a unique isomorphism class $V$ of smooth irreducible $C$-representations of $J$
such that $V^a$ is a subquotient of $C^a\otimes_C V$ (Theorem \ref{thm:DT}), and $\End_{C[ J ]}V$ has
finite dimension over $C$.

\begin{lemma} \label{le:CatoC} The pair $(J,V)$ is a $C$-type in $G$  satisfying intertwining.\end{lemma}

\begin{proof} We see that  $\ind_J^G \sigma (V^a) \simeq \ind_J^G V^a $ only if $\sigma (V^a) \simeq  V^a$ for $\sigma \in \Aut_C(C^a)$, since $\mathfrak Y^a$ satisfies unicity,
$\Aut_C(C^a)$-stability and $\Aut_C(C^a)$-unicity.
Proposition \ref{prop:compactind}  gives the result.
\end{proof}

From the set $\mathfrak Y^a$  of  $C^a$-types $(J,V^a)$ in $G$  we therefore get a set $\mathfrak  Y$ of $C$-types $(J,V)$ in $G$, 
satisfying intertwining. Note that  $\mathcal Z^a $ is the set
of isomorphism classes of irreducible subquotients of $C^a\otimes_C W$ for $W $  in the set $\mathcal Z$ of isomorphim classes of
 $\ind_J^G V\in \Irr_C(G)$  for $(J,V)\in \mathfrak  Y$.   
\begin{proposition} \label{prop:CatoC}The set $\mathfrak Y$ of $C$-types in $G$ satisfies intertwining and unicity. If $\mathfrak Y^a$ satisfies $\Aut (C^a)$-stability
and $\Aut (C^a)$-unicity, then $\mathfrak Y $ satisfies $\Aut(C)$-stability and $\Aut(C)$-unicity.
\end{proposition}
\begin{proof} That the set $\mathfrak Y$ satisfies intertwining comes from the lemma. 

Let $(J,V)$ and $(J',V')$ in $\mathfrak Y $ be such that $\ind_J^G V$ and $\ind_{J'}^G V'$ are isomorphic. 
Let $V^a$ be the class of an irreducible subquotient of $C^a\otimes_C V$, and choose
similarly $V'^a$. They belong to $\mathfrak  Y^a $, and since $\ind_J^G V\simeq \ind_{J'}^G V' $, we have
$\ind_J^G V^a\simeq \ind_{J'}^G \sigma (V'^a)$  for some $\sigma\in \Aut_C(C^a)$.
Since $\mathfrak  Y^a$ satisfies unicity, $V^a $ and $\sigma (V'^a)$ are conjugate by some $g \in G$,
so that $V$ and $V'$ are also conjugate by $g$.
That proves that $\mathfrak Y$ satisfies unicity.

Assume now that $\mathfrak Y^a$ satisfies $\Aut(C^a)$-stability and $\Aut (C^a)$-unicity. Let $(J,V) \in  \mathfrak  Y$ and $\tau \in \Aut(C)$. There is an extension of $\tau$ to an automorphism $\tau^a$ of $C^a$
(\cite{Bki-A5}  \S4, No 3,Corollary 2 of Theorem2).
If $\iota:C\to C^a$ is the embedding, then $\iota \circ  \tau= \tau^a \circ  \iota$ and 
$$C^a \otimes_{C  } \tau (V) = C^a \otimes_{ \iota \circ \tau}V \simeq \tau^a(C^a \otimes_CV). $$
If $V^a$ is the class of an irreducible
subquotient of $C^a \otimes_C V$, then $(J, V^a)  \in  \mathfrak Y^a$, and because $\mathfrak Y^a $ satisfies $\Aut(C^a)$-stability,
$(J,\tau ^a (V^a ))$ is also in $\mathfrak Y^a$, and it follows that  $(J,\tau (V\)) \in \mathfrak Y$, proving that $\mathfrak Y$ satisfies $\Aut(C)$-stability.

Let now  $\tau \in \Aut(C)$ be such that $\ind_J^G V\simeq \ind_J^G \tau (V)$. Choosing $V^a$  and $\tau^a $ as before, we have 
 see that $\ind_J^G V^a$ and $\ind_J^G \tau^a(V^a)$
are conjugate under $\Aut_C(C^a)$, and changing $\tau^a$ we may assume that
they are isomorphic. But $\mathfrak Y^a$ satisfies $\Aut(C^a)$-unicity, so $V^a=\tau^a(V^a)$,
and it follows that $V= \tau (V)$. That proves that $\mathfrak Y$ satisfies $\Aut(C)$-unicity.
\end{proof}

\section{Cuspidal types in reductive groups} \label{s:3} 

 In this section,  $F $ is a non-Archimedean  local field with finite residual characteristic $p$, $C$ is a field of characteristic $c\neq p$ of algebraic closure $C^a$,
  $\underline G$ is  a connected reductive $F$-group of centre $\underline Z$ with  rank $n_Z$ (the dimension of a  maximal $F$-split subtorus),  $G=\underline G(F)$
 and $Z^\flat$ is a closed subgroup of $Z=\underline Z(F)$ with compact quotient $Z/Z^\flat$ (for example $Z^\flat=Z$).
    
 The group $G$
contains an  open pro-$p$ subgroup;  the compact subgroups of $G$ generate
an open and normal  subgroup $G^0$, and the  quotient $G/G^0$ a finitely generated free abelian group. The maximal compact subgroup $Z^0$
 of $Z$ is   $Z \cap G^0$;  the injective homomorphism $Z/Z^0\to G/G^0$ has finite cokernel.
 Similarly, the  maximal compact subgroup  $Z^{\flat 0}$ of  $Z^\flat$ is  $Z^\flat \cap G^0$,  and the injective homomorphism $Z^\flat/Z^{\flat 0}\to G/G^0$ has finite cokernel (as   $Z/Z^\flat$ is compact).   
 The  groups $Z$ and  $Z^\flat$  are  almost finitely generated (Definition   \ref{def:almostfinite}) and  are the product  of their maximal compact subgroup by a  (non unique) subgroup  isomorphic to 
 $\mathbb Z^{n_Z}$. We can choose $Z^{\flat}$ isomorphic to $\mathbb Z^{n_Z}$ (we can choose $Z^{\flat}$ trivial if and only if $n_Z=0$).  
 
  For a parabolic subgroup $\underline P$ of $\underline G$ with Levi decomposition   $\underline P=\underline M \underline N$ of rational points $P=MN$ (we will say that $P$ is a parabolic subgroup of $G$), the  (unnormalized) parabolic induction $$\Ind_P^G:\Mod_C(M)\to \Mod_C(G)$$  is faithful and exact, with a left adjoint  the $N$-coinvariant functor $(-)_N$  and  a right adjoint $R_P^G$  \cite{V13}. This implies that $(-)_N$ is right exact and  $R_P^G$ is left exact.
  
  As $c\neq p$, $(-)_N$ is exact (\cite{V96} I.4.10) and   the   second adjunction conjecture  says that $R_P^G$  is equivalent to  
 $ \delta_P(-)_N$   where  $\delta_P$  is the modulus of $P$.  The equivalence is a celebrated result of Bernstein when $C$ is the complex field, and a theorem of Dat (\cite{D09} Proposition 6.3, Theorem 9.2) when $G$ is  $ GL(n,F)$  or is 
a  classical $p$-adic group and  $p\neq 2$;  it is also true for any $G$, for the restriction of the functors to the 
subcategories of level $0$ representations (see below \S \ref{ss:3.3}  for that notion).

When  $G$  is the group of points of a connected reductive group over a finite field of characteristic $p$ (we will say simply that $G$ is a finite reductive group in characteristic $p$), 
the parabolic induction $\Ind_P^G$ is faithful and exact, of left adjoint  $(-)_N$ and right adjoint the $N$-invariant functor $(-)^N$ (\cite{OV} Proposition3.1). As $c\neq p$, 
 there exists an idempotent $e\in C[N]$ such that $eV=V^N$ for $V\in \Mod_C(G)$   (\cite{V96} 4.9 Proposition a) and $R_P^G$ is equivalent to
 $ (-)_N$ (this is the form of the second adjunction).
 
   \subsection{Review on cuspidal representations}\label{ss:c}   The definitions and results in this subsection are valid also for finite reductive groups.
 
 \begin{definition} (\cite{V96}  2.2, 2.3) \label{def:cusp} A representation  $\tau\in \Mod_C(G)$  is called cuspidal if  $\tau_N=0$   for all proper parabolic subgroups $P=MN $ of $G$.
  \end{definition}
  
  Equivalently by adjunction, $\tau$  is  cuspidal    if  $\Hom_{C[G]}(\tau, \Ind_P^G \rho) =0$ for all proper parabolic subgroups $P=MN$ of $ G$  and $\rho\in \Mod_C(M)$. 
  
  \begin{remark}  \label{re:cirr} When $\tau$ is irreducible one can restrict to $\rho$ irreducible:  $\tau\in \Irr_C(G)$ is cuspidal if and only if $\Hom_{C[G]}(\tau, \Ind_P^G \rho) =0$ for all proper parabolic subgroups $P=MN$ of $ G$  and $\rho\in \Irr_C(M)$ (\cite{V96}  II.2.4).
  \end{remark}
  
 \begin{remark}  \label{re:Nv} We note that $\tau_N=0$ if and only if  for any element $v\in \tau$ there exists an open compact subgroup $N_v$ of $N$ such that $e_{N_v} v=0$ where $e_{N_v}:\tau\to \tau^{N_v}$ is the projection on the $N_v$-invariants (\cite{V96} I.4.6). This remark will be used  in Proposition \ref{prop:cnirr}.
\end{remark} 
   
\begin{remark} \label{re:rightcusp} What we call cuspidal here is called  left cuspidal  in \cite{AHV19} Definition   6.3, (2.3),   because there is a symmetric notion, which we call   right cuspidal, where one asks $R_P^G\tau=0$,  for all proper parabolic subgroups $P=MN $ of $G$, or equivalently by adjunction $\Hom_{C[G]}(\Ind_P^G \rho, \tau) =0$,   for all proper parabolic subgroups $P=MN $  and $\rho\in \Mod_C(M)$.  Of course when  the second adjunction holds true, in particular for level $0$ representations, these notions are equivalent.   Without assuming  the second adjunction,  we will  verify in Proposition \ref{pro:cusp} (iii)  their equivalence for admissible representations.
\end{remark}

 We are lead naturally to the notion of supercuspidality for  semi-simple representations. 

\begin{definition} \label{def:sc} (\cite{V96}, II.2.5)  An irreducible smooth $C$-representation
  of $G$   is called supercuspidal   if it is not isomorphic to a subquotient of   $\Ind_P^G \rho$ for all proper parabolic subgroups $P=MN \subset G$ 
and $\rho\in \Mod_C(M)$. 
 
 A semi-simple  smooth $C$-representation $\pi$ of $ G$  is supercuspidal if each  irreducible component of $\pi$ is supercuspidal.
\end{definition} 

Clearly supercuspidal implies (left) cuspidal and right cuspidal; the converse is not true (when $C$ is the prime field of characteristic $\ell$ dividing $p+1$, the principal series of $GL(2,\mathbb Q_p)$  induced by the trivial $C$-character of the diagonal torus has a cuspidal non-supercuspidal subquotient). 

\begin{lemma}\label{ex:sc} An irreducible quotient $\pi$ of a projective cuspidal representation $V$ is supercuspidal. 
\end{lemma}
This will be used in the proof of Proposition \ref{prop:cnirr}.  We will see that  all  known irreducible supercuspidal representations are  of this form (Theorems \ref{thm:sc01},  \ref{thm:sc0}, \ref{thm:sc01positive}).
\begin{proof} If $\pi$ is a subquotient of a parabolically induced representation $\Ind_P^G \rho$ for a proper parabolic subgroup $P=MN$ of $G$ and a smooth $C$-representation $\rho$ of $M$, then $\Ind_P^G \rho$ contains a subrepresentation $W$ of quotient  $\pi$. As $V$ is projective of quotient $
\pi$, there is a non-zero $C[G]$-map $V\to W$.  Any quotient of a cuspidal representation is cuspidal and $\Ind_P^G \rho$ does not contains a cuspidal representation. Hence, a contradiction.
\end{proof}
\begin{remark}\label{re:scKn}  One does not need to consider all $\rho\in \Mod_C(M)$  in Definition \ref{def:sc};  it suffices to  take  $\rho =C[K_n  \backslash M]$  for all $K_n$ in  some decreasing
sequence of pro-$p$-open subgroups   of $M$  with trivial intersection. Indeed, $\Mod_C(M)$ is an abelian  Grothendieck  category  with generator $\oplus_n C[K_n\backslash M]$   (\cite{V13} Lemma 3.2), and 
  the functor $\Ind_P^G$ is  exact  and  commutes with direct sums (it has a right adjoint and \cite{KS06} remark after Proposition2.2.10).
\end{remark}
  
\begin{remark}\label{re:scirr} In the definition of supercuspidality of  $\pi\in \Irr_C(G)$,   one can suppose $\rho$ irreducible when 
 $G$ satisfies the second adjunction or  when $\pi$  has level $0$ (a theorem of Dat \cite{D18} Theorem 1.1), or when $G$ is finite  (as $C[M]$ has finite length). Such an  assumption is  forgotten but   used in the proof in  (\cite{V96} II.2.6).  
\end{remark}

\begin{remark}   Assume only  for this remark that $c=p$. When $G/Z$ is not compact, the trivial representation    is right cuspidal and not left cuspidal \cite{AHV19}; an irreducible representation may be  not admissible. A definition of supercuspidality  (\cite{AHHV} for $C$  algebraically closed,  \cite{AHV19} in general) is given  for  admissible irreducible 
$C$-representations $\pi$ of $G$: $\pi$   is called supercuspidal if it is not isomorphic to a subquotient of   $\Ind_P^G \rho$ for all proper parabolic subgroups $P=MN \subset G$ 
and all  admissible irreducible $C$-representations $\rho$ of $M$.  
 \end{remark}

\begin{proposition} \label{pro:cusp} 
(i)  $(\Ind_P^G\rho)^\vee\simeq  \Ind_P^G(\delta_P \rho^\vee)$ for all  parabolic subgroups $P=MN$  of $G$ and   $\rho\in \Mod_C(M)$.

(ii) Let  $\tau \in \Mod_C(G)$. Then $\tau$ is $Z$-compact $\Leftrightarrow \tau$  is cuspidal $\Leftrightarrow \tau^\vee $ is right cuspidal. 

(iii) If $\tau $ is admissible, then 

  $ (\tau _N)^\vee \simeq (\tau^\vee)_{N'} $ for   $P'=MN'$     opposite to $P$ with respect to $M$.   

$\tau$ is cuspidal $\Leftrightarrow \tau$ is  right cuspidal.

(iv) Let $\pi \in \Irr_C(G) $  and  $\chi:G\to C^*$ a smooth $C$-character of $G$. Then 

$\pi \subset \Ind_P^G \rho$ for  some  $P=MN$ and $ \rho\in \Irr_C(M)$ cuspidal.  

$\pi$
is admissible. 

$ \pi$  is cuspidal   $\Leftrightarrow \pi^\vee$  is cuspidal
 $\Leftrightarrow  \chi \pi$  is cuspidal.  
 
$\pi$  is supercuspidal   $\Leftrightarrow \pi^\vee$  is supercuspidal $\Leftrightarrow \chi \pi$  is supercuspidal.

(v)  Any irreducible smooth $C^a$-representation $\pi$ of $G$ is absolutely irreducible.
  
\end{proposition}

\begin{proof} For   (i), (iii) except for  $\tau$  cuspidal $\Leftrightarrow \tau$   right cuspidal,  and for (iv) except for  the last line, we refer to  (\cite{V96} I.4.18 (iii), I.5.11, II.2.1 (vi), II.2.4, II.2.7,  II.2.8 where the proof does not use    the assumption that $C$ is algebraically closed, II.2.9).

We have  $\tau$ cuspidal    $\Leftrightarrow\tau^\vee$ right cuspidal (hence (ii)) because  ((i) and   (\cite{V96} I.4.13)) 
$$\Hom_{C[G]}( \Ind_P^G \rho ,\tau ^\vee) \simeq \Hom_{C[G]}(\tau, ( \Ind_P^G \rho)^\vee )\simeq 
\Hom_{C[G]}(\tau, \Ind_P^G( \delta_P\rho^\vee))$$
and $\Hom_{C[G]}(\tau, \Ind_P^G \rho)\neq  0$  implies  $\Hom_{C[G]}(( \Ind_P^G \rho)^\vee,\tau ^\vee) \neq 0$.
    
If  $\tau$ is admissible, we deduce  $\tau $ cuspidal $\Leftrightarrow \tau$  right cuspidal (ending the proof of (iii)). Indeeed,   $\tau \simeq (\tau^\vee)^\vee$ and $\tau$ is cuspidal if and only if $\tau^\vee$ is cuspidal as  $ (\tau _N)^\vee \simeq (\tau^\vee)_{N'} $.  

We  prove the last line of (iv). As $\chi \Ind_P^G\rho \simeq  \Ind_P^G(\rho \chi|_M)$, it is clear that $\pi$ (super)cuspidal $\Leftrightarrow \chi \pi$ (super)cuspidal. If $\pi$ is non supercuspidal,  it is a quotient of a subrepresentation $ Y $ of   $\Ind_P^G \rho$  for some proper parabolic subgroup   $P=MN $ of $G$ and $\rho\in \Mod_C(M)$; taking the contragredient which is exact,   $Y^\vee$ is a quotient of $(\Ind_P^G \rho)^\vee$ containing $\pi^\vee$, and as $(\Ind_P^G \rho)^\vee\simeq \Ind_P^G( \delta_P  \rho^\vee)$ we see that $\pi^\vee$ is non supercuspidal.
 As $(\pi^\vee )^\vee \simeq \pi$,  the reverse is true.  
 
(v)  The  intertwining algebra of $\pi$  is a finite extension of $C^a$, hence is equal to  $C^a$,  because 
 $\pi$ is admissible  by (iv).  Therefore $\pi$  is absolutely irreducible  (remarks following Theorem \ref{thm:DT}).
\end{proof}

\begin{proposition}\label{prop:c} Let  $C'/C$ be an extension. Let  $\pi\in \Irr_C(G)$ and $\pi' \in \Irr_{C'}(G)$ an irreducible  subquotient of $C'\otimes_C \pi$. Then,
 $\pi$  is supercuspidal if and only if $\pi'$ is; similarly for cuspidal.  \end{proposition}

\begin{proof}   a) Let $P=MN$ be a parabolic subgroup of $G$ and $\rho\in \Mod_C(M)$.  If $\pi $ is a  subquotient of $\Ind_P^G \rho$
 then 
$\pi'$ is a  subquotient of $\Ind_P^G (C'\otimes _C\rho)$ as $\Ind_P^G$ commutes with scalar extension. 
Therefore $\pi'$  supercuspidal implies $\pi$ supercuspidal.

b) Conversely, if $\pi'$ is a  subquotient of $\Ind_P^G(C'[K_n\backslash M)]$ (Remark \ref{re:scKn}) then $\pi$ is a subquotient of $\Ind_P^G(C[K_n\backslash M])$ because, as $C$-representations,  
$\pi'$ is $\pi$-isotypic and $\Ind_P^G(C'[K_n\backslash M])$ is isomorphic to a  direct sum of $\Ind_P^G(C[K_n\backslash M])$. Therefore  $\pi$ supercuspidal implies $\pi'$ supercuspidal. 

c) Similarly for  quotient and right cuspidal, replacing subquotient and supercuspidal in a) and b). Hence  $\pi$  is right cuspidal if and only if $\pi'$ is.  As    right cuspidal = cuspidal for an irreducible representation (Proposition \ref{pro:cusp}), $\pi$  is  cuspidal if and only if $\pi'$ is. \end{proof}

   The next proposition is used  in the  proof of  Proposition \ref{prop:d}.
 
\begin{proposition}\label{induction} 
The scalar extension and the parabolic induction respect  finite length.
\end{proposition} 
\begin{proof} Let  $C'/C$ be an extension and $\pi\in \Irr_C(G)$.  Let $P=MN$ a parabolic subgroup of $G$ and $\sigma\in \Irr_C(M)$.  We show that the  lengths of $C'\otimes_C \pi$ and of $\ind_P^G\sigma$ are finite.
The scalar extension and the parabolic induction being exact, this implies the proposition.

a)    The dimension of $ \End_{C[G]}\pi $ is finite as $\pi$ is admissible (Proposition \ref{pro:cusp} (iv)). The length of $C'\otimes_C \pi$ is finite bounded by the length of $C^a\otimes_C\pi$ by \cite{HV19} Corollary I.2.
  
b) It is already known  that the length of $\ind_P^G\sigma$ is  finite when $C$ is algebraically closed (\cite{V96} 5.13).   The parabolic induction commutes with scalar extension, and by a) $C^a\otimes_C \sigma$ has finite length. 
Hence the length of
$\ind_P^G(C^a\otimes_C \sigma)\simeq C^a\otimes_C \ind_P^G(\sigma)$ is  finite. A fortiori  the length of $\ind_P^G(\sigma)$ is  finite. 
   \end{proof}

\subsection{Cuspidal type} \label{ss:3.1}

Compared to the   case  of a general locally profinite group,   $Z$-compact $C$-types  in $G$ (Definition \ref{def:type}) and irreducible smooth $C$-representations of $G$ have peculiar features:

\begin{proposition}  \label{prop:cusptype} (i) A $Z$-compact $C$-type in $G$ has finite dimension.

(ii) A $Z$-compact $C^a$-type in $G$ is defined over a finite extension of $C$.

(iii) If  $(J,V)$ is a $Z$-compact $C$-type in $G$, the  representation
$\ind_J^G V\in \Irr_C(G)$  is cuspidal, and $\ind_J^G V= \Ind_J^G V$.

(iv) A $Z$-compact $C^a$-type $(J,V) $ in $G$ is absolutely irreducible, and so
is $\ind_J^G V$. In particular $(J,V)$ satisfies intertwining.
\end{proposition}
\begin{proof}  (i)  Proposition \ref{prop:finitedim}.

(ii) Corollary \ref{cor:finitedim}.

 (iii) Any  coefficient of $V $ extended by $0$ is a coefficient of $\ind_J^GV $ with $Z$-compact support. By Remark \ref{re:contra},  $\ind_J^GV\in \Irr_C(G)$ is $Z$-compact. By  Proposition \ref{pro:cusp} (ii),  $\ind_J^GV$  is cuspidal. By Remark \ref {re:compactind} b), $\ind_J^GV\simeq \Ind_J^GV$.

  (iv)  When   $(J,V)$ is a $Z$-compact $C^a$-type   in $G$, then $V$ is irreducible  of finite $C^a$-dimension  (Proposition \ref{prop:finitedim}), so
$V$ is absolutely  irreducible; the irreducible representation $\ind_J^G V$ is admissible  (Proposition \ref{pro:cusp} (iv)) so
 is absolutely irreducible.  The commutant of an absolutely irreducible $C^a$-representation is $C^a$.
\end{proof}

\begin{definition} \label{def:ctype} A $Z$-compact $C$-type in $G$ (Definition  \ref{def:type})    is  called a cuspidal $C$-type in $G$. 
\end{definition}
The definition is motivated by Proposition \ref{prop:cusptype} (iv). 
A cuspidal $C$-type in $G$ satisfies intertwining when $C$ is  algebraically closed (Proposition \ref{prop:cusptype} (iii)), but not in general.  

\begin{example} Take $C=\mathbb R$ and a quaternion division algebra $D$ over $\mathbb Q_3$. We construct an example of $\mathbb R$-type $(J,\lambda)$ in $D^*$ which does not satisfy intertwining.

Let $\mathbb H$ be the Hamilton quaternion algebra and $U$ the quaternion subgroup of order $8$, generated by the order $4$ elements $i,j, ij=-ji$. The left multiplication by $U$ on $\mathbb H$ gives an irreducible 
$\mathbb R$-representation $\rho$ of $U$. If $T$ is the subgroup of $U$ generated by $i$, and $\tau$ the $2$-dimensional representation of $T$ on $\mathbb R [i]$ then $\rho =\ind_T^U \tau$. The commutant of $\tau$ is  $\mathbb R [i]$ and the commutant of $\rho$ is $\mathbb H$. 

We show that $U$ is a quotient of $D^*$; then we take the  inverse image  $J$ of  $T  $ in $D^* $ and the inflation $\lambda$ of $\tau$ to $J$. The type $(J,\lambda)$ in $D^*$ does not satisfy intertwinining.
 The quotient $U_D/U_D^1$ of the unit group by its first congruence subgroup is cyclic of order $8$; choose a generator $\zeta$. Let   $\omega$  be the image of a uniformizer of $D$ in $D^*/U_D^1$. The  group morphism $D^*/U_D^1\to U$  sending $\zeta$ to $i$ and $\omega$ to $ j$ is surjective.
\end{example}

\begin{proposition}\label{prop:cusptype2}  Let $J\subset G$ be an open subgroup containing $Z$ with $J/Z$ compact. Then the $G$-normalizer $J'$ of $J $ 
is open in $G $ and $ J'/Z$ is compact.
\end{proposition}
 
As a consequence, if $(J,V)$ is a cuspidal $C$-type in $G$ 
then the pair $(J', \ind_J^{J' }V)$ is a cuspidal $C$-type in $G$.
Applying the procedure of \S \ref{ss:2.5} and  \S \ref{ss:2.6}, we immediately get:

\begin{theorem}\label{thm:cusptype} Let $\mathfrak X^a$ be a set of cuspidal $C^a$-types $(J,V^a)$ in $G$, satisfying unicity, $\Aut_C(C^a)$-stability
and $\ Aut_C(C^a)$-unicity. 

 Let 
 $\mathfrak Y^a$  denote  the set  of cuspidal $C^a$-types in $G$ of the form ($J', \ind_J^{J' }V^a)$, where $(J,V)\in \mathfrak X^a$ and $J' $ is the $G$-normalizer of $J$, and let  $\mathfrak Y$ denote  the set of cuspidal $C$-types $(J',V')$ in $G$ obtained by applying the decomposition theorem \ref{thm:DT}
to $\mathfrak Y^a$,

Let $\mathcal Z^a$ denote the set of isomorphism classes of $\ind_J^G V^a$ for $(J,V^a) \in \mathfrak X^a$ and  let $\mathcal Z$  denote the set of isomorphism classes of $C$-representations of $G$ obtained by applying
the decomposition theorem \ref{thm:DT} to $\mathcal Z^a$.

Then 

(i) $\mathcal Z^a$ is the set of isomorphism classes of $\ind_{J'}^G V'^a$ for $(J',V'^a) \in \mathfrak Y^a$ and  $\mathcal Z$ is the set of isomorphism classes of $\ind_{J'}^G V'$
for $(J',V') \in  \mathfrak Y$. 

(ii) The set  $\mathfrak Y$ satisfies unicity and intertwining.

(iii) If moreover $\mathfrak X^a$ satisfies $\Aut(C^a)$-stability and $\Aut(C^a)$-unicity, then $\mathfrak Y$ satisfies $\Aut(C)$-stability
and $\Aut(C)$-unicity.
\end{theorem}

The key to the proof of Proposition \ref{prop:cusptype2} is the next lemma. Let $\underline G^{ad}$  be the adjoint group of $\underline G$,   $f:  G\to G^{ad}=\underline G^{ad}(F)$ the natural group homomorphism  
with kernel $Z$ and   $\mathcal B = \mathcal B(G^{ad}) $ the Bruhat-Tits building of $G^{ad}$.
 Let  $J\subset G$ be an open subgroup containing $Z$ such  that $J/Z$ is compact. The $G$-normalizer $J'$ of $J$ which contains $J$
is open, but  $f(J) \subset G^{ad}$ might not be open  (when the characteristic of $ F$ is $2$,   the image of $SL(2,F)$ in  $PGL(2,F)$ is not  open because $1$ is not open in $F^*/(F^*)^2$).

\begin{lemma}\label{le:proof} The subset $ \mathcal B^ {f(J)} \subset \mathcal B $ of fixed points of $f(J) $ is  compact and non-empty.
\end{lemma}
\begin{proof}  The subgroup $f(J) \subset G^{ad}$ being  compact,  the set $ \mathcal B^{f(J)}$ is not empty because
any orbit of  $f(J)$ in $ \mathcal B$ is bounded hence contains a point fixed by its $G^{ad}$-stabilizer (\cite{BTI} 3.2.4). Let $\overline{\mathcal B}$ be the Landvogt compactification of $\mathcal B$. The  action of $f(J)$ on $\overline{\mathcal B}$ is continuous (\cite{L96} 14.31)
so the subset $\overline{\mathcal B} ^{f(J)}\subset \overline{\mathcal B}$ of fixed points of $f(J)$ is  closed, hence compact.  The open subgroup $J \subset G$
  is not contained in any proper parabolic subgroup of $G $ (which is $F$-analytic of dimension less than the dimension of $G$ and $J$). It follows that  $f(J)$ is not contained
in any proper parabolic subgroup of $G^{ad}$. This implies  $\overline{\mathcal B} ^{f(J)}= \mathcal B^{f(J)}$ (\cite{L96} (14.4 i),(12.4), (12.3), (2.4)  and the notations (0.20) and (0.15)). 
 Hence the lemma.
  \end{proof}

We finally prove Proposition \ref{prop:cusptype2}. It is clear that $J'$ contains $Z$ and is open as it contains $ J$. We prove that $J'$ is $Z$-compact. The  $G^{ad}$-normalizer  $K$ of  $f(J)$  is a closed subgroup of $G^{ad}$ stabilizing  $\mathcal B^ {f(J)}$, and  $\mathcal B^ {f(J)}$  is a  compact non-empty subset of $  \mathcal B$ by the lemma.
A non-empty compact subset of $\mathcal B$ is bounded in the metric space  $\mathcal B$ (\cite{BTI}, 2.5.1);  there exists $x \in \mathcal B^ {f(J)}$ fixed by  $K $   (\cite{BTI}, 3.2.4). By (\cite{BTI}, 3.3.1), the $G^{ad}$-stabilizer  
of $x$  is open and compact; as it contains $K$, we deduce that $K $  is compact. Hence   $f^{-1}(K)$  is $Z$-compact as well as  the closed subgroup $J' \subset  f^{-1}(K)$.

\subsection{Fields of the same characteristic} \label{ss:3.2}

For some groups $G$, a good set of types as in Theorem \ref{thm:cusptype} has been produced only for  $\mathbb C$.
  It is likely that all the arguments can be adapted to obtain such a set of types for all algebraically closed fields  of characteristic $0$, 
but that needs verification. Let $C, C'$ be two fields of the same characteristic $c$ and algebraic closure $C^a, C'^a$. We write $C_0 $ for the prime subfield of $C $ and $C_0^a$ for its algebraic closure in $C^a$.
Here we show directly, using twists by unramified characters, that a good set of $C^a$-types in $G$ gives rise to a good set of $C'^a$-types in $G$.


\bigskip   The centre $Z$ of $G$ acts on 
any irreducible smooth $C^a$-representation $\pi^a$ of $G$  by a  $C^a$-character $\omega^a $, called its central character,  because $\pi^a$ is admissible (Proposition \ref{pro:cusp}). 

\begin{theorem}\label{thm:c}  Let $\pi^a$ be an irreducible cuspidal $C^a$-representation of $G$. 
Assume that the central character  $\omega^a$ of $\pi^a$ has finite order. Then $\pi^a$ is defined over a finite extension of $C_0$ in $C^a$.
 \end{theorem}

\begin{proof} Because   $\omega^a$  has finite order by
assumption,  the subfield  $K \subset C^a$ generated by the values of $\omega^a$ is a finite cyclotomic extension of $C_0$.
 Choose an open compact subgroup $ J\subset G$  such that $(\pi^a)^J\neq 0 $. By Corollary \ref{cor:invariant2},  
it is enough to prove that the $H_{C^a}(G,J)$-module $(\pi^a)^J$ is defined over a finite extension of $C_0$. We know  that $\dim_{C^a} (\pi^a)^J=n$ is finite by admissibility.
 Moreover by $Z$-compactness  of  cuspidal representations (\cite{V96} II.2.7) we know that the coefficients vanish outside a finite
union of cosets $JgZJ$, and that there are only finitely many cosets $JgZJ$ which give non-zero operators $\pi^a(JgzJ) \in 
\End_{C^a}((\pi^a)^J), z\in Z$; 
 those operators generate a finite dimensional $K$-subalgebra of $\End_{C^a}((\pi^a)^J)$,
because $\pi^a(JgzJ)=\omega^a(z)\pi^a(JgJ)$.
Each such operator satisfies a (non-trivial) polynomial equation with coefficients in $K$. The finiteness
  up to $K^*$ shows that  the eigenvalues of the  $\pi^a(JgJ)$
taken together generate a finite extension $L/K$ in $C^a$; the characteristic polynomials of the $\pi^a(JgJ)$ all have their coefficients in $L$.
Choose a $C^a$-basis of $(\pi^a)^J$.  The coefficients of the  $\pi^a(JgJ) \in \End_{C^a}((\pi^a)^J) $ in the chosen basis generate a  finitely generated $L$-subalgebra $A\subset C^a$, and we get  an $L$-algebra homomorphism $\tau: H_L(G,J)\to  M(n,A)$.   Any quotient field $E$ of $A $  is a finite extension of $L$, and we get an $L$-algebra homomorphism $\overline \tau : H_L(G,J) \to M(n,E)$. 
Thus $\overline \tau$ gives an $H_E(G,J)$-module structure on $E^n$. Choose an extension  $E\to C^a$ of the embedding $L\to C^a$ (\cite{Bki-A5}  \S4, No 3,Corollaty 2 of Theorem 2).  That gives an $H_{C^a}(G,J)$-module
structure on $C^a \otimes_E E^n$. Let us prove that the two $H_{C^a}(G;J)$-modules $C^a \otimes_E E^n$ and $(\pi^a)^J$
are isomorphic, which shows that  $(\pi^a)^J$ is defined over $E$. Indeed for $h\in H_L(G,J) $, the characteristic polynomial of $\overline \tau(h)$   is the image in $E  [T]$  of     the characteristic polynomial of $\tau(h)$  which is 
the  characteristic polynomial   of $h  $ acting on $(\pi^a)^J$   and  has  coefficients in $L$. By (\cite{Bki-A8} \S20, No 6, Corollary 1  of Theorem 2), the two $H_{C^a}(G,J)$-modules $(C^a)^n$ and $(\pi^a)^J$ have isomorphic 
semisimplifications, but the second one is simple already, so  $(C^a)^n$ and $(\pi^a)^J$ are isomorphic.
 \end{proof}

 \begin{corollary}\label{cor:bc} Base change from $C_0^a$ to $C^a $ yields a bijection between isomorphism classes of irreducible cuspidal
$C_0^a$-representations of $G$, with central character of finite order, and isomorphism classes of irreducible cuspidal
$C^a$-representations of $G$, with central character of finite order. 
\end{corollary}
That bijection is clearly compatible with the action of
finite order  smooth characters, which have values in $C_0^a$. We will remove (in Corollary \ref{cor:d2} (iii)) the restriction in Corollary \ref{cor:bc}.

If $C'$ is another algebraically closed field  of characteristic $c$ then its prime field $C'_0$ is isomorphic to $C_0$, and so  are the algebraic closures $C_0^a$ in $C^a$ and $(C'_0)^a$ in $C'^a$. Consequently, 
any isomorphism $C_0^a\to (C'_0)^a$ will yield a bijection 
between isomorphism classes of irreducible cuspidal $C^a$-representations of $G$, with central character of finite order
and isomorphism classes of irreducible cuspidal $C'^a$-representations of $G$, with central character of finite order.

 \bigskip  A character of $G$  is said to be unramified if it is trivial on the subgroup $G^0$ generated by the compact subgroups of $G$. 
We will remove the restriction that the central character has finite order by twisting by  unramified $C^a$-characters of $G$. 

We choose    a  subgroup $Z^\sharp$ of  $Z$ isomorphic to $\mathbb Z^{n_Z}$ such that $Z $ is the product of $Z^\sharp$ and of  its maximal compact  subgroup $Z^{0}$.


\begin{proposition} \label{prop:Zsharp}  Let $\omega:Z\to (C^a)^*$ be a smooth  character. Then there is an unramified   character $\chi : G \to (C^a)^*$ such
 $\chi|_Z \, 
 \omega$ is trivial on $Z^\sharp$,  so has finite order.
 \end{proposition}

\begin{proof}  Since $Z/Z^{0}  \subset G/G^0$  has finite index, $\omega|_{Z^\sharp}$ can be extended to a $C^a$-character of $G$  trivial on $G^0$. Calling $\chi$ the inverse
of that character, $\chi|_Z \,  \omega :Z\to (C^a)^*$ is trivial on $Z^\sharp$ and coincides with $\omega$ on $Z^0$, so has finite order as $Z/Z^0 Z^\sharp$ is finite.
\end{proof}

\begin{corollary} \label{cor:Zsharp} Let $\pi$ be an irreducible smooth $C^a$-representation of $G$. There is an unramified $C^a$-character $\chi$
of $G$ such that  the central character of 
$\chi \pi$ has finite order. \end {corollary}

\begin{remark}\label{re:Zsharp}
  If $\pi$ already has a central character of finite order and  $\chi$  is an unramified character such that
the central character of $\chi \pi$ has finite order, then  $\chi |_Z$ has finite order, and so has  $\chi $ because
 $Z/Z^0$ has finite index in $G/G^0$. So in the corollary, the character $\chi $ is well-defined up to finite order unramified characters.
 \end{remark}

 Note  that twisting by unramified characters  preserves irreducibility, intertwining and cuspidality.
Stability under twisting by unramified characters is easy to verify
in all the known explicit constructions of types. 

\begin{proposition} \label{prop:d} Any irreducible smooth $C^a$-representation $\pi^a$ of $G$ is defined over a finite extension of $C$.
\end{proposition}
\begin{proof} a) There exists an unramified $C^a$-character $\chi^a$ of $G$ such that the central character of $\pi^a\chi^a$ has finite order (Corollary \ref{cor:Zsharp}).   If $\pi ^a\chi ^a$ is defined over a finite extension of $C$, then $\pi^a$ has the same property as the values of $\chi^a$ generate a finite extension of $C$ as $G/G^0$ is finitely generated.  

b) There is a parabolic subgroup $P=MN$ of $G$ and an irreducible cuspidal $C$-representation $\rho^a$ of $M$ such that $\pi^a$ is a subquotient of   $\Ind_P^G \rho^a$. 
We show that any  irreducible subquotient of 
$\Ind_P^G \rho^a$ is defined over a finite extension of $C$.  

b1) If $P=M=G$, $\pi^a$ is cuspidal this is clear by a) and Theorem \ref{thm:c}.

b2) In general, $\rho^a$ descends to a finite extension of $C$ by b1). Applying  the decomposition theorem \ref{thm:DT} there exists a finite extension $C'/C$ in $C^a$ and an absolutely irreducible $C'$-representation $\rho'$ of $M$ such that $\rho^a = C^a\otimes_{C'} \rho'$ (Proposition \ref{pro:cusp}  (v)).  If $C''/C'$ is an extension in $C^a$, the  length  $\ell(C'')$ of $\Ind_P^G (C''\otimes_{C'}\rho')$ is finite bounded (Proposition \ref{induction} and \cite{HV19} Corollary I.2). We have  $\ell(C''_1)\leq \ell(C''_2)$ for any finite extensions $C''_1 \subset C''_2$ of $C'$ in $C^a$. As an increasing bounded sequence of integers stabilize, there exists a finite extension $C''_1/C'$ such that $\ell(C''_1)=\ell(C''_2)$ for all finite extensions $C''_2/C''_1$ in $C^a$. The irreducible subquotients of $\Ind_P^G (C_1''\otimes_{C'}\rho')$ are admissible and  remain irreducible by any finite scalar extension. By Lemma \ref{le:DT} (ii), we deduce that any irreducible subquotient of 
$\Ind_P^G \rho^a$ is defined over $C''_1$.
\end{proof}

We can now remove the conditions $\dim_C\End_{C[G]}(\pi)<\infty$ and $\pi^a$ absolutely irreducible  and defined over a finite extension of $C$   in (\cite{HV19} Theorem III.4) as $c\neq p$.
We get:
 
\begin{corollary}\label{cor:d}  The map sending $\pi$ to the set of irreducible subquotients of $C^a\otimes_C \pi$ induces a bijection  from the set of isomorphism classes of irreducible smooth $C$-representations $\pi$ of $G$  to the set of orbits under $\Aut_C(C^a)$ of isomorphism classes of  irreducible smooth $C^a$-representations $\pi^a$ of $G$.
\end{corollary}

Let us now apply those considerations to cuspidal types.  Let $\mathcal Z^a$ denote  the set of isomorphism classes of cuspidal irreducible $C^a$-representations of $G$  with central character of finite
order.   Similarly $ \mathcal Z_0^a $ for the field $C_0^a$. 

\begin{corollary}\label{cor:d2}

(i) The base change  from $C_0^a$ to $C^a$ yields a bijection 
from

$\{ \mathfrak X_0^a, \  \mathfrak X_0^a \ \text{ is a set  of cuspidal $C_0^a$-types in $G$ satisfying  $\mathcal Z_0^a$-exhaustion}\}$ 
onto   

 $ 
\{\mathfrak X^a, \  \mathfrak X^a \ \text{ is a set  of cuspidal $C^a$-types in $G$  satisfying $\mathcal Z^a$-exhaustion}\}$.

(ii) Let $ \mathfrak X_0^a$ as in (i) and  $\mathfrak X^a$ its base change. Then 

$ \mathfrak X^a$ satisfies unicity if and only if $ \mathfrak X_0^a$ does, 

 $ \mathfrak X^a$ is $\Aut(C^a)$-stable if and only if $ \mathfrak X_0^a$  is $\Aut(C_0^a)$-stable,  
 
 $ \mathfrak X^a$ satisfies $\Aut(C^a)$-unicity if and only if $ \mathfrak X_0^a$ satisfies $\Aut(C_0^a)$-unicity,  
 
 $ \mathfrak X^a$  is  stable under
twisting by unramified  characters of $G$ with finite order if and only if $ \mathfrak X_0^a$ does.

(iii) If $ \mathfrak X^a$ as in (i) satisfies unicity and 
stability under unramified  characters with finite order of $G$, then
  the set $ \mathfrak  Y^a$ of cuspidal $C^a$-types in $G$ obtained from $\mathfrak X^a$ by twisting
by all unramified characters of $G$,  satisfies exhaustion for  the isomorphism classes of all cuspidal $\pi\in \Irr_{C^a}(G)$,  and unicity; it satisfies $\Aut(C^a)$-stability,  resp.  $\Aut(C^a)$-unicity,  
if and only if $\mathfrak X^a$ does.
\end{corollary}

\begin{proof} (i)
Base change from $C_0^a$ to $C^a$  of  a set $ \mathfrak X_0^a$ of cuspidal $C_0^a$-types in $G$ satisfying  $\mathcal Z_0^a$-exhaustion  yields a set $ \mathfrak X^a$ of cuspidal $C^a$-types in $G$ that satisfies $\mathcal Z^a$-exhaustion because of
the theorem. That works also in the reverse direction. Let $\mathfrak X^a$ be a set of cuspidal $C^a$-types in $G$ satisfying $\mathcal Z^a$-exhaustion and 
  $(J, V)\in \mathfrak X^a$.  Because $\ind_J^GV$ has central character of finite order, so has $V$, and as
remarked above $V$ is  the base change to $C^a$ of a $C_0^a$-representation $V^a_0$ of $J$, then
$\ind_J^G V$ is the base change to $C^a$ of $\ind_J^GV^a_0$, and consequently $(J, V^a_0)$ is a cuspidal $C_0^a$-type in $G$.
Note that $(J,V^a_0)$ is uniquely determined by $(J, V)$. If we let $ \mathfrak X_0^a$  be the set of types obtained from $ \mathfrak X^a$
in that manner, then $ \mathfrak X_0^a$ satisfies $\mathcal Z_0^a$-exhaustion. If  we base change again to $C^a$ we get back the set $\mathfrak X^a$ of $C^a$-types.

(ii) Because base changing types from $C_0^a$ to $C^a$  is compatible with $G$-conjugation, $  \mathfrak X^a$ satisfies unicity if and only if $ \mathfrak X_0^a$ does.

Note also that $\mathcal Z^a$ is $\Aut(C^a)$-stable and $\mathcal  Z_0^a$ is $\Aut(C_0^a)$-stable. All isomorphism classes of $C^a$-representations
obtained by base change from $C_0^a$ are obviously $\Aut_ {C_0^a}(C^a)$-invariant, so that $\Aut_{C_0^a}(C^a)$ acts trivially on
$\mathcal Z^a$ and $ \mathfrak X^a$, and the action of $\Aut(C^a)$ on $\mathcal Z^a$ and $ \mathfrak X^a$ factorizes through the quotient $\Aut(C_0^a)$.
The other properties of (ii)   are clear.

(iii) is clear.
 \end{proof}

\subsection{Level $0$ cuspidal types} \label{ss:3.3}  In this subsection, to conform to current usage in the relevant literature we write now the types $(J, \lambda)$ instead of $(J,V)$. 
 It is natural to conjecture that all irreducible cuspidal $C$-representations of $G$ are compactly induced from a cuspidal 
$C$-type $(J, \lambda)$ because all explicit examples have that form.

\bigskip  Of course, when $G $ has semisimple rank $0$  \footnote{The rank of $G$ is the $F$-rank of $\underline G$} , then $G/Z$ is compact, and all irreducible smooth $C$-representations
are cuspidal of finite dimension, so the conjecture is trivially true with $J=G$. Even so, it is interesting to have inducing types $(J,\lambda)$ where $J\neq G$. The example of the multiplicative group  of  a finite dimensional central division algebra over $F$ is examined in \cite{Z92}, \cite{Br98}.
   
   When $G$ has semisimple rank $1$, M. Weissman recently proved  the conjecture, using that the  building $\mathcal B$ of $G^{ad}$  is a tree.   
   If $\pi$ is an irreducible cuspidal $C$-representation of $G$, there is a  $C$-type $(J,\lambda)$ of   $G$  such that  $\pi\simeq \ind_J^G \lambda$ where $J$ is either the $G$-stabilizer   of a vertex   or  of of an edge, of $\mathcal B$.  
See \cite{W19} Corollary 2.6 when $C=\mathbb C$ and the note following it for general $C$.
 
  The other known cases assume $C$ algebraically closed (initially $C=\mathbb C$) and also make assumptions on $G$ or on the representations considered. But those cases include more precise information, in the guise of an explicit list of $C$-types   inducing to those representations. Very often, that list satisfies exhaustion for the kind  of representations considered,  and unicity. We will verify 
  stability by the group  of automorphisms of $C$  and  apply Theorem
\ref{thm:cusptype} to extend the result to a general coefficient field $C$, no longer assumed algebraically closed.

\bigskip   The case of level $0$ representations requires no assumption on $G$ and does not assume $C$ algebraically closed.  For any point $x $ in the  Bruhat-Tits  building $\mathcal B $ of the adjoint group $G^{ad}$,  we denote by 
  $G_x$   the $G$-stabilizer  of $x$,
 $G_{x.0} $ the  parahoric subgroup fixing $x$ and  $G_{x,0+} $ the pro-$p$ unipotent radical of $G_{x.0}$. The $G$-normalizer of $G_{x,0}$ is   $G_{x}$ {\color{red} if $x$ is a vertex} (\cite{Y01} Lemma 3.3 (i)) and 
 $G_{x,0}/G_{x,0+}$ is the group of  points of a connected reductive group over   the residue field $k_F$.  When $x$ is a vertex, it is known that  $G_{x}$  determines $x$ (the proof uses that two vertices $x\neq y$ are contained in the same apartment,  there is an affine root containing  $x$ and not  $y$ implying that  the corresponding unipotent subgroup of $G$ does not have the same intersection with $G_x$ and with $G_y$); this implies that $G_x$ is its own $G$-normalizer (as $gG_xg^{-1}=G_{gx}$  for $g\in G$).

  A   $C$-representation $ \pi $ of $G$ has {\it level $0$} if $\pi =\sum_x\pi^{G_{x,0+}}$, where $x$ runs through the vertices in $\mathcal B$, in particular  it is smooth.
For $\pi$ irreducible, this means that $\pi^{G_{x,0+}}\neq 0$ for some vertex  $x\in \mathcal B$. The category of level $0$ $C$-representations  of $G$ is a direct factor of $\Mod_C(G)$ and the parabolic induction respects level $0$.

Let  $\mathcal Z (0)$ denote the set of  isomorphism classes of level $0$ irreducible   cuspidal $C$-representations of $G$. Clearly,   $\mathcal Z (0)$  is  stable under  $\Aut(C)$.

 \begin{lemma}\label{le:0} Let $x, y \in \mathcal B$ be  two different  vertices and $\lambda  \in \Irr_C(G_x)$. If   $G_{x,0+}$ acts trivially  on $\lambda$ and $\lambda$ is  cuspidal as a representation of $G_{x,0}/G_{x,0+}$, then  $G_x\cap G_{y,0+} $ has no non-zero fixed vector in $\lambda$. 
\end{lemma} 
\begin{proof}By (\cite{V00}, lemma 5.2), 
the image of $G_{x,0} \cap G_y $ in $G_{x,0} /G_{x,0+}$ is a parabolic subgroup with unipotent radical the image of
 $G_{x,0} \cap G_{y,0+}$;  because $x$ and $y$ are two disctinct  vertices, that parabolic is not 
$G_{x ,0}/G_{x,0+}$ 
and the result comes from the  cuspidal assumption.
\end{proof}
     
 \begin{proposition}\label{prop:iso}   
 If   $G_{x,0+}$ acts trivially  on $\lambda  \in \Irr_C(G_x)$ and $\lambda$ is  cuspidal as a representation of $G_{x,0}/G_{x,0+}$, 
 then  the space of vectors fixed by $G_{x,0+}$ in $\ind_{G_x}^G \lambda$ is made out of the functions
with support in $G_x$; in particular it affords the representation $\lambda$ of $G_x$.
 \end{proposition}
\begin{proof}  Put $J=G_x, J^0=G_{x,0}, J^1=G_{x,0+}$. As in section  \ref{ss:2.4},
the restriction  of $\ind_J^G \lambda$ to $J^1$ splits as a direct sum $\oplus_{JgJ^1 }\ind_J^{JgJ^1} \lambda$ over the double cosets $JgJ^1$  of the subspaces $\ind_J^{JgJ^1} \lambda$ consisting of functions with support in $JgJ^1$. 
 The subspace of functions with support in $J$, as a representation of $J$,  is isomorphic to $\lambda$ and $\lambda$ is trivial on $J^1$.
  It is enough to show that for $g\in G\setminus  J$, a 
function in $\ind_J^G \lambda $ with support in  $Jg J^1 $  and  right invariant under $J^1$   is $0$.  
Putting $ x=gy,$ this follows from  Lemma \ref{le:0} as  $y \neq x$ for  $g\in G\setminus  G_x$ (and the isomorphisms  \eqref{eq:compactind3} and \eqref{eq:compactind4}).
\end{proof}

Proposition \ref{prop:iso2} will be a  generalization of  this proposition. 

\begin{corollary} \label{cor:0}  If   $G_{x,0+}$ acts trivially  on $\lambda \in \Irr_C(G_x)$ and $\lambda$ is  cuspidal as a representation of $G_{x,0}/G_{x,0+}$, then $ \ind_{G_x}^G \lambda$ is irreducible and $\End_{C[G]}(\ind_{G_x}^G \lambda)=\End_{C[G_x]}(\lambda)$.
\end{corollary}
\begin{proof} As before, put $J=G_x, J^0=G_{x,0}, J^1=G_{x,0+}$.
 A quotient of $\ind_J^G\lambda$ contains a representation of $J$ isomorphic to  $\lambda$ by Frobenius 
reciprocity for compact induction $\ind_J^G$, and  a subrepresentation of $\ind_J^G\lambda$ as a representation of $J$ has a quotient isomorphic to   $\lambda$, by
Frobenius reciprocity for smooth induction $\Ind_J^G $ and the inclusion  of  $\ind_J^G\lambda$ in $\Ind_J^G \lambda$.
But the restriction of $\ind_J^G\lambda$  to  the pro-$p$ group $J^1$ is semi-simple, and by the proposition  $\ind_J^G\lambda$, as a representation of $J$,   contains $\lambda$ as a subquotient only once. Hence $\ind_J^G\lambda$ is irreducible.
Similarly one infers that $\End_{C[G]}(\ind_J^G\lambda)=\End_{C[J]}(\lambda)$. Indeed, as in section  \ref{ss:2.4}, this means  that $\Hom_{C[J\cap g^{-1} J g]}(\lambda,{}^g \lambda)=0$ for all  double cosets $JgJ\neq J$. Putting $x=gy$ we have $g^{-1} J g=G_y$ with $x\neq y$ when  $JgJ\neq J$. In this case $G_y \cap G_{x,0+}$ has no non-zero fixed vector in $ {}^g \lambda $  by Lemma \ref{le:0},  but any vector in $\lambda$ is fixed by $G_y \cap G_{x,0+}$. \end{proof}

When  $C$ is algebraically closed, the   irreducibility of  $\ind_{G_x}^G \lambda$ is proved in \cite{V00}  with another proof - the result for $C=\mathbb C$ goes back to \cite{M99} and \cite{MP96}.  

\begin{corollary}  \label{cor:1} Assume that $ \ind_{G_x}^G \lambda\simeq \ind_{G_y}^G \mu $ for a vertex
  $y \in \mathcal B$  and $\mu\in \Irr_C(G_y)$ and that  $\lambda$   and $\mu$ as representations of $G_{x,0}$ and $G_{y,0}$  are the inflations of   cuspidal representations of   $ G_{x,0}/G_{x,0+}$ and  $ G_{y,0}/G_{y,0+}$ respectively. Then  $y=gx$ and $\mu = \lambda ^g$ for some $g\in G$.\end{corollary}
\begin{proof}
  If $y=gx$ for
some $g \in G$, we may conjugate $(G_y, \mu)$ to reduce to $y=x$ in which case the proposition implies 
 $\mu \simeq \lambda$.
  If $y $ is not of the form $gx$, then by the reasoning of the proposition, $G_{y,0+} \cap   G_x$  fixes no non-zero vector 
in $\lambda$, which yields a contradiction.
This argument rose out of conversations with R. Deseine.
\end{proof}

  \begin{definition} \label{def:level0}   A  level $0$ cuspidal  $C$-type in $G$ is  a pair $(J,\lambda)$ where  $J=G_x$ for some vertex $x$ of $\mathcal B$, and $\lambda$  is the isomorphism class of an  irreducible $C$-representation of $J$  trivial on $G_{x,0+}$
 and cuspidal as a representation of $G_{x,0}/G_{x,0+}$. If moreover  $\lambda$  is supercuspidal as a representation of $G_{x,0}/G_{x,0+}$, then we say that   $(J,\lambda)$ is supercuspidal. \end{definition}

 By Corollary \ref{cor:0},  a  level $0$ cuspidal  $C$-type $(J,\lambda)$ in $G$ is a cuspidal $C$-type (Definition  \ref{def:ctype}) since $\ind_J^G\lambda$ is irreducible.

   \begin{lemma}\label{le:scalar0}  Let  $x$ be a vertex of $\mathcal B$,  $\lambda\in \Irr_C(G_x)$  and  $\pi \in \Irr_C(G)$. Let $C'/C$ be a field extension, $\lambda' \in \Irr_{C'}(G_x)$   a subquotient of $C'\otimes_C\lambda$, and   $\pi'\in \Irr_{C'}(G)$   a subquotient of $C'\otimes_C \pi$.  

(i) $\pi$ has level $0$   if and only if $\pi'$ has level $0$.
 
(ii)  $(G_x,\lambda)$ is a  level $0$ cuspidal (resp. supercuspidal)  $C$-type in $G$ if and only if  $(G_x,\lambda')$ is a  level $0$ cuspidal (resp. supercuspidal)   $C^a$-type in $G$.  
 \end{lemma}
 
\begin{proof} 
(i) $\pi^{G_{x,0+}}\neq 0$  if and only   if $(\pi')^{G_{x,0+}}\neq 0$ (\cite{HV19}, III.1).  
 
 (ii)  As $C$-representations,  $\lambda'$  is a direct sum of 
 representations isomorphic to $\lambda $ (because $C'\otimes_C \lambda$ is). So
 $\lambda$   is  trivial on $G_{x,0+}$  if and only if $\lambda'$  does.  
  In  \S \ref{ss:c} which is valid for finite reductive groups, we  saw that $\pi$    is cuspidal  if and only if $\pi'$  is cuspidal, similarly for supercuspidal (Proposition \ref{prop:c}).  We deduce that    $\lambda$ is the inflation of a cuspidal  (resp. supercuspidal) representation of $G_{x,0}/G_{x,0+}$ 
if and only if $\lambda'$  does.  
   \end{proof}

\begin{theorem} \label{thm:level0} The set $\mathfrak X(0)$  of level $0$ cuspidal  $C$-types in $G$ satisfies intertwining, 
unicity, $\mathcal Z (0)$-exhaustion, and  $ \Aut (C)$-stability.  
\end{theorem}
 \begin{proof}   The set $\mathfrak X(0)$  satisfies intertwining by Corollary \ref{cor:0} and unicity by Corollary \ref{cor:1}; it is   $ \Aut (C)$-stable as cuspidality is preserved under the action of   $ \Aut (C)$ on $C$-representations of a finite reductive  group.   

When $C$ is algebraically closed,  $\mathcal Z (0)$-exhaustion (and unicity) is   in \cite{M99} and \cite{MP96} when $C=\mathbb C$ and the arguments of \cite{MP96}  carry over to $C$; exhaustivity  was implicit in \cite{V00}, and is established by Fintzen at the end of \cite{F2} (the hypothesis on  $G$ of \cite{F2}  plays no r\^ole for level $0$ representations). 

When  $C$ is not algebraically closed, $\mathcal Z (0)$-exhaustion  follows from
$\mathcal Z (0)$-exhaustion over $C^a$ by  Theorem \ref{thm:cusptype}   noting that the group $J=G_x$ is its  own $G$-normalizer.
\end{proof}

\begin{corollary} \label{cor:level0}  Any irreducible cuspidal $C$-representation $\pi$ of $G$ of level $0$ is compactly induced from a  level $0$  cuspidal $C$-type $(J,\lambda)$ in $G$ unique modulo $G$-conjugation; it satisfies intertwining $\End_{C[G]} \pi \simeq \End_{C[J]}\lambda$.
\end{corollary}

\section{Supercuspidality in level $0$}\label{s:4}
 
 Let  $(J,\lambda)$ be a level $0$ cuspidal $C$-type of $G$ inducing a level $0$ cuspidal irreducible representation  $\pi =\ind_J^G \lambda$ of $G$.  
Our goal is to  prove:

 \begin{theorem}\label{thm:sct0}  $(J,\lambda)$ is supercuspidal if and only if    $\pi $  is supercuspidal. 
\end{theorem}  

The equivalence  will be a  consequence of  Theorems \ref{thm:sc01} (for only if) and \ref{thm:sc0} (for if)  below in  \S\ref{ss:5}. We use injective hulls as our main tool.

\subsection{Injectives in the category of representations with a fixed action of the center}\label{ss:11}  Only in this subsection, $C$ is a field of any  characteristic. We fix  a closed subgroup  $Z^\flat$ of the center $Z$.  The abelian  group $Z^\flat$ is almost finitely generated (Definition \ref{def:almostfinite}).
   Let $\omega$ be an irreducible smooth $C$-representation of $Z^\flat$. The dimension of $\omega$  is  finite   (Proposition \ref{prop:finitedim} applied to $V=\omega, G=Z=Z^\flat$).  The  dimension  of $\omega$ is $1$ if and only if  $\omega$  is absolutely irreducible if and only if 
   $\End_{C[Z^\flat]}\omega=C$.  
   
   \begin{remark}\label{re:C'} The $C$-algebra $C[Z^\flat]$ acts on $\omega$ via a quotient field $C'$ which is a finite extension of $C$, and   $\End_{C[Z^\flat]}\omega=C'$ (as $Z^\flat$ is abelian).  
\end{remark}
   
   For  $\tau \in \Mod_C(G)$,  the    $\omega$-isotypic part  $\tau_\omega$   of $\tau$ is the sum of the subrepresentations of $\tau|_{Z^\flat}$ isomorphic to $\omega$. Because $Z^\flat$ is central in $G$, $\tau_\omega$ is a subrepresentation of $\tau$.
The representation $\tau$ is called $\omega$-isotypic if  $\tau=\tau_\omega$.  Any irreducible $C$-representation $\pi$ of $G$ is  $\omega$-isotypic for some $\omega\in \Irr_C(Z^\flat)$.   Let  $\Mod_C(G,\omega) $ denote the category   of $\omega$-isotypic representations in $\Mod_C(G)$ and  $\Irr_C(G,\omega)=\Irr_C(G)\cap  \Mod_R(G,\omega)$.     For  a  parabolic subgroup $P=MN$  of $G$,  the parabolic induction $\Ind_P^G$ and its adjoints give functors between $\Mod_C(G,\omega)$ and   $\Mod_C(M,\omega)$.

  \begin{lemma}\label{le:injomega} $\Mod_C(G)$ and  $\Mod_C(G,\omega) $ are Grothendieck  abelian categories.
  \end{lemma} 
 \begin{proof} The proof for  $\Mod_C(G)$ (\cite{V13} Lemma 3.2)  extends to  $\Mod_C(G,\omega) $.   \end{proof}
Recall that a Grothendieck category admits sufficiently many injectives  (any object embeds in an injective object)  and that every object has an injective hull (an essential extension which is injective) 
(\cite{L99} \S 3D).  

\begin{notation} For $\tau\in  \Mod_C(G, \omega)$, we denote by  $I_\tau$  an injective hull of $\tau$   in  $ \Mod_C(G)$ and  by $I_{\tau,\omega}$  an injective hull   in $ \Mod_C(G, \omega)$.
  \end{notation}

    \begin{lemma}\label{le:inomega} Let  $\tau\in  \Mod_C(G, \omega)$. The $\omega$-isotypic part of an injective hull   of $\tau$ in 
$  \Mod_C(G)$  is an injective hull     of  $\tau$ in $\Mod_C(G,\omega) $.
  \end{lemma} 
  \begin{proof}  If  $\pi \in \Mod_C(G)$ is injective, its  $\omega$-isotypic part  $\pi_\omega$   is injective in  $\Mod_C(G, \omega)$,  as $\Hom_{C[G]}(\tau, \pi)=\Hom_{C[G]}(\tau, \pi_\omega)$ for any $\tau \in \Mod_C(G,\omega)$.
  As a consequence   $I_{\tau, \omega}$  is isomorphic to a direct summand of the $\omega$-isotypic part of $I_\tau$.
  A supplement $ I'$  has a trivial intersection with $\tau$. 
   As  $I_\tau$ is an essential extension of $\tau$ containing $ I' $ we have  $I'=0$, hence the result. \end{proof}

\subsection{Supercuspidality and injective hulls} Let us revert to our running hypothesis that  $C$ has characteristic different from $p$. Supercuspidality can be seen on the injective hull; this was proved by Hiss for  finite reductive groups (\cite{Hiss96} Proposition2.3):
 \begin{lemma} \label{le:Hiss} When $G$ is a finite reductive group in characteristic $p$, an irreducible $C$-representation $\pi$ of $G$ is supercuspidal if and only if an injective hull $I_\pi$ of $\pi \in \Mod_C(G)$ is cuspidal.
 \end{lemma}
Hiss formulates this result in terms of projective cover, but in that case $I_\pi$ is a projective cover.
We imitate the proof of Hiss to show:

     \begin{proposition}\label{prop:scI} Let  $\pi \in \Irr_C(G, \omega) $.  Then   $\pi$ is supercuspidal if and only if $I_\pi$ is right cuspidal.
        If the second adjunction holds for $(G,C)$ or if $\pi$ has level $0$, then  $\pi$ is supercuspidal if and only if $I_{\pi, \omega}$ is  cuspidal.
       \end{proposition}

 \begin{proof}    By Definition  \ref{def:sc},  
   $\pi$ is supercuspidal if and only if
  $\pi$  is not a subquotient of $\Ind_P^G\rho$. 
  for all  proper parabolic subgroups $P=MN$ of $G$ and $\rho\in \Mod_C(M)$.
  We have
 \begin{equation}\label{eq:scinj} \text{ $\pi\in \Irr_C(G )$ is  a subquotient of $\tau \in  \Mod_C(G)$ if and only if $ \Hom _{C[G]}(\tau, I_\pi) \neq 0 $.}
 \end{equation}  
Indeed,  if $f\in \Hom_{C[G]}(\tau, I_\pi)$  is  non-zero   then  $\pi \subset f(\tau)$ hence $\pi$ is a subquotient of $\tau$;  conversely if $\pi$ is a subquotient ot $\tau$, then $\pi$ is  a subrepresentation of  a quotient $\tau'$ of $\tau$. The inclusion of $\pi$ in $I_\pi$ extends to a $R[G]$-map $\tau'  \to I_\pi$ inflating to a $R[G]$-map $\tau\to I_\pi$. 

      By adjunction,  $\pi$ is supercuspidal  if and only if 
 $R^G_P(I_\pi)=0$ for all proper parabolic subgroups $P $ of $G$, which means by definition that  $I_\pi$ is right cuspidal.  Right cuspidality passes to subrepresentations because   $R^G_P$   is left exact, so  $I_\pi$   right cuspidal implies 
  $I_{\pi, \omega} $  right cuspidal.

Conversely,  $I_{\pi, \omega} $  right cuspidal implies  $\pi$ is supercuspidal  when   the second adjunction holds  for $(G,C)$ or $\pi$ has level $0$ because in this case    $\pi$ is supercuspidal if and only if
  $\pi$  is not a subquotient of $\Ind_P^G\rho$ 
  for all  proper parabolic subgroups $P=MN$ of $G$ and  $\rho \in \Irr_C(M)$  (Remark   \ref{re:scirr}). As $\rho \in \Irr_C(M)$ is $\omega_\rho$-isotypic for some $\omega_\rho \in \Irr_C(Z^b)$, 
 the representation $\Ind_P^G\rho$ is also $\omega_\rho$-isotypic.   If $\pi$ is a subquotient of $\Ind_P^G\rho$ with $\rho \in \Irr_C(M)$,  then $\omega=\omega_\rho$, and  $\Hom_{C[G]}(\Ind_P^G\rho, I_\pi)= \Hom_{C[G]}(\Ind_P^G\rho, I_{\pi, \omega})$.
   By adjunction,  we deduce that $\pi$ is supercuspidal if and only if $I_{\pi, \omega}$ is  right cuspidal. Our assumption implies that $I_{\pi, \omega}$ right cuspidal  is equivalent to
   $I_{\pi, \omega}$ cuspidal  (Remark \ref {re:rightcusp}, recalling that level zero representations form a direct factor in $Mod_C(G)$).  \end{proof}  
   
  Recall that a functor between abelian categories having an exact  left adjoint   respects injectives, similarly a functor having an exact  right  adjoint respects projectives (\cite{HS} II.10).

\begin{example}  \label{ss:injpro} a) $\Ind_P^G$ and $R^G_P$  respect injectives (the left adjoint functors  $(-)_N$  and $\Ind_P^G$ are exact),
 and  $(-)_N$ respects projectives (its right adjoint  $\Ind_P^G$ is exact).

b)   Let   $C\to C'$ a field homomorphism. The scalar extension $C'\otimes_C-:\Mod_C(G)\to \Mod_{C'}(G)$  respects projectives  and  the restriction (right adjoint of the extension) respects injectives (they are both exact).

c)  Let $J$ be  an open subgroup of $G$ containing $Z$ and $\omega\in \Irr_C(Z^b)$. 
The restriction $\Res_{J}^{G}:\Mod_C(G, \omega)\to \Mod_C(J, \omega)$ has a right  adjoint  the smooth  induction $\Ind_J^{G} $ and a left adjoint the compact induction  $\ind_{J}^{G}$ (\cite{V96} I.5.7; to see that $\ind_J^G$ and $\Ind_J^G$ preserve $\omega$-isotypic representations,
 use Remark \ref{re:C'}).  The three functors are exact (\cite{V96} I.5.9, I.5.10). The restriction $\Res_{J}^{G}$ respects injectives and projectives,  the compact induction  $\ind_{J}^{G}$ respects projectives and  the smooth induction $\Ind_{J}^{G}$ respects injectives.   

  Let $J^1$ be a normal subgroup of $J$ such that $\omega$  inflates a representation $\omega^0$ of $Z^\flat / (Z^\flat \cap J^1)$.  The  $J^1$-invariant functor $\Mod_C(J,\omega)\to \Mod_C(J/J^1,\omega^0)$ is  right adjoint to 
 the inflation. The inflation  is exact,  preserves injectives, and    identifies $\Mod_C(J/J^1,\omega^0)$ to the  category  $\Mod_C(J,\omega)^{J^1}$ of representations in $ \Mod_C(J,\omega)$  trivial on $J^1$. 

If moreover the pro-order of $J^1$ is invertible in $C$,  the  $J^1$-invariant functor is exact and $\Mod_C(J,\omega)^{J^1}$ is a direct factor of $ \Mod_C(J,\omega)$. So the inflation and the $J^1$-invariant preserve injectives  and projectives, and the inflation preserves injective hulls.

\end{example}

\subsection{Supercuspidality and types}\label{ss:5} In the remaining of Section \ref{s:4} we assume that $Z^\flat=Z^\sharp$.  
    Let  $(J,\lambda)$ be a level $0$ cuspidal $C$-type of $G$ where $\lambda$ is $\omega$-isotypic for $\omega\in \Irr_C(Z^\sharp)$, $I_{\lambda, \omega}$ an injective hull of $\lambda$ in $\Mod_C(J,\omega)$, and $J^0$ and $J^1$ as in  the proof of Proposition \ref{prop:iso}. Put $\pi=\ind_J^G\lambda$.
   
    Since $J^1$ is a pro-$p$ normal subgroup of $J$,  $I_{\lambda, \omega}$ is trivial on $J^1$ and   $\ind_J^GI_{\lambda, \omega}$ is injective in $\Mod_C(G,\omega)$  (Example \ref{ss:injpro} c).  Since $\ind_J^GI_{\lambda, \omega}$ contains $\pi$, 
      the injective hull $I_{\pi,\omega}$ of $\pi$ in $\Mod_C(G,\omega)$ is a direct factor of     $\ind_J^GI_{\lambda, \omega}$.
The compact induction $\ind_J^G$ induces an injective inclusion preserving map from the lattice   of  subrepresentations of $I_{\lambda, \omega}$ to the lattice of subrepresentations of  $\ind_J^GI_{\lambda, \omega}$. 

 \begin{proposition} \label{prop:Ilambda} $I_{\lambda,\omega}$ is finite dimensional, projective, indecomposable with socle and  cosocle  isomorphic to $\lambda$.
  
   $(J,\lambda)$ is supercuspidal if and only if  $I_{\lambda,\omega}$ is cuspidal as a representation of $J^0/J^1$.
 \end{proposition}
 
 The proof will be given  in \S \ref{ss:6}.  Let us assume the proposition and prove:

\begin{theorem} \label{thm:sc01}  
Assume   $(J, \lambda)$  supercuspidal. Then $\pi=\ind_J^G\lambda $  is supercuspidal.  

Moreover $\ind_{J}^{G}I_{\lambda, \omega}$
is an injective hull  $I_{\pi,\omega}$ of $\pi $ in $\Mod_C(G, \omega)$.  It is   cuspidal projective indecomposable  with socle and  cosocle  isomorphic to $\pi$.   The lattices of subrepresentations  of $I_{\lambda, \omega}$  and  of  $I_{\pi,\omega}$
  are isomorphic  by the map    $W \mapsto \ind_J^GW$ (which is equal to $\Ind_J^GW$), with inverse
  $V\mapsto V^{J^1}$.
     \end{theorem}
 
\begin{proof} By Proposition \ref{prop:Ilambda}, $I_{\lambda, \omega}$ has finite length and is cuspidal as a representation of $J^0/J^1$.  Any irreducible subquotient $\mu$ of $I_{\lambda,\omega}$ is   cuspidal as a representation of $J^0/J^1$. By Proposition \ref{prop:iso},   $\ind_J^G\mu=\Ind_J^G\mu$  is cuspidal and $\mu=(\ind_J^G\mu)^{J^1}$. By induction on the length, this is also true for any subquotient of  $I_{\lambda,\omega}$. This gives the last assertion of the theorem, and that $\ind_J^GI_{\lambda,\omega}$ has finite length,  is cuspidal and is projective (since
$\ind_J^G$ preserves projectives and $\ind_J^GI_{\lambda,\omega}= \Ind_J^GI_{\lambda,\omega}$). As the socle and the cosocle of $I_{\lambda,\omega}$ are both isomorphic to $\lambda$,   the socle and the cosocle of $\ind_J^GI_{\lambda,\omega}$ are  both isomorphic to $\pi$. As $I_{\lambda,\omega}$ is indecomposable, $\ind_J^GI_{\lambda,\omega}$ is indecomposable and is an injective hull  $I_{\pi,\omega}$ of $\pi$ in $\Mod_C(G,\omega)$. Since $I_{\pi,\omega}$ is cuspidal,    $\pi$ is supercuspidal (Proposition \ref{prop:scI}).
 \end{proof} 
 
In the reverse direction:
  \begin{theorem} \label{thm:sc0}  Assume  $\pi=\ind_J^G\lambda$  supercuspidal. Then $(J,\lambda)$  is  supercuspidal and $I_{\pi,\omega}^{J^1}$  is an injective hull $I_{\lambda, \omega}$ of  $\lambda$ in $\Mod_C(J,\omega)$. \end{theorem}
 
\begin{proof} Since $\pi$ has level $0$ and is supercuspidal,  $I_{\pi,\omega}$ is cuspidal and has level $0$ (Proposition \ref{prop:scI} and its proof). Let $\tau$ be an irreducible subquotient of $I_{\pi,\omega}$. Then $\tau$ is cuspidal of level $0$, and  induced from a level $0$ cuspidal type $(G_y, \mu)$ (Corollary \ref{cor:level0}). If $G_y$ is not conjugate to $J$ in $G$ then $\tau^{J^1}=0$ (proof of Corollary \ref{cor:1}). If $G_y$ is  conjugate to $J$ we may take $G_y=J$ and then $\tau^{J^1}=\mu$ (Proposition \ref{prop:iso}). We deduce that $I_{\pi,\omega}^{J^1}$ is cuspidal  as a representation of $J^0/J^1$, by the following lemma.

\begin{lemma} Let $\tau\in \Mod_C(G)$. If $\rho^{J^1}$ is cuspidal or $0$ as a representation of $J^0/J^1$ for  each irreducible subquotient $\rho$ of $\tau$, then the same is true for 
 $\tau$.
\end{lemma}
\begin{proof} Let  $P=MN$ a proper parabolic subgroup of $J^0/J^1$. Assume that  there exists $f\in \tau^{J^1}$ such that  the average $f_N$ of $f$ along $N$ is not $0$.  Let $\rho$ be an  irreducible quotient  of  the subrepresentation of $\tau$ generated by $f_N$. The  image  of $f_N$ in $\rho$  is not $0$ and is fixed by $J^1$. Hence   $\rho^{J^1}$ is not $0$ and $\rho$ is not cuspidal as a representation of $J^0/J^1$, a contradiction proving the lemma.
 \end{proof}

As $I_{\pi,\omega}^{J^1}\in \Mod_C(J, \omega) $ is injective (Example \ref{ss:injpro} c) and contains $\lambda=\pi^{J^1}$ (Proposition \ref{prop:iso}), $I_{\lambda, \omega}$ is a direct factor of $I_{\pi,\omega}^{J^1}$.  As $I_{\pi,\omega}^{J^1}$ is cuspidal as a representation of $J^0/J^1$, the same is true for  $I_{\lambda, \omega}$ hence $(J,\lambda)$ is supercuspidal  (Proposition \ref{prop:scI}). By Theorem \ref{thm:sc01}, $\ind_J^G I_{\lambda, \omega}$ is an injective hull of $\pi$ in $\Mod_C(G, \omega) $ and  $(\ind_J^G I_{\lambda, \omega})^{I^1} =I_{\lambda, \omega}$.
 That proves the theorem.
\end{proof}

\subsection{Proof of Proposition \ref{prop:Ilambda}}\label{ss:6} 

With the notations of  \S\ref{ss:5} we  put  $H=J/J^1$. As $Z^\sharp \cap J^1$ is trivial, $Z^\sharp$ identifies with a subgroup $Y$ of $H$,  $\omega $ with $\zeta\in \Irr_C(Y)$, $\lambda$ inflates $\tau \in  \Mod_C(H,\zeta)$ and  $I_{\lambda, \omega} $ inflates 
 an injective hull  $I_{\tau, \zeta}$ of $\tau$ in  $\Mod_C(H,\zeta)$.

\bigskip In general,  let $H$ be a group with a central subgroup $Y$ of finite index,  $\zeta\in \Irr_C(Y)$, 
 $\tau\in \Irr_C(H,\zeta)$ and  $I_{\tau, \zeta}$ an injective hull  of $\tau$ in  $\Mod_C(H,\zeta)$.

 By adjunction,  $\ind_Y^H\zeta$ is a generator of  the abelian category $\Mod_C(H,\zeta)$, or equivalently, the functor $\Hom(\ind_Y^H\zeta, -)$ is faithful in  $\Mod_C(H,\zeta)$ (\cite{KS06} Proposition 5.2.4).  As $Y$ has finite index in $H$,   $\ind_Y^H\zeta$ is projective in  $\Mod_C(H,\zeta)$. By Morita theory, the category $\Mod_C(H,\zeta)$ is equivalent to the category of right modules over the $C$-algebra 
$\End_{C[H]}(\ind_Y^H\zeta).$ The algebra $C[Y]$ acts on $\zeta$ via a quotient field $C'$ which is a finite extension of $C$ and $\End_{C[Y]} \zeta= C'$ (Remark \ref{re:C'}). The convolution $C$-algebra $\mathcal H $ of functions $f:H\to C'$ such that $f(yh)=\zeta(y) f(h)$ for $y\in Y, h\in H$,  is isomorphic to $\End_{C[H]}(\ind_Y^H\zeta)$.
    
Recall (\cite{L99} (16.54)) that  a  finite dimensional $C$-algebra $A$ is called symmetric, if there exists  a linear map $\lambda $ on $A$  satisfying 
    $\lambda(ab)=\lambda(ba)$ for $a,b\in A$, and $\Ker (\lambda)$ does not contain a non-zero right  ideal of $A$.  

\begin{lemma} $\mathcal H$ is a symmetric $C$-algebra.
\end{lemma}
\begin{proof} We choose a  $C$-linear map $\phi:C'\to C$ with $
\phi(1)\neq 0$ and we consider the linear map $f \mapsto \lambda (f)=\phi (f(1) )$ on $\mathcal H$.  
  We have  $\lambda(f*f')=\lambda(f'*f)$ for $f,f'\in \mathcal H$ because 
 $(f*f')(1)=(f'*f)(1)$. 
 For $g\in G$ and $d\in C'$,  $d\neq 0$,  let $e_{g,d}\in \mathcal H$ with support $gZ$ such that $e_{g,d}(g)=d$.  
  For any non-zero $f\in \mathcal H$,  there exists $g\in G$  and $d\in D$,  $d\neq 0$,  such that $ (f*e_{g,d} )(1)=1$. This shows that $\Ker(\lambda)$ does not contain a  non-zero right  ideal of $\mathcal H$.
\end{proof}

\begin{proposition} \label{prop:H}  $I_{\tau, \zeta}$   is finite dimensional, projective, indecomposable with $\tau$ as its socle and cosocle.
\end{proposition}

\begin{proof} As   $\mathcal H$ is a symmetric $C$-algebra, any simple $\mathcal H$-module $\rho$ has a finite dimensional projective cover which is also an injective hull (\cite{L99} Corollary 16.64); it is consequently indecomposable with $\rho$ as its socle and cosocle. The proposition follows by Morita equivalence.
\end{proof}

We consider now only   the example   $(H,Y,\zeta,\tau, I_{\tau, \zeta})$ given in the beginning of this subsection \S\ref{ss:6}.  Then the first assertion of the proposition  \ref{prop:Ilambda} follows from  Proposition \ref{prop:H}. Put $H^0=J^0/J^1$ with the notations of  \S\ref{ss:5}.  Then $H^0$ is a finite normal subgroup of $H$ such that $Y\cap H^0$ is trivial.  Let $\rho $ be an irreducible quotient  of $\tau|_{H^0}$ and $I_\rho$  an injective hull of $\rho$ in  $\Mod_C(H ^0)$.

 \begin{proposition} \label{prop:H0} The restriction of $I_{\tau,\zeta}$ to $H^0$ is a sum of $H$-conjugates of $I_\rho$.
 \end{proposition}

\begin{proof} a) We restrict first to $H^0Y$. Let $\rho'\in \Irr_C(H^0Y)$ be a quotient of   $\tau|_{H^0Y}$ and $I_{\rho',\zeta}$  an injective hull of $\rho'$ in  $\Mod_C(H ^0Y, \zeta)$. By Proposition  \ref{prop:H},  $I_{\rho',\zeta}$ is indecomposable, projective with  $\rho'$ as its socle and cosocle. By Mackey formula, the restriction of  $\ind_{H ^0Y}^HI_{\rho',\zeta}$ to $H^0$ is a finite direct sum of $H$-conjugates of $I_{\rho',\zeta}$. The representation  $\ind_{H ^0Y}^HI_{\rho',\zeta}$
is  injective in $\Mod_C(H,\zeta)$ and contains $\ind_{H^0Y}^H \rho'$. By adjunction $\tau \subset \ind_{H^0Y}^H \rho'$ hence $I_{\tau, \zeta}$ is a direct factor of $\ind_{H ^0Y}^HI_{\rho',\zeta}$. Therefore  the restriction of $I_{\tau,\zeta}$ to $H^0Y$ is a direct sum of $H$-conjugates of $I_{\rho', \zeta}$.  

  We consider now  the  restriction  of  $I_{\rho', \zeta} $  to $H^0$.
  The functor  $$V^0\mapsto V^0\otimes_{C'} \zeta: \Mod_{C'}(H^0) \to \Mod_C(H^0Y, \zeta)$$ is an equivalence of categories which is $H$-equivariant. Write $\rho'=\rho'^0\otimes_{C'} \zeta$ with 
$\rho'^0 \in \Irr_{C'}(H^0) $ and  $I^0$ for an injective envelope of $\rho'^0$ in $\Mod_{C'}(H^0)$. Then $I_{\rho', \zeta}\simeq I^0\otimes_{C'} \zeta$, so   $I_{\rho', \zeta} |_{H^0}$   is isomorphic to   $I^0$ seen as a $C$-representation  by Krull-Remak-Schmidt's theorem.  

b) We assume, as we may, that $\rho$ is a quotient of $\rho'|_{H^0}$ and we show that, seen as a $C$-representation,  $I^0$ is a direct sum of copies of $I_\rho$.

Seen as a $C$-representation $\rho'^0$  is equal to $\rho$ by our assumption.
 Extending scalars from $C'$ to $C$  preserves projectives (Example 
\ref{ss:injpro} b)); as $I_\rho$ is projective in $\Mod_C(H^0)$, so is $C'\otimes_{C} I_\rho$ in $\Mod_{C'}(H^0)$.  But $I^0$ is a projective cover of $\rho'^0$ in $\Mod_{C'}(H^0)$,     so $I^0$ is a direct factor of $C'\otimes_{C} I_\rho$. As a $C$-representation, $C'\otimes_{C} I_\rho$ is  a direct sum of copies of $I_\rho$, so is $I^0$ by Krull-Remak-Schmidt's theorem.
    \end{proof}

 From that proposition,
$I_\rho$  is cuspidal if and only if  $I_{\tau,\zeta}|_{H^0}$  is cuspidal, or equivalently,  $I_{\lambda, \omega}$ is cuspidal as a representation of $J^0/J^1$.
The second assertion of Proposition  \ref{prop:Ilambda} follows, as  $\lambda$ is supercuspidal if and only if  $I_\rho$  is cuspidal.

 \section{Positive level cuspidal types} \label{s:pos}
  
We use the notations  of section \ref{s:3}. 

 \subsection{Positive level}\label{ss:3.4} An irreducible  smooth $C$-representation  of $G$,  or a  cuspidal $C$-type  in $G$ (Definition \ref{def:ctype}) which is not of level $0$  (Definition \ref {def:level0}),  is said to be of positive level.
  The  known cases of   cuspidal $C$-types require special assumptions on $G$, but give types for all positive level  irreducible cuspidal $C$-representations. Pioneer investigations were done in the 1970's  by G\'erardin  and Howe, but  the main results originate either from Bushnell-Kutzko's approach \cite{BK93} or from J.-K.Yu's construction \cite{Y01}.

\bigskip In both approaches, positive level  cuspidal $C$-types $(J,\lambda)$ are constructed  - when  $C$ is algebraically closed -  via an explicit but intricate procedure.  Actually, there is a general procedure based on facts established in each case. We now explain that procedure for application   to $\Aut (C)$-stability, and  also to supercuspidality  in \S \ref{ss:scpositive}. We shall give the references  in  each case below in  \S \ref{ss:3.5} and \S \ref{ss:3.6}.  

We assume that $C$ is algebraically closed, only at the very end do we generalize to non algebraically closed field.

\bigskip  The group $J$ possesses a filtration by normal open subgroups $J\supset J^0 \supset J^1 \supset H^1$ where $J^0/J^1$ is  the group of points of a connected reductive group over the residue field $k_F$ of $F$, $J^1$ is a pro-$p$ group, and $J^1/H^1$ is a finite $\mathbb F_p$-vector space. 
 
The starting datum  for the  representation $\lambda\in \Irr_C(J)$ is  a {\it very special} smooth $C$ character $\theta$  of $H^1$ satisfying:

{\bf (i)} The $G$-normalizer of $\theta$ is  $J$, 

{\bf (ii)} The  $G$-intertwining  of $\theta$  is $JG'J$ where $G'\subset G$ is the group of  $F'$-points of a connected reductive group defined over a finite extension $F'$ of $F$, and there is a vertex $y$ in the Bruhat-Tits building $\mathcal B (G'^{ad})$ of the  adjoint group $G'^{ad}$ of $G'$ such that 
$G'\cap J= G'_y$,  $G'\cap J^0= G'_{y,0}, \ G'\cap J^1= G'_{y,0+}$; in particular, the inclusion $G'\subset G$ induces isomorphisms $J/J^1\simeq G'_y/G'_{y,0+}$ and $ J^0/J^1\simeq G'_{y,0}/G'_{y,0+}$.

{\bf (iii)} There is a unique representation $\eta=\eta_\theta\in \Irr_C(J^1)$ restricting on $H^1$ to a multiple of $\theta$.

{\bf (iv)}  The  $G$-intertwining  of $\eta$ is the   $G$-intertwining of $\theta$,  and  for  any $g\in G$ intertwining $\eta$ the $C$-dimension of  the space $I(g,\eta)$ of $g$-intertwiners of $\eta$  is $1$. 

{\bf (v)} There are  {\it preferred} extensions of $\eta$ to irreducible $C$-representations $\kappa$ of $J^0$; preferred means that  the   $G$-intertwining  of the restriction of $\kappa$ to a    pro-$p$ Sylow subgroup of $J^0$ contains 
  the  $G$-intertwining    of $\theta$.

{\bf (vi) }  $\lambda$  is
any irreducible $C$-representation  of $J$ of restriction  to $J^0$    of the form
  $ \rho\otimes _C \kappa$  where $\kappa$ is  a preferred extension of $\eta$ to $J^0$, and $\rho$  is a $C$-representation of $J^0$ trivial on $J^1$
  inflated from a  cuspidal $C$-representation $\rho^0$  of the finite reductive group $J^0/J^1$.

\begin{remark}\label{rem:vii}  Sometimes one knows that  
 ``intertwining implies conjugacy''   in the sense that two very special characters appearing in the same  irreducible  cuspidal $C$-representation of $G$ are in fact $G$-conjugate.
  Note that if intertwining implies conjugacy and   $\pi$ is  induced from a cuspidal type $(J',\lambda')$  which contains a very special character $\theta'$, then if  $\pi$ contains the very special character $\theta$, the characters $\theta$ and $\theta'$ are $G$-conjugate, so we may assume $J=J', \theta=\theta'$ and $\pi$ is induced from $(J,\lambda')$ which is a cuspidal type as described in (i) to (vi).
  
  Sometimes one knows a weaker condition:

{\bf  (vii)}  If an irreducible  cuspidal $C$-representation $\pi$ of $G$ contains a very special character $\theta$, then it is induced from a cuspidal type $(J,\lambda)$ as described in (i) to (vi).

\end{remark}

The level $0$ case (Definition \ref {def:level0}) enters that framework, if we decide that $H^1=J^1=G_{x,0+}$ and that $\theta$ is the trivial character of $G_{x,0+}$, so that $\eta=\theta$, $J=G_x $ is  the $G$-normalizer of $\theta$; the trivial character of $J^0=G_{x,0}$ is a preferred extension  extending to $J$. With these definitions, a level $0$ cuspidal $C$-type $(J,\lambda) $ of $G$ satisfies the properties (i) to (vii).

 \bigskip   For later use, it is worth elaborating on conditions (v) and (vi).
 What  means intertwining was recalled in \S \ref{ss:2.4} before Remark \ref{re:compactind}.
    There are usually several preferred extensions $\kappa$ of $\eta$ to $J^0$. If $\kappa$ is one and $\chi:J^0\to C^*$ a character trivial on $J^1$ of order prime to $p$, then $\chi \kappa$ is also a preferred extension, because $\chi$ is trivial on the pro-$p$ Sylow subgroups of $J^0$. There is a converse.
    
      \begin{lemma}\label{le:kappa} If $\kappa$ is a preferred extension  of $\eta$ to $J^0$, the other preferred extensions  of $\eta$ to $J^0$ are $\chi \kappa$ where  $\chi:J^0\to C^*$ is a character trivial on $J^1$ and of order prime to $p$.
  \end{lemma}

\begin{proof} An   arbitrary  extension of $\eta\in \Irr_C(J^1)$ to $J^0$ has the form $\chi \kappa$ where $\chi:J^0\to C^*$ a character trivial on $J^1$, which we can identify with a character of the finite reductive group $G'_{y,0}/G'_{0,y +}$ for $y$ in (ii).   If the order of $\chi$ is divisible by $p$, $\chi$ is not trivial on the pro-$p$ Sylow subgroups of $J^0$, and we show that condition (ii) implies that $\chi \kappa$ is  not a preferred extension of $\eta$.
 Indeed, the vertex  $y$ of $\mathcal B (G'^{ad})$ lies  in the apartment associated to  a maximal split subtorus $T'\subset G'$ 
  in   $\mathcal B(G'^{ad})$. Let $B'\subset G'$  be a minimal parabolic subgroup  containing $T'$ with unipotent radical $U'$. 
Then $(U'\cap G'_{y,0})/(U'\cap G'_{y,0+}) $ is a pro-$p$ Sylow subgroup of $G'_{y,0}/G'_{y,0+}$,  $S=(U'\cap G'_{y,0})G'_{y,0+}$ is a pro-$p$ Sylow subgroup of $G'_{y,0}$, and  $SJ^1$ is a pro-$p$ Sylow subgroup of $J^0$. 
 By condition (ii), $J^1\cap G'= G'_{y,0+}$ so  $SJ^1\cap G'=S$. 
 There exists  $t'\in T'$ such that $ {t'}^{-1}(U' \cap G'_{y,0})t' \subset G'_{y,0+}$. 
 By condition (ii), $t'$ intertwines  $\theta$.
      
These properties imply  that,  for any  preferred extension  $\kappa_1$ of $\eta$, each $ t' $-intertwiner  $\Phi$ of $\theta$ is also  a $ t' $-intertwiner of the restriction of $\kappa_1$  to $S$.
 Indeed,   $\Phi$ is a $t'$-intertwiner of the restriction of $\kappa_1$ to the pro-$p$ Sylow subgroup $ SJ^1$ of $J^0$. Since 
  $SJ^1\cap G' =S$,  $\Phi$ is also a $t'$-intertwiner of the restriction of $\kappa_1$ to $ S$,  so  $\Phi \kappa_1(t'^{-1}xt')=\kappa_1(x) \Phi$ for all $x\in t' S t'^{-1} \cap S$.   In particular,    for $x\in  (U'\cap G'_{y,0})\subset  t' S t'^{-1} \cap S$.

  If $\chi \kappa$ is   a preferred extension,  as $\chi$ is trivial on $t'^{-1}(U'\cap G'_{y,0}) t' \subset  J^1$, we deduce  $\Phi \kappa (t'^{-1}xt')=\chi (x) \kappa (x) \Phi$   for all  $x  \in (U'\cap G'_{y,0})$; as $\kappa$ is a preferred extension, we have also $\Phi \kappa (t'^{-1}xt')= \kappa (x) \Phi$ hence  $\Phi = \chi (x) \Phi$ for all  $x  \in (U'\cap G'_{y,0})$.
       If the order of $\chi$ is divisible by $p$, then $\chi$ is not trivial on $U'\cap G'_{y,0}$; we get  $\Phi=0$, a contradiction because $t'$ interwines $\theta$, which shows that $\chi \kappa$ cannot be a preferred extension.
 \end{proof}
\begin{remark}   Most of the time all characters of  $ G'_{y,0}/G'_{y,0 +}$ have order prime to $p$ (see the list in Digne-Michel \cite{DM91}), but there are some exceptions, for example  a non-trivial complex character of  $SL(2, \mathbb F_2)$  has order $2$  (the signature,  after identifying $SL(2, \mathbb F_2)$  with  the group of  permutations on $3$ elements), and  the two non-trivial complex characters of $SL(2, \mathbb F_3)$  have order $3$.
 \end{remark}

 \begin{remark}\label{re:vi} Applying Clifford theory  to $J$ and its normal subgroup $ J^0$,  the restriction   of $\lambda$ to $J^0$ is semi-simple,  by the condition (vi) its  irreducible components are the $J$-conjugates of $\tau \otimes_C\kappa$  where 
 $\tau\in \Irr_C(J^0)$ is trivial on $J^1$,  and choosing one  $\lambda$ is obtained from the $\tau \otimes_C\kappa$-isotypic component $\lambda_\tau$ of $\lambda|_{J^0}$, by induction   to $J$  from the $J$-stabilizer $J_\tau$ of the isomorphism class of $\tau \otimes_C\kappa$.

    Let $j\in J$. The $j$-conjugate $\kappa^j $ of $\kappa$ is again a preferred extension of $\eta$, because
  $J$ normalizes $J^0$ hence permutes its pro-$p$ Sylow subgroups, and $J$ normalizes $\theta$ and $\eta$ by the conditions (ii) and  (iv).   By Lemma \ref{le:kappa},  $\kappa^j=\chi \kappa$ where $\chi:J^0\to C^*$ is a character trivial on $J^1$ of order prime to $p$. The $j$-conjugate of $\tau \otimes_C\kappa$  is $ \tau^j\otimes_C\kappa^j= \chi \tau^j\otimes_C\kappa$, and 
$ \chi \tau^j $ is, as $\rho$, an irreducible representation of $J^0$ trivial on $J^1$ and cuspidal as a representation of $J^0/J^1$. We conclude that (vi) is independent of the choice of the preferred extension, that all irreducible components of $\lambda|_{J^0}$ have the form prescribed in (vi) with the same $\kappa$. The condition on $\lambda|_{J^0}$ in  (vi)  is equivalent to:
 
  $\lambda|_{J^0}= \rho \otimes_C \kappa$  where $\kappa$ is  a preferred extension of $\eta$ to $J^0$, and $\rho$  is a $C$-representation of $J^0$
  inflated from a cuspidal $C$-representation $\lambda^0$ of  $J^0/J^1$.
  
 \end{remark}

Very often the construction of preferred extensions gives one which is normalized by $J$. Often too, there is an extension $\tilde \kappa $ of $\kappa$ to $J$ and  then $\lambda= \tilde \rho  \otimes_C \tilde \kappa $ where $\tilde \rho $ is an irreducible  representation of $J$ trivial on $J^1$ containing $\rho$ on restriction to $J^0$ as in the level $0$ case.

\bigskip We generalise now Proposition \ref{prop:iso} for level $0$ cuspidal $C$-types to  level $>0$ cuspidal $C$-types as above

 \begin{proposition}\label{prop:iso2}   In the setting  described with the  conditions  (i) to (vi), 
   the space of vectors in $\ind_J^G \lambda$ transforming by $\theta$ under right translation by $H^1$ is made out of the functions
with support in $J$; in particular it affords the representation $\lambda$ of $J$.
\end{proposition}
\begin{proof} The restriction  of $\ind_J^G \lambda$ to $H^1$ splits as a direct sum $\oplus_{Jg H^1 }\ind_J^{JgH^1} \lambda$. The subspace of functions with support in $J$, as a representation of $J$,  is isomorphic to $\lambda$. By (vi), the restriction of    $\lambda$ to $H^1$ is  $\theta$-isotypic.
  It is enough to show that for $g\in G\setminus  J$, a 
function in $\ind_J^G \lambda $ with support in  $Jg H^1$  and  transforming by $\theta$ under right translation by $H^1$,   is $0$. 
 Saying that there exists a non-zero function in $\ind_J^G \lambda $ with support in  $Jg H^1$  transforming by $\theta$ under $H^1$ is saying that $g$ interwines $\lambda$ with $\theta$; since $\lambda|_{H^1}$ is $\theta$-isotypic, if $g$ interwines $\lambda$ with $\theta$ then  $g$ interwines $\theta$ so belongs to $JG'J$ by  (ii); because $J$ normalizes $\theta$ we can assume $g\in G'$. Since $g$ interwines $\lambda$ with $\theta$, it intertwines $\lambda|_{J^0}$ with $\theta$; it also intertwines $\kappa$  as $g^{-1}Jg\cap H^1$ is a pro-$p$ group and 
  $\kappa$ is a preferred extension  by (v).
  Reasoning as in (\cite{BK93} Proposition 5.3.2), we see that $g$ intertwines $\rho$ with the trivial representation of $J^1$. Restricting to $G'\cap J^0=G'_{y,0}$, we get that $g$ intertwines $\rho$, seen as a  representation of $G'_{y,0}$, with  the trivial representation $1$ of $G'_{y,0+}$. We are now in a level $0$ situation.  So $g\in G'$  must satisfy $\Hom_{G'_{y,0} \cap g^{-1}G'_{y,0+}g}(\rho, 1)\neq 0$, which means that the  irreducible  cuspidal  representation $\rho $ of $G'_{y,0}$ has a non-zero vector fixed by $G'_{y,0} \cap G'_{g^{-1}y,0+}$. By Lemma \ref{le:0}, $g'$   has to be in the $G'$-stabilizer $G'_{y}$ of $y$, hence in $J$.
 \end{proof}
 \begin{corollary} \label{cor:2}  The conditions (i) to (vi) imply that $\ind_J^G\lambda$ is irreducible and $\End_{C[G]}(\pi)=\End_{C[J]}(\lambda)$. \end{corollary}
\begin{proof} Apply the proposition  as in  Corollary \ref{cor:0}.
\end{proof}

 \begin{proposition} \label{prop:cnirr} The conditions (i) to (v) imply that $\ind_{J^0}^G (\rho \otimes_C \kappa)$ is cuspidal when 
  $\kappa$ is  a preferred extension of $\eta$ to $J^0$, and $\rho$  is a $C$-representation of $J^0$ trivial on $J^1$
  inflated from a  cuspidal $C$-representation $\rho^0$  of the finite reductive group $J^0/J^1$.
\end{proposition}
 \begin{proof} Put $V= \rho \otimes_C \kappa$ and $V^\sharp$ for the inflation of $V$ of $J^0Z^\sharp$ (as $J^0\cap  Z^\sharp$ is trivial).
 The representation $\ind_{J^0Z^\sharp}^G V^\sharp$ has finite length and its irreducible subquotients are of the form  $\ind_{J}^G\lambda$ where $\lambda$ satisfies the condition (vi), and  $\ind_{J}^G\lambda$ is irreducible (Corollary \ref{cor:2}) and cuspidal (a coefficient of $\lambda$ is a coefficient of $\ind_{J}^G\lambda$ with $Z$-compact support).
  Hence  $\ind_{J^0Z^\sharp}^G V^\sharp$ is cuspidal.
 
To show the cuspidality of  $\ind_{J^0}^GV$, we use the criterion (Remark \ref{re:Nv}):  for  any $f\in \ind_{J^0}^GV$ and any arbitrary proper parabolic subgroup $P=MN$ of $G$, there exists  a compact  open subgroup $N_f$ of $N$ such that $e_{N_f}(f)=0$ where $e_{N_f}$ is the projection on the $N_f$-invariants.

We have $\ind_{J^0}^GV= \ind_{J^0Z^\sharp}^G (V \otimes_C C[Z^\sharp])$ and the linear map $C[Z^\sharp])\to C$ sending $z\in Z^\sharp$ to $1$, induces a surjective $C[G]$-map  $\ind_{J^0Z^\sharp}^G (V \otimes_C C[Z^\sharp])\to \ind_{J^0}^GV^\sharp$ sending $f$ to  $f^\sharp=\sum_{z\in Z^\sharp}zf$ as $ (zf)(g)= z (f(g))$ for $g\in G$.   By the criterium of cuspidality recalled above, there exists a compact  open subgroup $N_{f}$ of $N$ such that $e_{N_f }(f^\sharp)=0$.  As the  intersection $J^0Z^\sharp\cap N$ is trivial we have also $e_{N_f}(f ) =0$. 
  \end{proof} 
  
  \begin{corollary}  \label{cor:cnirr}The conditions (i) to (v) imply that $\ind_{J}^G \lambda$ is cuspidal when $\lambda$ satisfies the condition (vi) without  the irreducibility.
  \end{corollary}
  \begin{proof}  By adjunction $\lambda$ embeds in $\ind_{J^0}^J(\lambda|_{J^0})$ and by exactness   $\ind_{J}^G \lambda$ embeds in  $\ind_{J^0}^G(\lambda|_{J^0})$.
  By  Proposition \ref{prop:cnirr}, $\ind_{J^0}^G(\lambda|_{J^0})$ is cuspidal. A subrepresentation of a cuspidal representation is cuspidal hence $\ind_{J}^G \lambda$ is cuspidal.
  \end{proof}
 
 In that setting, $\Aut(C)$-stability  can be established as follows. Let $(J,\lambda)$ be a   cuspidal $C$-type as above and let $\sigma\in \Aut(C)$. To prove that 
 $(J, \sigma(\lambda))$ is also a  cuspidal $C$-type as above, there are two issues.
 
 (a) If $\theta$ is a  {\it very special} character, then $\sigma(\theta)$ is also  {\it very special}. 
  
 If (a) is true, certainly $\sigma(\eta)$ is an irreducible representation of $J^1$  restricting on $H^1$ to a multiple of $\sigma(\theta)$, and we need:
 
 (b) if $\kappa$ is a {\it preferred} extension to $J^0$ of $\eta$, then $\sigma (\kappa)$ is a  {\it preferred} extension to $J^0$ of  $\sigma(\eta)$.
 
Clearly  $\sigma(\rho \otimes_C \kappa )\simeq \sigma(\rho)\otimes_C \sigma(\kappa)$ and $\sigma(\rho)$ is, as $\rho$, an irreducible $C$-representation of  $J^0$ trivial on $J^1$  with  restriction to $J^0$ inflated from a cuspidal $C$-representation $\sigma(\lambda^0)$ of $J^0/J^1$. If (a) and (b) are true, $(J, \sigma(\lambda))$ is a   cuspidal $C$-type as desired. Assuming (a),  let us prove (b): indeed $\sigma(\kappa)$ extends $\sigma(\eta)$ and since $\sigma$ preserves intertwining, (b) comes from condition (v). Our task in the examples of  \S \ref{ss:3.5} and \S \ref{ss:3.6} below will be to verify property (a).
 
 Note that  underlying the construction of the {\it very special} characters $\theta$   is the choice
 of an additive character $\psi:F\to C^*$, assumed to be trivial on  $P_F$ but not on $O_F$.    
Applying $\sigma\in \Aut(C)$  transforms $\psi$ to the character 
 $\psi^{\xi}:x\mapsto \psi(\xi x)$  for some $\xi=\xi_\sigma\in O_F^*$. Actually, if $a$ is the characteristic of $F$, then   $\xi_\sigma$ is in $\mathbb Z_p^*\subset O_F^*$ if $a=0$ and in    $  \mathbb F_p  ^*\subset O_F^*$   if $a=p$. In fact, as we shall see, changing $\psi$ to  $\psi^{\xi}$
 for $\xi \in O_F^*$ does not change the set of types constructed inducing cuspidal irreducible representations, and property (a) holds.

 \subsection{Supercuspidality and types}\label{ss:scpositive}

Let $(J,\lambda)$  be a cuspidal $C$-type in $G$ as in  \S \ref{ss:3.4}  satisfying the properties (i) to (vi).  Put $\pi= \ind_J^G\lambda$.  
 
  We are in the situation where $C$ is algebraically closed, $\lambda|_{J^1}$ is $\eta$-isotypic for a representation  $\eta\in \Irr_C(J^1)$ which is normalized by $J$ and extends to a  representation $\kappa\in  \Irr_C(J^0)$; moreover   for any $h\in J$, the conjugate of $\kappa$ by $h$ is isomorphic to $\chi_h \kappa$ where $\chi_h$ is  a $C$-character of $J^0$ trivial on $J^1$ of order prime to $p$.   We   choose a preferred extension $\kappa$ of $\eta$ (Remark \ref{re:vi}, Lemma \ref{le:kappa}). 
 We have  
\begin{equation}\label{eq:rJ0}\text{$\lambda|_{J^0}=\rho \otimes_C \kappa$ where 
 $\rho$   is   cuspidal as a representation  of $J^0/J^1$.}
 \end{equation}  The definition of $\rho$  depends on the choice of the preferred extension $\kappa$ of $\eta$. 
The other preferred extensions have the form $\chi \kappa$ where $\chi$ is a character of $J^0$ trivial on $J^1$ of order prime to $p$ by the discussion in  \S \ref{ss:3.4} before Lemma \ref{le:kappa}, so that $ \rho \otimes_C\kappa =\chi^{-1}  \rho \otimes_C\chi \kappa$. Therefore, another choice of $\kappa$ gives
 $\rho$  twisted by a character of order prime to $p$. By Clifford's theory, 
$\lambda|_{J^0} $   is semi-simple of finite length. The irreducible components are $J$-conjugate of the form $ \sigma \otimes \kappa$ where $\sigma$ is  an irreducible  component of $\rho$. Let  $I_{\rho}$  be the injective hull of $\rho$  in $\Mod_C(J^0)$. The following properties are equivalent:

(i)  Some irreducible component of $\rho$    is   supercuspidal as a representation  of $J^0/J^1$. 

 (ii)  $I_{\rho}$  is cuspidal 
 as a representation of $J^0/J^1$ (Lemma \ref{le:Hiss}).
 
 \begin{definition}\label{def:scpositive}   $(J,\lambda)$ is called supercuspidal if 
 the properties   (i), (ii) are satisfied.
  \end{definition}

The definition  does not depend on the choice of the preferred extension $\kappa$ in  of $\eta$.  
 In level $0$ where 
 $\eta$ and $\kappa$ are trivial,  it  coincides with Definition \ref{def:level0}, and 
$(J,\lambda)$ is supercuspidal if and only if   $\pi $  is supercuspidal (Theorem \ref{thm:sct0}). In positive level, our goal is to   prove:
 
  \begin{theorem}\label{thm:sct0positive}  If $(J,\lambda)$ is supercuspidal then   $\pi $  is supercuspidal. The converse is true if $(G,C)$ satisfies the second adjunction  and $(J,\lambda) $ satisfies the property (vii) of \S \ref{ss:3.4}.
\end{theorem}  

The proof is parallel to the previous ones in level $0$ if $(G,C)$ satisfies the second adjunction, with one extra complication coming from $\eta $, and a slight simplification due to the fact that $C$ being algebraically closed,   $\lambda $ and $\pi$ are  $\omega$-isotypic for some $\omega\in \Irr_C(Z^\sharp)$ of dimension $1$, which  can be considered as a $C$-character of $Z^\sharp$.

\bigskip The category  $\Mod_C(J^0, \eta)$  of $C$-representations of $J^0$ which are $\eta$-isotypic on restriction to $J^1$,  is a  direct factor of $\Mod_C(J^0)$. When  $\eta=1_{J^1}$ is  trivial, $\Mod_C(J^0, 1_{J^1})$ is equivalent to 
 $\Mod_C(J^0/J^1)$. We have similar results when $J^0$ is replaced by a subgroup of $G$ containing $J^1$ as a normal subgroup.
 Since $\kappa$ is an extension of $\eta$ to $J^0$, the functor 
 \begin{equation} \label{eq:ka} W \mapsto W\otimes_C\kappa: \Mod_C(J^0, 1_{J^1}) \to \Mod_C(J^0, \eta)\end{equation} is an equivalence (depending on the choice of $\kappa$),  a reverse equivalence being given by $  \Hom_{C[J^1]}(\kappa, -)$ with the natural action of $J^0$.  
The  functor $V\mapsto V_\eta : \Mod_C(G )\to \Mod_C(J, \eta)$ sending  a representation to its $\eta$-isotypic part on restriction to $J^1$, with the natural action of $J$, is exact and respects injectives and projectives. And also the functor
$$e_\kappa:\Mod_C(J , \omega)\to  \Mod_C(J^0, 1_{J^1}) \quad V\mapsto   \Hom_{J^1}(\kappa, V_\eta),$$ 
where $\omega$ is the $C$-character of $Z^\sharp$ such that $\lambda$ is $\omega$-isotypic. In level $0$ where $\eta $ and $\kappa $  are  trivial, $V_\eta=e_{\kappa}(V)=V^{J^1}$.  We have $\pi_\eta= \lambda $ (Proposition \ref{prop:iso2}) and $e_\kappa(\pi)=\rho$ (formula \eqref{eq:rJ0}).   Let $I_{\lambda,\omega}$ be an injective hull of $\lambda$ in $\Mod_C(J,\omega)$.    
  
 \begin{proposition} \label{prop:Ilambdapositive} $I_{\lambda,\omega}$ is finite dimensional, projective, indecomposable with socle and  cosocle  isomorphic to $\lambda$.
  
  $(J,\lambda)$ is supercuspidal if and only if $e_\kappa(I_{\lambda,\omega})$ is cuspidal as a representation of $J^0/J^1$.
 \end{proposition}
 
 That  corresponds to Proposition \ref{prop:Ilambda} in level $0$.
 \begin{proof}   The  kernel $\Ker \theta $ of the very special character $\theta$ of $H^1$ is a normal open pro-$p$ subgroup of $J$ and $\eta$ is trivial on $\Ker \theta $, by the properties (i) and (iii)  in \S \ref{ss:3.4}.
 We put $H=J/\Ker \theta $. As $Z^\sharp\cap \Ker \theta $ is trivial, $Z^\sharp$ identifies with a subgroup $Y$ of $H$, $\omega$ with $\zeta\in \Irr_C(Y)$,  $\lambda$ inflates $\tau \in \Mod_C(H,\zeta)$ and  $I_{\lambda, \omega} $ inflates 
 an injective hull  $I_{\tau, \zeta}$ of $\tau$ in  $\Mod_C(H,\zeta)$ (see the example \ref{ss:injpro} c)). The first assertion of the proposition follows from Proposition \ref{prop:H} applied to $(H,Y,\zeta, \tau, I_{\tau, \zeta})$. 
 By the equivalence \eqref{eq:ka}, $ \rho \otimes_C\kappa$ where  $\rho=e_{\kappa}(\pi)$ is an irreducible quotient of $\lambda|_{J^0}$.   By Proposition  \ref{prop:H0} applied to  $H^0=J^0/\Ker \theta$, the restriction of $I_{\lambda,\omega}$ to $J^0$ is a sum of $J$-conjugates of $ I_\rho \otimes_C\kappa$ where 
 $I_\rho$ is an  injective hull of $\rho $ in $\Mod_C(J^0, 1_{J^1})$. A $J$-conjugate of 
$I_\rho \otimes_C\kappa$   is of the form $\chi I_\rho'\otimes_C\kappa $ for a $J$-conjugate $I_\rho' $ of $ I_\rho$ and a character $\chi$ of $J^0$ trivial on $J^1$ (Remark \ref{re:vi}). Hence
 $e_\kappa(I_{\lambda,\omega})$ is a sum of $\chi  I'_\rho$. It is cuspidal as a representation of $J^0/J^1$ if and only if $I_\rho$ if and only if   $(J,\lambda)$ is supercuspidal (Definition \ref{def:scpositive}).
  \end{proof}

\begin{theorem} \label{thm:sc01positive}  
1) Assume   $(J, \lambda)$  supercuspidal. Then $\pi=\ind_J^G\lambda $  is supercuspidal.  

Moreover $\ind_{J}^{G}I_{\lambda, \omega}$
is an injective hull  $I_{\pi,\omega}$ of $\pi $ in $\Mod_C(G, \omega)$.  It is   cuspidal projective indecomposable  with socle and  cosocle  isomorphic to $\pi$.   The lattices of subrepresentations  of $I_{\lambda, \omega}$  and  of  $I_{\pi,\omega}$
  are isomorphic  by the map    $W \mapsto \ind_J^GW$ (which is equal to $\Ind_J^G\nu$), with inverse
  $V\mapsto V_\eta $.
  
  2) Assume $\pi$ supercuspidal.  Then $(J, \lambda)$ is  supercuspidal if $(G,C)$ satisfies the second adjunction  and $(J,\lambda) $ satisfies the property (vii) of \S \ref{ss:3.4}.
     \end{theorem}
     
  That   is a stronger form of  Theorem \ref{thm:sct0} which 
corresponds to Theorems \ref{thm:sc01} and  \ref{thm:sc0}  in level $0$.
        
\begin{proof} 
 The proof of second part of 1)  is the same as in Theorem \ref{thm:sc01}. So $I_{\pi,\omega}$ is cuspidal of finite length, hence is admissible and consequently right cuspidal (Proposition \ref{pro:cusp}).
 When the second adjunction holds true,  the proposition \ref{prop:scI} implies that 
 $\pi$ is supercuspidal. The proof of 2) which assumes the second adjunction follows   the same method as for Theorem \ref{thm:sc0}  replacing Corollary \ref{cor:level0} by the property (vii).

We show now  that $(J,\lambda)$ supercuspidal implies $\pi$ supercuspidal without assuming the second adjunction.  
We recall  the injective hull $I_{\lambda|_{J^0}}$ of $\lambda|_{J^0}$ in $\Mod_C(J^0)$ and  we consider the representation  $V= \ind_{J^0}^JI_{\lambda|_{J^0}}$ of $J$.  The restriction of $V$ to $J^0$ is $\tau \otimes_C \kappa$ where $\tau$ is a  finite sum of $ \chi I_\rho' $ where 
$\chi$ is a character of $J^0 $ trivial on $J^1$ and $I_\rho'$ a $J$-conjugate of $I_\rho$ (Remark \ref{re:vi}).  The representation   $\ind_J^G V= \ind_{J^0}^G I_{\lambda|_{J^0}}$ of $G$  is projective with quotient $\pi=\ind_J^G\lambda$, as $I_{\lambda|_{J^0}}$  is projective of quotient $\lambda|_{J^0}$ and
  $\lambda$ is a quotient of $V=\ind_{J^0}^J\lambda|_{J^0}$ by adjunction. If $(J,\lambda)$ is supercuspidal, then $I_\rho$ is cuspidal as a representation of $J^0/J^1$, hence also  $\tau$, and Proposition \ref{prop:cnirr} tells us that $\ind_J^G V$  is cuspidal - so   $\pi$ is the quotient of the  projective cuspidal representation $\ind_J^GV$, hence is supercuspidal  (Lemma \ref{ex:sc}).
    \end{proof}

  In the setting  of  \S \ref{ss:3.4}, the field $C$ is algebraically closed. 
 When $C$ is not algebraically closed,   
 a cuspidal $C^a$-type $(J,\lambda^a)$ of  $G$ defines  by restriction to $C$ a cuspidal $C$-type $(J,\lambda)$ of $G$ such that 
  $\lambda^a$  seen as a $C$-representation is $\lambda$-isotypic.

\begin{definition}\label{def:scpositivenac}  A cuspidal $C$-type $(J,\lambda)$ of  $G$ 
arising by restriction to $C$  of    a cuspidal $C^a$-type $(J,\lambda^a)$   of  $G$    in the setting of  \S \ref{ss:3.4}  with the properties (i) to (vi) satisfied, 
 is called supercuspidal if $(J,\lambda^a)$ is supercuspidal (Definition \ref{def:scpositive}).
\end{definition}

That definition ensures via Theorem \ref{thm:sct0} and Proposition \ref{prop:c} that    $\pi =\ind_J^G\lambda$  is supercuspidal if  $(J,\lambda)$ is supercuspidal, and that  the converse is true if $(G,C)$ satisfies the second adjunction  and $(J,\lambda^a) $ satisfies also the property (vii) of \S \ref{ss:3.4}.

\begin{remark} The definition  is compatible in level $0$ with Definition \ref{def:sc}  which does not suppose $C$ algebraically closed (Lemma \ref{le:scalar0}). 

The definition   does not depend on the choice of $\lambda^a$ because another irreducible component is a conjugate $\sigma(\lambda^a)$ of $\lambda^a$ by some $\sigma\in \Aut_C(C^a)$. We have $\lambda^a|_{J^0}=\rho^a\otimes \kappa^a$,  $\sigma( \rho^a \otimes_C\kappa^a)=\sigma( \rho^a) \otimes_C\sigma(\kappa^a)$, 
$\sigma(\kappa^a)$ is a preferred extension of $\sigma(\eta^a)$, and an irreducible component of $\sigma( \rho^a)$ is supercuspidal if and only if an irreducible component of $ \rho^a $ is.
 \end{remark}

\subsection{Types \`a la Bushnell-Kutzko}\label{ss:3.5} Let us review the types constructed with the techniques of Bushnell-Kutzko.   The reader needs familiarity with the references, as we only indicate why properties (i) to (vii) are true and how to  establish $\Aut(C)$-stability.  

\subsubsection{$GL(N, F)$} \label{sss:GL}We start with $GL (N, F)$ and $C=\mathbb C$ treated in \cite{BK93}.

The basic concepts are those of simple stratum and simple character; the maximal simple characters in \cite{BK93} are the {\it very special} characters here. 

We let $V$ be an $F$-vector space of dimension $N$ (e.g. $V=F^N$) so that $G=\Aut_F(V)$ is isomorphic to $GL(N,F)$. 
A {\it simple stratum} $(\mathfrak A, n,r,\beta)$ is made out of 
an hereditary $O_F$-order  $\mathfrak A$ in $\End_F(V)$, integers $n\geq r  \geq 0$, and an element $\beta\in G$ normalizing $\mathfrak A$ such that $E=F(\beta)$ is a field and  satisfying the conditions of (\cite{BK93},1.5.5). Those conditions involve only the conjugation of $\beta$ on $\End_F(V)$ and on 
 $\mathfrak A$, so it is straightforward that for $a\in O_F^*$,  $(\mathfrak A, n,r,a\beta)$
is again a simple stratum, 
moreover the groups $J, J^0, J^1,H^1 $ attached to the two strata are the same (\cite{BK93}, 3.1.14) : we write $J^0, J^1, H^1$ for Bushnell-Kutzko's $J(\beta,\mathfrak A), J^1(\beta,\mathfrak A), H^1(\beta,\mathfrak A)$, $G'= B^*$ where $B$ is the centralizer of $\beta$ in $\End_F(V)$.   The  $O_E$-hereditary order $\mathfrak  B= \mathfrak  A \cap B$ in $B$ corresponds to a point   in the  Bruhat-Tits building $\mathcal B(G')$ of $G'$. 
Write $y$  for the image   of this  point   in   $\mathcal B(G'^{ad})$. 
We have  (\cite{D09} 7.1 p.313)  $G'_{y,0}= \mathfrak B ^*$ and $G'_y$ is  the normalizer of $\mathfrak B$ in $B$. 
We put $J=G'_y J^1$ and we have $J^0=G_{y,0}J^1 $ (\cite{BK93}, 3.1.15)  
 so that the inclusion $B\subset \End_F(V)$ induces an isomorphism $G'_y/ G'_{y,0}\to J/J^1$, as demanded by property (ii). 

To the simple stratum $(\mathfrak A, n,r,\beta)$ is attached the set  $\mathfrak C(\mathfrak A, r,\beta, \psi) $ of {\it simple characters} (\cite{BK93},3.2.1 and 3.2.3) -- we add the underlying character $\psi$ in the notation  of  \cite{BK93}.

Following the definitions, one gets that $\mathfrak C(\mathfrak A,  r,\beta,\psi )=\mathfrak C(\mathfrak A,  r,a\beta, \psi)  $ for $a\in O_F^*$, and that for $\sigma \in \Aut(C)$, the map  $\theta\mapsto \sigma(\theta)$ yields a bijection 
$\mathfrak C(\mathfrak A,  r,\beta, \psi )\to \mathfrak C(\mathfrak A,  r,\beta, \sigma(\psi))= \mathfrak C(\mathfrak A,  r,\xi_{\sigma}\beta, \pi)$. In particular property (a) of \S \ref{ss:3.4} is satisfied.
Only $r=0$ is used in the sequel, so we suppress it from the notation. The simple characters occuring in the cuspidal representations are the {\it maximal} ones, meaning that $\mathfrak B$ is a maximal  $O_E$-order in $B$ (\cite{BK93} 6.2.1, \cite{BH13} Corollary 1), corresponding to the case where $y$  is a vertex to $\mathfrak B (G'^{ad})$.
   
If $\theta\in \mathfrak C(\mathfrak A,   \beta,\psi )$ is a simple character,  its $G$-normalizer is $J$ (\cite{BK93} 3.3.17) and 
  its  $G$-intertwining   is $JG'J$ (\cite{BK93} 3.3.2). The non-degenerate alternating bilinear form on $J^1/H^1$ is in (\cite{BK93} 3.4.1), the  existence and uniqueness of $\eta$ are in  (\cite{BK93} 5.1.1) and the $G$-intertwining 
 of $\eta$ is  in (\cite{BK93} 5.1.8). The conditions  (i), (ii), (iii), (iv)   of  \S \ref{ss:3.4} are satisfied.

There are $\beta$-extensions of $\eta$ to $J^0$  (\cite{BK93} 5.2.1). A $\beta$-extension is an extension which is intertwined by $G'$, or equivalently by $J^0G'J^0$, or equivalently with the same $G$-intertwining than $\eta$ because $J^0G'J^0= JG'J$. In particular it is a preferred extension and is normalized by $J$, giving property (v) of \S \ref{ss:3.4}. For a maximal simple character $\theta$, $J/J^0$ is cyclic, so a $\beta$-extension even extends to $J$. In any case if $\kappa$ is a $\beta$-extension and $\sigma \in \Aut (C)$, then $\sigma(\kappa)$ is also a $\beta$-extension.

The cuspidal types  $(J, \lambda)$  of Bushnell-Kutzko (\cite{BK93} 6.2) are obtained by the procedure of  \ref{ss:3.4}, property (v), starting from a maximal simple character $\theta$ and a $\beta$-extension $\kappa$ of $\eta=\eta_\theta$; in fact they  are such that $\lambda=\rho \otimes_C \tilde \kappa$ where  $\tilde \kappa$  is  an extension of $\kappa$ to $J$ and $\rho$ is a  representation of 
$J=G'_yJ^1$ trivial on $J^1=G'_{y,0+}$ with restriction of $J^0$ inflated from an irreducible cuspidal representation of $J^0/J_1\simeq G'_{y,0}/ G'_{y,0+}$.
The discussion in \ref{ss:3.4} shows that if one uses instead any preferred extension in lieu of $\kappa$, we get the same set of cuspidal types.  

Following the procedure indicated after Corollary \ref{cor:2}, we deduce that the set of types thus obtained satisfies $\Aut(C)$-stability.
 
  Exhaustion and unicity  for the set of cuspidal types obtained  by varying the maximal simple characters, and including level $0$, are given by (\cite{BK93} 8.4.1).
  Finally, intertwining implies conjugacy is true for maximal simple characters
\cite{BH13}, giving property (vii).

The second adjointness holds for $(G,C)$ \cite{D09}.

\subsubsection{$GL(m,D)$}  The case of inner forms of $GL_N$ (of course it includes  the split case, but uses \cite{BK93} as a basis) is due to Minguez, S\'echerre and Stevens \cite{SS08}, \cite{MS14}. In their setting, $D$ is a  central division algebra over $F$ of finite reduced degree $d$,  $V$ is a right $D$-vector space of finite dimension $m$, and $G=\Aut_D(V)$ is an inner form of $GL_N(F)$, $N=md$. When $m=1, G=D^*$ has semisimple rank $0$.

Cuspidal complex types were known before
 (\cite{Z92}, \cite{Br98}). Minguez, S\'echerre and Stevens, for a general algebraically closed field $C$ of characteristic $c\neq p$ construct a set of "cuspidal simple $C$-types", using simple strata and simple characters for non-level $0$ types, and they show  exhaustion and unicity (\cite{MS14} Theorem 3.11). Let us now give detail enough to verify  $\Aut(C)$-stability and properties (i) to (vii)  of \S  \ref{ss:3.4}.

 There is a notion (\cite{S04} Definition  2.3) of simple stratum  $(\mathfrak A, n,r,\beta)$ made out of 
an hereditary $O_D$-order  $\mathfrak A$ in $\End_D(V)$,  where $O_D$ is the ring of integers of $D$,  corresponding to a chain $\Lambda$ of $O_D$-lattices in $V$. We write indifferently  $\mathfrak A$ or $\Lambda$ in the notation of the simple stratum. To such a stratum is associated the centralizer $B$ of $\beta$ in $\End_D(V)$ and open subgroups $J, J^0,J^1,H^1$ all normal in $J$, see (\cite{S04} formula (65)) for $J^0,J^1,H^1$ whereas S\'echerre writes $J$ for $J^0$.
We write $J$ for the group written  $(\mathcal K (\mathcal A) \cap B)J^0$ in \cite{S04}; the normality property is Proposition 3.43 there, which also says that $J^1/H^1$ is a finite $p$-group. 
The chain $\Lambda$ is stable under $E^*$ where $E=F(\beta)$ and defines a point $y$ in $\mathcal B(G'^{ad})$ where $G'=B^*$. 
We have $J=G'_y J^1$ and $J^0=G'_{y ,0}J^1$. To get cuspidal types we have to restrict ot maximal simple strata (\cite{MS14} Proposition 3.6) which means that $y$ is a vertex. 
 
As in  \ref{sss:GL} we restrict to $r=0$ and suppress it from the notation. To a simple stratum  $(\Lambda, n , \beta)$ in $G$ is  attached a set $  \mathfrak C(\Lambda,   \beta,\psi )$ of  simple characters $\theta:H^1\to C^*$ (\cite{S04} Definition   3.45), obtained by a restriction process from simple characters constructed in \cite{BK93}. Following the definition, it is straightforward that for $a\in O_F^*$
 $(\Lambda, n , a\beta)$ is again a simple stratum with the same attached groups and 
$  \mathfrak C(\Lambda,   \beta,\psi ^a)=   \mathfrak C(\Lambda,  a \beta,\psi )$.
As in  \ref{sss:GL}, we verify that $\sigma \in \Aut(C)$ induces a bijection $\theta\mapsto \sigma(\theta): \mathfrak C(\Lambda,   \beta,\psi)\to  \mathfrak C(\Lambda,   \beta,\sigma(\psi))= \mathfrak C(\Lambda,  \xi_\sigma \beta,\psi)$.
The $G$-normalizer  of $\theta\in  \mathfrak C(\Lambda,   \beta,\psi )$ is   $J$  and its $G$-intertwining  is $JG'J=J^1G'J^1$  (\cite{S04} Theorem 3.50 and Rem. 3.51)   giving properties (i) and (ii). 

Existence and uniqueness of $\eta=\eta_\theta$  come from  (\cite{S04} Theorem 3.52) yielding property (iii). The intertwining property (iv)
of $\eta$ is (\cite{S05} Proposition 2.10).  

An extension of $\eta$ to $J^0$ is a $\beta$-extension if it is normalized by $B^*=G'$(\cite{S05}  \S 2.4). As the $G$-centralizers of $\beta$ and $\xi_\sigma \beta$   coincide for any $\sigma\in \Aut(C)$ we get $\sigma$-stability for the $\beta$-extensions. 
 The $\beta$-extensions are  preferred extensions in our sense (property (v)). As in \ref{sss:GL}, the other preferred   extensions are obtained by twisting by a character of order prime to $p$, and can equally be used to construct the cuspidal types.
 
 If $\kappa$ is a $\beta$-extension and $\rho$ as in property (vi), one can form $(J^0, \rho\otimes_C \kappa)$ and consider the $G$-normalizer $\tilde J$ of $(J^0, \rho\otimes_C \kappa)$ which is included in $J$  (as  $\rho\otimes_C \kappa$  is $\theta$-isotypic). The ``extended  maximal  cuspidal simple types''  of \cite{MS14} are the $(\tilde J, \tilde \lambda)$ where 
$\tilde \lambda$ is any extension of $ \rho\otimes_C \kappa$ to $\tilde J$ (such extensions exist as $J/J^0$ is cyclic (as before)). For such a pair $\ind_{\tilde J}^G\tilde \lambda$ is irreducible, and it is for that set of pairs  $(\tilde J, \tilde \lambda)$, including the level $0$ ones, that  (\cite{MS14} Theorem 3.11) gives exhaustion and unicity. From Proposition \ref{prop:II.55}, we deduce that the set of pairs $(J,\lambda)$ where $\lambda=\ind_{\tilde J}^J \tilde \lambda$, also satisfies exhaustion and unicity, and property (vi) is valid  by Clifford's theory. That set of cuspidal types is verified to be $\Aut(C)$-stable as in  \ref{sss:GL}, starting from the analysis above of the action of  $\Aut(C)$ on simple characters and $\beta$-extensions.

Finally, we mention that property (vii) comes from (\cite{MS14} Lemma 3.9 and 3.10).

\subsubsection{$SL_N$} Next we turn to $SL_N$ treated in \cite{BK94} for complex representations, and  extended  recently to positive characteristic coefficients by Cui \cite{C19}, \cite{C20}. She also treats Levi subgroups of $SL_N$. 
To keep with her notation, we let $M$ be a Levi subgroup of $GL_N$, and add the exponent $'$ to indicate the intersection with $G'=SL_N$.

To get cuspidal  simple  types for $M'$, one starts from such types for $M$; as $M$ is a product of $GL_{r_i}$, one can take those obtained in \ref{sss:GL}. If $(J,\lambda)$ is a cuspidal  simple type for $M$, one defines (\cite{C19} 3.44) its projective normalizer $\tilde J$; it contains $J$ as a finite index subgroup; the induced representation $\tilde \lambda=\ind_J^{\tilde J}\lambda$ is irreducible, its restriction to $\tilde J'$ is semisimple. Let $\mu$ be any   irreducible component of  $\tilde \lambda|_{\tilde J'}$ and $H=N_{M'}(\mu)$ its $M'$-normalizer. In fact, $H$ is the $M'$-intertwining  of $\mu$ and any irreducible representation   $\upsilon $ of $H$ containing $\mu$ on restriction to $\tilde J'$, induces irreducibly to a cuspidal irreducible representation $\ind_H^{M'}(\upsilon)$ of $M'$; moreover each cuspidal irreducible representation of $M'$ has this form for some choice of $(J,\lambda)$ and  $\upsilon$. A pair  $(H,\upsilon)$ obtained in this way is a cuspidal type  in $M'$, and the set  $\mathfrak X'$   of  such types satisfies exhaustion and is stable under conjugation by $M'$ (\cite{C19} Theorem 3.5.1); unicity is obtained for $\mathbb C$  in (\cite{BK94}, 5.3 Theorem),  and in general in
(\cite{C19} 3.5.6). 

Let us verify that $\mathfrak X'$ is $\Aut(C)$-stable. Start with a  cuspidal simple type $(J,\lambda)$ in $M$ and choose $\mu$ and  $\upsilon$ as above. Let $\sigma\in \Aut( C)$.  By (\cite{C19} 3.44),  the projective
normalizer $\tilde J$ of $(J, \lambda)$ is the same for 
$\sigma(\lambda) $ and clearly $\widetilde{\sigma(\lambda)}=\sigma(\tilde \lambda) $.  Then $\sigma(\mu)$ is  an irreducible component  of $\sigma(\tilde \lambda)|_{\tilde J'}$,  and $N_{M'}(\mu)=N_{M'}(\sigma(\mu))$; furthermore  $\sigma(\upsilon)$ is an irreducible representation of $H=N_{M'}(\sigma(\mu))$ containing $\sigma(\mu)$ on restriction to $\tilde J'$.  This shows that
  $(H,\sigma(\upsilon))$ belongs to $\mathfrak X'$, as desired.

\begin{remark}  That case of $SL_N$ does not immediately conform to the
common pattern described before. That question needs further study.\end{remark}

\subsubsection{Classical groups} \label{sss:classical} The case of classical groups, for any $C$  but only when {\it $p$ is odd} is due to Kurinczuk  and Stevens \cite{KS} (for $C=\mathbb C$ \cite{St08}). In this context,  $F/F_0$ is an extension of degree $1$ or $2$, $V$ is a finite dimensional $F$-vector space, $\epsilon\in \{1, -1\}$ and $h$ is a non-degenerate $\epsilon$-hermitian form on $V$ with respect to $F_0$.  The group  $G^+=\{ g\in \Aut_F(V) : h(gv,gw)=h(v,w) \ \text{ for all } v,w\in V\}$ is the group of   $F_0$-points of a unitary, symplectic or orthogonal group $\underline G^+$ and $U(V,h)$  the $F_0$-points of the connected component $\underline G$ of 
$\underline G^+$. In the unitary and symplectic case  $U(V,h)=G^+$, in the orthogonal case $F=F_0$ and $  \epsilon=1$, $U(V,h)$ is the special orthogonal group. One needs  {\it semisimple} strata $(\Lambda, n, \beta)$ in $ \End_F(V)$ where $\Lambda$ this time is a sequence of $O_F$-lattices in $V$  (\cite{St08} Definition  2.4, again only $r=0$ is used and suppressed  from the notation) and the corresponding sets  $\mathfrak C(\Lambda, \beta, \psi)$ of semisimple characters  in $\Aut_F(V)$ (\cite{St08} \S3.1). 
 
Now assume that $\psi=\psi_0 \circ \tr_{F/F_0}$ for some character $\psi_0:F_0\to C^*$, and write $x\mapsto \overline x $
for the involution on $ \End_F(V)$ associated to $h$,   $\iota$ for the involution $ x \mapsto  {\overline x}^{-1} $ on  $ \Aut_F(V)$, and $\Lambda^\flat$  is the lattice sequence in $ \End_F(V)$ dual to
$\Lambda$ with respect to $h$.
Then $\iota$ induces a bijection $\theta \mapsto \theta \circ \iota: \mathfrak C(\Lambda,  \beta, \psi) \to 
\mathfrak C(\Lambda^\flat,  - \overline \beta, \psi)$.  Clearly $\sigma (\theta \circ  \iota) = \sigma (\theta) \circ  \iota $ for $\sigma \in \Aut(C) $ and $\theta \in 
 \mathfrak C(\Lambda,  \beta, \psi)$.
When the stratum $(\Lambda, n, 
\beta)$ is self-dual (that is when $\Lambda^\flat$ is $\Lambda$ up to a translation
in indices, and $-\overline \beta=\beta$), the subgroups of  $ \Aut_F(V)$ attached to that stratum by Bushnell-Kutzko are invariant under $\iota$, 
and intersecting them with $G$ gives  subgroups $H^1= H^1(\beta, \mathfrak A)\cap G,  J^1= J^1(\beta, \mathfrak A)\cap G, J=J^0(\beta, \mathfrak A)\cap G$.  Then $J/J^1$ is the group of points of a possibly non-connected reductive group over $k_F$ and we define $J^0$ of the subgroup of $J$ such that $J^0/J^1$ is the connected component of $J/J^1$.
The set $ \mathfrak C(\Lambda,  \beta, \psi)$ of    semisimple   characters of $G$ is obtained by restricting to $H^1$ the $\iota$-invariant  semisimple   characters  of $ \Aut_F(V)$ corresponding to $(\Lambda,  \beta, \psi)$. It is clear that the semisimple  characters of $G$ satisfy:  for $a \in O_F^*$,  $\mathfrak C(\Lambda, \beta,\psi ^a)=\mathfrak C(\Lambda, a\beta, \psi) $  and   $\sigma \in \Aut(C)$ induces a  bijection $\theta\mapsto \sigma(\theta):  
\mathfrak C(\Lambda, \beta, \psi )\to \mathfrak C(\Lambda, \beta, \sigma(\psi))= \mathfrak C(\Lambda, \xi_{\sigma}\beta, \pi)$. 
Our very special characters are those semi-simple characters satisfying some maximality condition, and a procedure parallel
to the previous ones gives a set of cuspidal $C$-types in $G$ (\cite{KS} Theorem A (i)). 

Let us  verify   properties (i) to (vii).
First property (ii) is a special case of (\cite{KS} Theorem 3.10), and property (i) is an easy consequence.
Property (iii) is  (\cite{KS} Theorem 2.6) while (iv) is a special case of (\cite{KS}  Theorem 4.1).
In their \S 5, Kurinczuk and Stevens define $\beta$-extensions of $\eta=\eta_\theta $   for a semisimple
character $\theta$ in $\mathfrak C(\Lambda, \beta, \psi)$. A $\beta$-extension in \cite{KS} is a representation
of $J^+=J(\Lambda, \beta) \cap G^+$, and   for property (v) we need to verify that its restriction to $J^0$ (which \cite{KS} also calls a $\beta$-extension)
deserves to be called a preferred extension, at least when the maximality condition is satisfied.
  This comes from (\cite{St08} \S 4.1). Indeed the very special characters (that is the semisimple characters
occurring in cuspidal representations) are those attached to a skew semisimple stratum such that the 
associated order in $B$ is maximal. In that case Theorem 4.1 in \cite{St08} defines $\beta$-extensions to $J^+$,
and their construction and Corollary 3.11 in \cite{St08} show that they are exactly our preferred extensions.
(Note that \cite{St08} works over the complex numbers, but the constructions of  \S 3 and \S 4 are valid over our
field $C$).
Now the cuspidal types of \cite{KS} have the form $\rho \otimes \kappa$, where $\kappa$ is a $\beta$-extension of $\eta$
to $J$ and $\rho$ an irreducible representation of $J$ with restriction to $J^0$ inflated from a cuspidal representation
of $J^0/J^1$, and condition (vi) is satisfied. Property (vii) comes along the proof of exhaustion in (\cite{KS},
see the proof of Theorem 11.2). (See \cite{KSS} for general results about "intertwining implies conjugacy"
for semisimple characters).
Adding as before the level $0$ cuspidal $C$-types, one gets
the set of cuspidal simple $C$-types in $G$, which satisfies exhaustion (\cite{KS}, TheoremA (ii)) and unicity (\cite{KSS} Main Theorem). Using the action of $\Aut(C)$ on semisimple characters analysed above, verifying $\Aut(C)$-stability for cuspidal simple types
follows as before.

The second adjointness holds for $(G,C)$ \cite{D09}.

\subsubsection{Quaternionic form} Finally the case of  a quaternionic form $G$ of a classical group   for odd $p$ is obtained by Skodlerak  \cite{Sk17}   for $C=\mathbb C$,  \(cite{Sk20} Theorem1.1) in the modular case. That case is a mix of the previous two and Skodlerak constucts a set of $\mathbb C$-types satisfying irreducibility, exhaustion and unicity  (\cite{Sk20} Theorem 1.1). The procedure to define semisimple characters is the same as for classical groups but starting with $\Aut_DV$ where $V$ is a right vector space of dimension $m$ over a central quaternion division $F$-algebra $D$ equipped with an anti-involution $d\mapsto \overline d$ (it is necessarily of the first kind), and a non-degenerate $\epsilon$-hermitian form $h$ on $V$ with $\epsilon\in \{1,-1\}$. 
The group $\underline G$ is the group of isometries of $h$; it is connected reductive, indeed over a quadratic unramified extension
of $F$ it gives a unitary group (\cite{Sk17}, Proposition 2.2). Starting with a semisimple stratum $(\Lambda, n, \beta) \in \End_D(V)$ (again
$r=0$ is omitted), one defines  dual stratum $(\Lambda^\flat, n , -\overline \beta) $(\cite{Sk17}  Definition   4.1) as in \ref{sss:classical}, and, for a self-dual stratum,
the set $ \mathfrak C(\Lambda, \beta, \psi )$ of semisimple characters in $G$. The intertwining of semisimple characters in $G$  computed in (\cite{Sk17} 4.4)    is the same in the modular case.
For the very special semisimple characters $ \theta$ (that is those giving rise to a cuspidal $\mathbb C$-type), it has
indeed the form $JG'J$ (\cite{Sk17}, proof of Proposition 4.3), which gives property (ii), and (i) follows.
Properties (iii) and (iv) come from (\cite{Sk17} Lemma 4.2 and Proposition 4.3 ). We have already said
why (v) is true, and (vi) comes from (\cite{Sk17} Definition   6.2 ), (vii) from (\cite{Sk17} Theorem 8.1).
The action of $\Aut(C)$ follows the same pattern as \ref{sss:classical}, and adding the level $0$
$C$-types one gets the set of simple cuspidal $\mathbb C$-types in $G$, which satisfies $\Aut(\mathbb C)$-stability. 
 
 \subsection{Yu types}\label{ss:3.6}  We now turn to the representations constructed by Yu  when  $G$ is  a connected reductive group   which splits over a tamely ramified field extension
of $F$ \cite{Y01}. We refer  to the papers of Fintzen \cite{F1},\cite{F2}, \cite{F3} because she corrects an error in the proof of irreducibility in \cite{Y01}, and also because she proves  in \cite{F2} that, when $p$ does not divide the order of the absolute Weyl group of $G$, and for any algebraically closed field $C$ with $c\neq p$, the set of $C$-types constructed by Yu satisfies irreducibility and exhaustion. Hakim and Murnaghan (\cite{HM08}, Theorem 6.3), go a long way towards proving unicity when $C=\mathbb C$, but their result is not expressed in terms of a list of types; the translation and the extension to   $C$ algebraically closed of characteristic $c\neq p$    is done in the  Ph.D. thesis of R. Deseine  \cite{De21}. 

 \subsubsection{} We follow the account and notation of (\cite{F2}, 2.1,2.4, 5.1).  The input for the construction comprises a sequence $G=G_1 \supset G_2 \supsetneq   \ldots\supsetneq G_{n+1}$ of twisted Levi subgroups of $G$ splitting over a tamely ramified extension of $F$ and such that $Z(G_{n+1})/Z(G)$ is anisotropic, a sequence 
$r_1>  \ldots >r_n>0$ of real numbers ($n=0$  is allowed and gives the level $0$ cuspidal representations of $G$), and an element  $x$ in the extended building $\mathcal B(G_{n+1})\subset \mathcal B(G)$ with image $[x]\in \mathcal B(G^{ad}_{n+1})$ a vertex. 
On the representation side, the input consists of:

-   an irreducible representation $\rho$ of $(G_{n+1})_{[x]}$ trivial on  $(G_{n+1})_{x,0+} $ of restriction    to the parahoric subgroup $(G_{n+1})_{x,0} \subset G_{n+1}$  inflated  from a cuspidal representation of the finite connected reductive  group  $(G_{n+1})_{x,0}/ (G_{n+1})_{x,0+}  $.

-   if $n>0$,  a sequence of characters $\varphi_i$ of $G_{i+1}$, assumed  of depth $r_i$ with respect to $x$ (meaning  trivial on $(G_{i+1})_{x,r_i+}$ and not trivial on $(G_{i+1})_{x,r_i}$),  and $G_i$-generic with respect to $x$ (in the sense of \cite{Y01} \S9, p.59, a condition on $\varphi_i$ restricted to $(G_{i+1})_{x,r_i}$)   if $G_i\neq G_{i+1}$.

 \begin{remark}Recall that 
 the Bruhat-Tits  building $\mathcal B(G )$ of $G$ is the direct product of   the  Bruhat-Tits  building $\mathcal B(G^{ad} )$ of the adjoint group $G^{ad}$, by a real affine space.  For any point $x\in \mathcal B(G )$  we denote by 
 $[x]$ its projection in  $\mathcal B(G^{ad} )$, and by $G_x$ and $G_{[x]}$ the $G$-stabilizers of $x$ and $[x]$.  The  parahoric subgroup
 $G_{x,0} $ of $G$  fixing $x$ and  its pro-$p$ unipotent radical depend only on $[x]$ and we put  $G_{[x],0}=  G_{x,0}, G_{[x],0+}=  G_{x,0+}$. 
   \end{remark}

When $n>0$, we call  $((G)_{i})_{1\leq i \leq n+1}, (r_i)_{1\leq i \leq n}, x, (\varphi_i)_{1\leq i \leq n}, \lambda_0)$ a Yu datum.
The associated cuspidal $C$-type of $G$ 
 is   $(J, \lambda=   \lambda_0\otimes_C  \kappa  )$,   where 
\footnote { $(G_i)_{x,r_i,r_i/2} \subset G_i$ is the open compact subgroup denoted,  $(G_{i+1},G_i)(F)_{x,r_i,r_i/2}$ in \cite{Y01} p.585-586.  As that last notation underlies, the group depends on both $G_i $ and $G_{i+1}$.

  $(G_{n+1})_{[x]}\subset G_{n+1}$ is the $G_{n+1}$-stabiliser of the image $[x]$ of $x$ in the  building of $G_{n+1}^{ad}$; it normalizes $G_{x,r_i,r_i/2}$ and $G_{x,r_i,(r_i/2)+}$; it is an open subgroup containing the center $Z(G_{n+1})$ of  $G_{n+1}$
  and $(G_{n+1})_{[x]}/Z(G_{n+1})$ is compact.
 
 
 The second equality in (\ref{eq:tildeK}) follows from  $(G_i)_{x,r_i/2}=(G_i)_{x,r_i,r_i/2}(G_{i+1})_{x,r_i/2}$. 

We have also $(G_i)_{x,(r_i/2)+}=(G_i)_{x,r_i,(r_i/2)+}(G_{i+1})_{x,(r_i/2)+}$.

}:
  \begin{equation}\label{eq:tildeK}J =(G_1)_{x, r_1/2}\ldots (G_n)_{x, r_n/2}(G_{n+1})_{[x]}= (G_1)_{x,r_1, r_1/2}\ldots (G_n)_{x, r_n, r_n/2}(G_{n+1})_{[x]},
  \end{equation}
 and  $\lambda_0$  is the representation of $J$ trivial on $(G_1)_{x, r_1/2}\ldots (G_n)_{x, r_n/2}(G_{n+1})_{x,0+}$ inflating 
the representation $ \rho$ of $(G_{n+1})_{[x]}$. In  (\cite{F2}, \S 2.4), $J$ is denoted by $\tilde K$ and  $\lambda_0$ is still denoted  by $\rho$.
 To describe the representation  $  \kappa $   of $\tilde K$ we introduce more notations following (\cite{F2} 2.5). For $1\leq i \leq n$, 
there exists  a  unique $C$-character 
$$\hat \varphi_i: (G_{n+1})_{[x]}(G_{i+1})_{x,0}G_{x,(r_i/2)+} \to C^*$$ given  on $(G_{n+1})_{[x]}(G_{i+1})_{x,0}$ by the restriction of $\varphi_i$, and on $G_{x,(r_i/2)+}$ factorizing through a   natural homomorphism from $G_{x,(r_i/2)+}/G_{x,r_i+} $ to $(G_{i+1})_{x,(r_i/2)+}/(G_{i+1})_{x,r_i+}$ on which it is induced by $\varphi_i$. That homomorphism is described in (\cite{F3}, \S2.5 after second bullet), after (\cite{Y01}, \S4).  Let  $\mu$ denote the group of $p$-roots of $1$ in $C$.
 By (\cite{Y01}, Proposition 11.4), $V_i=G_{x,r_i,r_i/2}/ G_{x,r_i,(r_i/2)+}$ admits the symplectic form 
$$(x,y)\mapsto \langle x, y \rangle_{\hat \varphi_i}= \hat \varphi_i (xyx^{-1}y^{-1}) : V_i\times V_i\to \mu,$$ and   a canonical special isomorphism 
$$j_{\hat \varphi_i}:  G_{x,r_i,r_i/2}/ (G_{x,r_i,(r_i/2)+} \cap \Ker \hat{\varphi_i})\to V^\sharp_{\hat \varphi_i},$$
where $V^\sharp_{\hat \varphi_i}$ is  the finite Heisenberg $p$-group with underlying set $V_i\times \mu$ and law given by $(v,\epsilon)(v',\epsilon')=(v+v', \epsilon+\epsilon'+(1/2) \langle v , v'\rangle_{\hat \varphi_i})$ (we use an additive notation for both $V_i$ and $\mu$). The special isomorphism $j_{\hat \varphi_i}$ identifies the 
centres 
 $G_{x,r_i,(r_i/2)+}/ (G_{x,r_i,(r_i/2)+} \cap \Ker \hat{\varphi_i}) $ and $\mu$ of the two groups.
  The  conjugation action of $(G_{n+1})_{[x]}$ on $G_{x,r_i,r_i/2}/G_{x,r_i,(r_i/2)+}$ preserves $ \hat \varphi_i$ so gives a group morphism $(G_{n+1})_{[x]}\to \Sp(V_i,  \langle \ , \  \rangle_{\hat \varphi_i})$ which is  trivial on $(G_{n+1})_{x,0+}$, and  with $j_{\hat\varphi_i}$, gives a group morphism $$\tilde j_{\hat \varphi_i}:(G_{n+1})_{[x]}(G_{x,r_i,r_i/2}/ (G_{x,r_i,(r_i/2)+} \cap \Ker \hat{\varphi_i}))\to \Sp(V_i,  \langle \ , \  \rangle_{\hat \varphi_i})\ltimes V^\sharp_{\hat \varphi_i}.$$
The Heisenberg $C$-representation $(\eta_i, V_{\eta_i})$ of $V^\sharp_{\hat \varphi_i}$  with
restriction to $\mu$  a multiple of the character given by the inclusion of $\mu$  into $C^* $, extends canonically  to an irreducible representation $\omega_i$ of $ \Sp(V_i,  \langle \ , \  \rangle_{\hat \varphi_i})V^\sharp_{\hat \varphi_i}$  (the Weil representation \cite{G} Theorem 2.4), hence  a representation  $\omega_i \circ \tilde j_{\hat \varphi_i}$    of $(G_{n+1})_{[x]}(G_{x,r_i,r_i/2}/ (G_{x,r_i,(r_i/2)+} \cap \Ker \hat{\varphi_i}))$ on $V_{\eta_i}$, which inflates to an action  of  $(G_{n+1})_{[x]}G_{x,r_i,r_i/2}$ on $V_{\eta_i}$.
 
 There is a unique representation $\kappa$ of   $J$  on the tensor product $\otimes _{i=1}^n V_{\eta_i}$  such that $(G_{n+1})_{[x]}$ acts on  $V_{\eta_i}$ as above for $1\leq i \leq n$, and $(G_i)_{x,r_i,(r_i/2)}$ acts by $\eta_i$ on $V_{\eta_i}$  
 and  by multiplication by the character $\hat \varphi_j|_{(G_i)_{x,r_i,(r_i/2)}}$ on $V_{\eta_j}$ for  $1\leq i\neq j \leq n$  \footnote{The group $(G_{n+1})_{[x]}(G_{i+1})_{x,0}G_{x,(r_i/2)+}$ contains $(G_i)_{x,r_i,r_i/2}$ for $1\leq i \leq n$. }.

 \subsubsection{} 
To the above  data we attach  groups $H^1\subset J^1\subset J^0  \subset J$  and representations $\theta, \eta, 
 \kappa$ satisfy the properties (i) to (vii) of our setting of \S\ref{ss:3.4} as follows:

\bigskip  $J$ is the group \eqref{eq:tildeK}.
 
 Replacing $(G_{n+1})_{[x]}$ with $(G_{n+1})_{x,0} $ in $J$ we get \begin{equation}\label{eq:J0}J^0=(G_1)_{x, r_1/2}\ldots (G_n)_{x, r_n/2}(G_{n+1})_{x,0}= (G_1)_{x, r_1,r_1/2}\ldots (G_n)_{x, r_n,r_n/2}(G_{n+1})_{x,0}.
\end{equation}  Replacing  $(G_{n+1})_{x,0}$ with $(G_{n+1})_{x,0+} $ in $J^0$  we get   
  \begin{equation}\label{eq:J1} J^1=(G_1)_{x, r_1/2}\ldots (G_n)_{x, r_n/2}(G_{n+1})_{x,0+}=(G_1)_{x, r_1, r_1/2}\ldots (G_n)_{x, r_n,  r_n/2}(G_{n+1})_{x,0+},\end{equation}
    The quotient $J^0/J^1 \simeq (G_{n+1})_{x,0} / (G_{n+1})_{x,0+}$ is the finite connected reductive quotient of the parahoric subgroup $(G_{n+1})_{x,0} $ of $G_{n+1}$.
Replacing   $r_1/2, \ldots, r_n/2$ by $ (r_1/2)+,\ldots ,  (r_n/2)+$  for $i=1,\ldots,n$ in $J^1$, we get  
 \begin{equation}\label{eq:H1} H^1=(G_1)_{x, (r_1/2)+}\ldots (G_n)_{x, (r_n/2)+}(G_{n+1})_{x,0+}=(G_1)_{x,r_1, (r_1/2)+}\ldots (G_n)_{x, r_n, (r_n/2)+}(G_{n+1})_{x,0+}.\end{equation} 

    $\theta$  is the unique character  of $H^1$ trivial on $(G_{n+1})_{x,0+}$, and 
equal to  $\hat \varphi_i$ on $(G_i)_{x, r_i, (r_i/2)+}$ for $1\leq i \leq n$. 

 $\eta=\eta_\theta$ is the  unique   representation of $J^1$ on $\otimes_{i=1}^n V_{\eta_i}$ trivial on $(G_{n+1})_{x,0+}$, where  $(G_i)_{x,r_i,(r_i/2)}$ acts by $\eta_i$ on $V_{\eta_i}$ 
 and  by multiplication by the character $\hat \varphi_j|_{(G_i)_{x,r_i,(r_i/2)}}$ on $V_{\eta_j}$ for  $1\leq i\neq j \leq n$.   
 
A {\it preferred} extension  of $\eta$ is $\kappa$.

\bigskip Let us say why properties (i) to (vii) are true. The intertwining of $\theta$ was determined by Yu (\cite{Y01}  Theorem 9.4) giving (i) and (ii). Properties (iii) and (iv) come from (\cite{Y01}  Theorem 11.5
and  Proposition 12.3)
(note that Yu works over complex numbers, but his reasoning for properties (i) to (iv)
is valid here (see \cite{F1})). The fact that the construction above  gives a preferred
extension is essentially due to Fintzen; it is somewhat hidden in (\cite{F1} proof of 
Lemma 3.5), it would much more space to give detail, so we omit them.
Thus we have (v) and (vi). Finally (vii) comes in the proof of exhaustion by Fintzen (\cite{F1}
proof of Theorem 7.1). Once again giving detail would take us much more space.

Dat proved the second adjointness holds for $(G,C)$ \cite{D09}, \cite{D20}. Deseine \cite{De21} proves unicity for the types constructed above.
   
\subsubsection{}  We  verify now that  the list of Yu types is $\Aut(C)$-stable. Let $\sigma\in \Aut(C)$.  We   show that  if   $((G)_{i})_{1\leq i \leq n+1}, x, (\varphi_i)_{1\leq i \leq n}, \rho)$ is  a Yu datum  of associated  type $(\tilde K,\lambda)$,  then $((G)_{i})_{1\leq i \leq n+1}, x, (\sigma(\varphi_i))_{1\leq i \leq n}, \sigma(\rho))$ is a Yu datum of associated type
 $(\tilde K, \sigma(\lambda))$.
 
 As in the other cases,  $\sigma(\rho)$ is trivial on  $(G_{n+1})_{x,0+} $ and restricts on $(G_{n+1})_{x,0}$ to a cuspidal representation of $ (G_{n+1})_{x,0}/ (G_{n+1})_{x,0+}$, because $\rho$ does.
  
 We explain  now why $\sigma(\varphi_i)$ has depth $r_i$ and is $G_i$-generic (if $G_i\neq G_{i+1}$) with respect to $x$ because $\varphi_i$ does.  
 
 Underlying the notion of genericity, is the choice of an additive character $\psi:F\to C^*$  as before, giving an identification of  the group $\hat { \mathfrak g}$  of smooth characters of ${\mathfrak g}=\Lie(G)$,  with the dual  ${ \mathfrak g}^* = \Hom_F( {\mathfrak g},F)$. Explicitly,  each element $f\in \mathfrak g^*$ identifies with the smooth character $\phi_{\psi,f}(u)=\psi (f(u))$ of the additive group $ \mathfrak g $. For $r\in \mathbb R$, the orthogonal of ${\mathfrak g}_r$ is ${\mathfrak g}^*_{x,-r+}$, and that of  ${\mathfrak g}_{x,r+}$ is  ${\mathfrak g}^*_{x,-r}$. Our choice of $\psi$ yields an isomorphism $\iota: ( \mathfrak g_{x,r}/ \mathfrak g_{x,r+})^{\hat {}} \to  \mathfrak g^*_{x,-r}/ \mathfrak g^*_{x,-r+} $. 
Changing $\psi$ to $\psi^a$ for $a\in O_F^*$   multiplies $\iota$ by $a^{-1} $ as $\phi_{\psi,f}=\phi_{\psi^a, a^{-1}f}$.
For $r>0$, $G_{x,r}/G_{x,r+}$ identifies canonically with ${\mathfrak g}_{x,r}/{\mathfrak g}_{x,r+}$, because $G$ splits on a tamely ramified extension \cite{F0} Rem.3.2.4.  

Those considerations apply to $G_{i+1}$ as well, and changing $\psi$ to $\psi^a$ for $a\in O_F^*$ does not change  the  depth of the smooth characters of $G_{i+1}$ with respect to $x$  or the $G_i$-genericity of the elements of $\widehat{(G_{i+1})_{x,r_i}/(G_{i+1})_{x,r_i+} }$ (by \cite{Y01} \S9, p.59, that genericity is expressed  in terms of the element of ${\mathfrak g}_{x,-r_i}/{\mathfrak g}_{x,-r_i+}$ corresponding to it, and multiplication  by $a\in O_F^*$ does not affect it).  Consequently if $\varphi_i$ is of depth $r_i$, and $G_i$-generic (if $G_i\neq G_{i+1}$), with respect to $x$, then so is $\sigma(\varphi_i)$. 

We proved  that $((G)_{i})_{1\leq i \leq n+1}, x, (\sigma(\varphi_i))_{1\leq i \leq n}, \sigma(\rho))$ is a Yu datum. 
It remains to show that   $(\tilde K, \sigma(\lambda))$ is the associated type.
We have  $\sigma(\lambda)= \sigma(\kappa) \otimes \sigma(\tilde \rho)$, and clearly $\sigma ({\tilde \rho})=\tilde{\sigma(\rho)}$. We explain now why  $\sigma(\kappa)$ is the representation of $\tilde K$
   associated to $(\sigma(\varphi_i))_{1\leq i \leq n}$. 
  
It is clear that $\sigma (\hat{\varphi_i})=\widehat{\sigma (\varphi_i) } $.
We have an isomorphism $\tilde \sigma: V^\sharp_{\hat\varphi_i}\to V^\sharp_{\widehat{\sigma(\varphi_i)}}$ given by identity on $V_i$ and $x\mapsto \sigma(x)$  on $\mu$; it extends to an isomorphism $\tilde \sigma:  \Sp(V_i,  \langle \ , \  \rangle_{\hat \varphi_i}) V^\sharp_{\hat\varphi_i}\to  \Sp(V_i,  \langle \ , \  \rangle_{\widehat{\sigma(\varphi_i)}}) V^\sharp_{\widehat{\sigma(\varphi_i)}}$. 
One checks  from the construction (\cite{Y01} \S 11) that the special isomorphisms satisfy $j_{\widehat{\sigma (\phi_i)} } =\tilde \sigma \circ j_{\hat{\varphi_i}}$, and also  $\tilde j_{\widehat{\sigma (\phi_i)} } =\tilde \sigma \circ \tilde j_{\hat{\varphi_i}}$.

The  representation $\sigma(\eta_i, V_{\eta_i})$ of $V^\sharp_{\hat \varphi_i}$  is the Heisenberg representation of $V^\sharp_{\hat \varphi_i}$ where $\mu$ acts by multiplication by 
$\sigma(\hat \varphi_i )\circ j_{\hat \varphi_i}^{-1}$, and the associated Weil representation of  $ \Sp(V_i,  \langle \ , \  \rangle_{\hat \varphi_i})V^\sharp_{\hat \varphi_i}$   is 
 $\sigma(\omega_i)$.
Composing with $\tilde \sigma^{-1}$,  we get an action of $V^\sharp_{\widehat{ \sigma(\varphi_i)}}$ on  $\sigma( V_{\eta_i}) = C\otimes_{\sigma, C}  V_{\eta_i}$  which is the Heisenberg representation given by $ \widehat {\sigma(\varphi_i)} \circ  j_{ \widehat {\sigma(\varphi_i)}}^{-1}$ on $\mu$, and the associated Weil representation of  $ \Sp(V_i,  \langle \ , \  \rangle_{\widehat { \sigma(\varphi_i)}})V^\sharp_{\widehat{ \sigma(\varphi_i)}}$   is 
 $\sigma(\omega_i) $.
  Following the action of $\sigma $ through the rest of the construction of $\kappa$  is straightforward and we get that $\sigma(\kappa)$ is the representation of $\tilde K$ associated with $(\sigma(\varphi_i))_{1\leq i \leq n}$.


\begin{thebibliography}{22}





\bibitem{AHHV} N. Abe,  G. Henniart, F. Herzig, M.-F. Vign\'eras --  {A classification of admissible irreducible modulo $p$ representations of reductive $p$-adic groups}. J. Amer. Math. Soc. 30, p.495-559  (2017).

\bibitem{AHV19} N. Abe, G. Henniart, M.-F. Vign\'eras -- {Mod $p$ representations of reductive $p$-adic groups: functorial properties.}  Trans. Amer. Math. Soc. 371  p.8297-8337 (2019).



\bibitem{B91}D.~J.~Benson -- Representations and  Cohomology I, Basic Representation Theory of Finite Groups and Associative    Algebras, Cambridge Studies in Advanced Mathematics, (1991). 











 
 \bibitem{Bki-A5} N.~Bourbaki --  Alg\`ebre, Chap. 4-7.  Masson (1981).



\bibitem{Bki-A8} N.~Bourbaki --  \'El\'ements de math\'ematiques.   Alg\`ebre, Chap.\ 8.  Modules et anneaux semi-simples. Springer, Berlin-Heidelberg  (2012).



    
    

\bibitem{Br98} P.~Broussous --  Extension du formalisme de Bushnell et Kutzko au cas d'une alg\`ebre \`a division.   Proc. London Math. Soc. (3) 77, no. 2, 292--326 (1998).

\bibitem{BTI}F.~Bruhat, J.~Tits --  Groupes r\'eductifs sur un corps local. I. Donn\'ees radicielles valu\'ees,   {Publ.\ Math.\ IHES}  {41}, p.\  5--251  (1972).


\bibitem{B90} C.J.~Bushnell  -- Induced representations of locally profinite groups, J. of Algebra 134, p.104--114 (1990).


\bibitem{BH06} C.J.~Bushnell,  G.~Henniart -- The local Langlands conjecture for $GL(2)$, Grundlehren der mathematischen Wissenschaten 335, Springer-Verlag (2006).

\bibitem{BH13} C.J.~Bushnell,  G.~Henniart -- Interwining of simple characters in $GL(n)$, IMRN vol.2013, p.3977--3987 (2013).


\bibitem{BK93} C.J.~Bushnell,  P.C.~Kutzko -- The admissible dual of $GL(N)$ via open compact subgroups, Annals of Math. Studies, Princeton Univeristy press (1993).


\bibitem{BK93'} C.J.~Bushnell,  P.C.~Kutzko -- The admissible dual of $SL(N)$, I. Annales scientifiques de l'E.N.S. 4e s\'erie, tome 26, no 2 (1993), p. 261--280.

\bibitem{BK94} C.J.~Bushnell,  P.C.~Kutzko -- The admissible dual of $SL(N)$, II. Proc. LM.S. {68}, p.317--378 (1994).  


\bibitem{BK98} C.J. Bushnell,  P.C. Kutzko -- Smooth representations of reductive p-adic groups~: structure theory via types. Proceedings of the London Mathematical Society,  {77}, p. 582--634  (1998).

 




 




\bibitem{C19}Peiyi~Cui -- Modulo $\ell$-representations of  $p$-adic groups $SL(N,F)$,	arXiv:1912.13473 (31/12/2019).

\bibitem{C20}Peiyi~Cui -- Modulo $\ell$-representations of $p$-adic groups $SL_n(F)$: maximal simple $k$-types,
arXiv:2012.07492




 
 \bibitem{D09} J.-Fr.~Dat -- Finitude pour les repr\'esentations lisses des groupes p-adiques.  J. Inst. Math. Jussieu, 8 (1) p.261--333 (2009).
 
  \bibitem{D18} J.-Fr.~Dat -- Simple subquotients of big parabolically induced representations of $p$-adic groups. J. Algebra 510 p.499-507 (2018).
  
  \bibitem{D20}  J.-Fr.~Dat -- email (2020/10/14).

    \bibitem{De21} R.~Deseine  -- Autour des repr\'esentations complexes et modulaires des groupes r\'eductifs $p$-adiques. Th\`ese Paris-Saclay 2021. 


\bibitem{DM91} F.~Digne, J.~Michel --  Representations of finite groups of Lie type, London Mathematical Society Student Text 21 (1991).
 




\bibitem{F82} W.~Feit -- The representation theory of finite groups, North-Holland Math. Library, vol.25 (1982).
\bibitem{HV15} G.~Henniart, M.-F.~Vign\'eras -- The  Satake isomorphism modulo $p$ with weight.  J. f\"ur die reine und angewandte Mathematik, vol. 701, p. 33--75, (2015).
\bibitem{F0} J.~Fintzen -- On the Moy-Prasad filtration, arXiv:1511.00726v4.  To appear in  JEMS.
\bibitem{F1} J.~Fintzen -- Types for tame $p$-adic groups. Annals of Mathematics 193 no. 1 (2021), p. 303-346.
\bibitem{F2} J.~Fintzen -- Tame cuspidal representations in non-defining characteristics, arXiv:1905.06374.
\bibitem{F3} J.~Fintzen -- On the construction of tame cuspidal representations, arXiv:1908.09819.


\bibitem{G} P.~G\'erardin -- Weil representations associated to finite fields, J. Algebra 46 (1977), no. 1, p.54--101.

\bibitem{G82}  I.~Giorgiutti -- Groupes de Grothendieck-Introduction. Ann. Fac. Sci. Toulouse,  s\'erie 4, t.26, p.151--207 (1982)

\bibitem{HM08}J.~Hakim, F.~Murnaghan -- Distinguished tame supercuspidal representations, I.M.R.P. 2, p.5--166 (2008).





\bibitem{HV19} G.~Henniart, M.-F.~Vign\'eras -- Representations of a $p$-adic group in characteristic $p$. For Joseph Bernstein. Proceedings of Symposia in Pure Mathematics, Volume 101, p.~171-210 (2019).
 





\bibitem{Hiss96} G.~Hiss  -- Supercuspidal Representations
of Finite Reductive Groups. J. of Algebra 184, p.839-851 (1996).

\bibitem{HS} P.J.~Hilton, U.~Stammbach  -- A course in homological algebra. Graduate Texts in Mathematics 4, Springer-Verlag (1971).







\bibitem{KS06}M.~Kashiwara, P.~Shapira -- Categories and Sheaves. Grundlehren der mathematischen Wissenschaften, vol. 332, Springer Berlin-Heidelberg (2006). 


\bibitem{KHV20} K.~Koziol, Fl.~Herzig, M.-F.~Vign\'eras (appendix by Sug-Woo Shin) -- On the existence of admissible supersingular representations pf $p$-adic reductive groups,  Forum of Mathematics, Sigma  8, e2, 73pp, (2020).





\bibitem{KS}R.~Kurinczuk,  S.Stevens -- Cuspidal $\ell$-modular representations of $p$-adic classical groups. J. Reine Angew. Math. Volume 2020 (764) p.23-69 (2020).

\bibitem{KSS}R.~Kurinczuk, D.~Skodlerack, S.Stevens --  Endo-parameters for  $p$-adic classical groups. Inventiones mathematicae volume 223, p. 597-723 (2021).




\bibitem{L99} T.Y.~Lam --  Lectures on Modules and Rings, Graduate Texts in Math. 189 (1999)

\bibitem{L96} E.~Landvogt --  A compactification of the Bruhat-Tits building. Lecture Notes 1619, Springer-Verlag (1996)









\bibitem{M20} A.~Mayeux -- Comparison of Bushnell-Kutzko and Yu's constructions of supercuspidal representations, arXiv:2001.06259 (2020).
 

\bibitem{MS14} A.~Minguez, V.~S\'echerre  -- Types modulo $\ell$ pour les formes int\'erieures de $GL_n$ sur un corps local non archim\'edien, Proc. L.M.S., p.1-69 (2014).

\bibitem{M99} L.~Morris -- Level $0$ $G$-types, Comp.Math.(118), p.135-157 (1999).

\bibitem{MP96} A.~Moy, G.~Prasad -- Jacquet functors and unrefined minimal $K$-types,  Comment.Math.Helv.   {71}, p.~98--121  (1996).

\bibitem{Mu19} G.~Muic -- On representations of reductive $p$-adic groups over $\mathbb Q$-algebras, arXiv:1905.01949v2.












\bibitem{OV}R.~Ollivier, M.-F.~Vign\'eras -- Parabolic induction modulo $p$. Selecta Mathematica, Volume 24, Issue 5, p.~3973-- 4039  (2018).









\bibitem{Schm81} P.~Schmid -- On the Clifford Theory of Blocks of Characters, 
J. Algebra 73, p.~44--55 (1981).

\bibitem{S04} V.~S\'echerre -- Repr\'esentations lisses de $GL(m,D)$, I: caract\`eres simples, Bull.Soc.math.France 132, p.~327--396 (2004).

\bibitem{S05} V.~S\'echerre -- Repr\'esentations lisses de $GL(m,D)$, II: beta-extensions, Compos. Math.  141, p.~1531--1550 (2005).

 

\bibitem{SS08} V.~S\'echerre, S.~ Stevens -- Repr\'esentations lisses de $GL(m,D)$, IV: repr\'esentations supercuspidales, Journal I.M.J. {7}, p.~527--574 (2008).

\bibitem{Sk17} D.~Skodlerack -- Semisimple characters for inner forms II: quaternionic inner forms of classical groups, arXiv:1801.00265.

 
  \bibitem{Sk20} D.~Skodlerack -- Cuspidal irreducible complex or $\ell$-modular representations of quaternionic forms
of p-adic classical groups for odd p, arXiv:1907.02922 v2. 



\bibitem{S71} J-P.~Serre -- Re[r\'esentations lin\'eaires des groupes finis. Hermann (1971).



\bibitem{St05} S.~ Stevens -- Semisimple characters for p-adic classical groups, Duke Math. J. 127(1),
p.~123--173 (2005).
 
\bibitem{St08} S.~ Stevens -- The 
 supercuspidal representations of $p$-adic classical groups, Invent. Math. {172}, p.~289--352   (2008).





\bibitem{V96} M.-F. Vign\'eras -- Repr\'esentations $\ell$-modulaires d'un groupe r\'eductif fini $p$-adique avec $\ell \neq p$. Birkhauser Progress in math. 137 (1996).


\bibitem{V00} M.-F. Vign{\'e}ras -- Irreducible modular representations of a reductive p-adic group and simple modules for Hecke algebras, {European Congress in Mathematics Barcelona}.   Progress in Math. 201 p.~117-133 (2000).






\bibitem{V13} M.-F. Vign{\'e}ras -- The right adjoint of the parabolic induction.    Arbeitstagung Bonn 2013.  Progress in Mathematics 319 p.~405-425 (2016).









\bibitem{W19} M.H. Weissman  -- An induction theorem for groups acting on trees, {Representation theory} 23 p.~205--212 (2019).

\bibitem{Y01} J.K.Yu -- Construction of tame supercuspidal representations, J.A.M.S. {14}, p.~579--622 (2001).

\bibitem{Z92} E.-W. Zink -- Representation theory of local division algebras. J. Reine Angew. Math. 428, p.~1--44 (1992).


\end{thebibliography}
  \end{document}